\documentclass[12pt,leqno]{amsart}
\newcommand\ignore[1]{}
\input amssym.def
\usepackage[raiselinks=false,colorlinks=true,citecolor=blue,urlcolor=blue,linkcolor=blue,bookmarksopen=true,pdftex]{hyperref}
\usepackage{booktabs,multirow,lscape,longtable} % for tables 
\usepackage{mathtools}
\usepackage{tikz} 
\numberwithin{equation}{section}
\numberwithin{equation}{subsection} 
\newtheorem{theorem}{Theorem}[section]
\newtheorem{proposition}[theorem]{Proposition}
\newtheorem{lemma}[theorem]{Lemma}
 
\newtheorem{remark}[theorem]{Remark}

\begin{document}
\baselineskip=15.5pt

\title
[Special cycles]{Special cycles in compact locally Hermitian symmetric spaces of type III associated with the Lie group $SO_0(2,m)$}  
\author{Ankita Pal, Pampa Paul}
\address{Department of Mathematics, Presidency University, 86/1 College Street, Kolkata 700073, India}
\email{ankita.maths1995@gmail.com \\ pampa.maths@presiuniv.ac.in}

\thispagestyle{empty}
\date{}

\begin{abstract}
Let $G = SO_0(2,m),$ the connected component of the Lie group $SO(2,m);\ K = SO(2) \times SO(m),$ a maximal compact subgroup of $G;$ and $\theta$ be the associated 
Cartan involution of $G.$ Let $X = G/K,\ \frak{g}_0$ be the Lie algebra of $G$ and $\frak{g} = \frak{g}_0^\mathbb{C}.$ In this article, we have considered the special cycles 
associated with all possible involutions of $G$ commuting with $\theta.$ We have determined the special cycles which give non-zero cohomology classes in 
$H^*(\Gamma \backslash X; \mathbb{C})$ for some $\theta$-stable torsion-free arithmetic uniform lattice $\Gamma$ in $G,$ by a result of Millson and Raghunathan.  
For each cohomologically induced representation $A_\frak{q}$ with trivial infinitesimal character, we have determined the special cycles for which the non-zero cohomology 
class has no $A_\frak{q}$-component, via Matsushima's isomorphism. 
\end{abstract}
\maketitle

\noindent 
\section{Introduction} 

Let  $I_{p,q} = \textrm{diag}(-I_p, I_q) = \left(
\begin{array}{ccc}
-I_p & 0 \\
0 & I_q \\
\end{array}
\right),$ where for any positive integer $n,\ I_n$ denotes the identity matrix of order $n.$
Let $G = SO_0(2,m),$ the connected component of the Lie group $SO(2,m) = \{ g \in SL(m+2, \mathbb{R}) : g^t I_{2,m} g = I_{2,m} \}.$ 
The Lie algebra of $G$ is $\frak{g}_0 = \frak{so}(2,m),$ which is given by 
$\{ \left(
\begin{array}{ccc}
X_1 & X_2 \\
X_2^t & X_3 \\
\end{array}
\right): X_1 \in \frak{so}(2), X_3 \in \frak{so}(m), X_2 \in M_{2 \times m}(\mathbb{R})\},$ where $\frak{so}(n)$ is the Lie algebra of all real skew symmetric matrices of order $n$ 
for all $n \in \mathbb{N}.$ For any $(m+2)\times(m+2)$ 
matrix $g,$ let $\theta(g)$ be the conjugation of $g$ by the matrix $I_{2,m}.$ Then $\theta(G) = G$ and $\theta$ is a Cartan involution of $G$ corresponding to the maximal compact subgroup 
$K= \{ \left(
\begin{array}{ccc}
A & 0 \\
0 & B \\
\end{array}
\right): A \in SO(2), B \in SO(m)\}.$ Also $\theta (\frak{g}_0) = \frak{g}_0$ and $\theta : \frak{g}_0 \longrightarrow \frak{g}_0$ be the differential of $\theta$ at the identity. 
Let $\frak{k}_0 = \{ \left(
\begin{array}{ccc}
X_1 & 0 \\
0 & X_3 \\
\end{array}
\right): X_1 \in \frak{so}(2), X_3 \in \frak{so}(m)\},$ and $\frak{p}_0 =  \{ \left(
\begin{array}{ccc}
0 & X_2 \\
X_2^t & 0 \\
\end{array}
\right): X_2 \in M_{2 \times m}(\mathbb{R})\}.$ Then $\frak{k}_0$ and $\frak{p}_0$ are  $+1$ and $-1$-eigenspaces of $\theta : \frak{g}_0 \longrightarrow \frak{g}_0$ and 
$\frak{k}_0$ is the Lie algebra of $K_0.$ Also $X = G/K$ is a Hermitian symmetric space of non-compact type. The complexifications of $\frak{g}_0$ and $\frak{k}_0$ are 
$\frak{g} = \{\left(
\begin{array}{ccc}
Z_1 & Z_2 \\
Z_2^t & Z_3 \\
\end{array}
\right): Z_1 \in \frak{so}(2,\mathbb{C}), Z_3 \in \frak{so}(m,\mathbb{C}), Z_2 \in M_{2 \times m}(\mathbb{C})\}$
and $\frak{k} = \frak{so}(2, \mathbb{C}) \oplus  \frak{so}(m, \mathbb{C})$ respectively. Also the complexification of $\frak{p}_0$ is 
$\frak{p} =\{ \left(
\begin{array}{ccc}
0 & Z_2 \\
Z_2^t & 0 \\
\end{array}
\right): Z_2  \in M_{2 \times m}(\mathbb{C})\}.$ Let $f : \frak{so}(m+2, \mathbb{C}) \longrightarrow \frak{g}$ be the map 
$f \left(
\begin{array}{ccc}
Z_1 & Z_2 \\
-Z_2^t & Z_3 \\
\end{array}
\right) = 
\left(
\begin{array}{ccc}
Z_1 & iZ_2 \\
iZ_2^t & Z_3 \\
\end{array}
\right),$ 
where $Z_1 \in \frak{so}(2,\mathbb{C}), Z_3 \in \frak{so}(m,\mathbb{C}), Z_2 \in M_{2 \times m}(\mathbb{C}).$ Then 
$f$ is an isomorphism. 
Let $\frak{t}_0$ be a maximal abelian subalgebra of  $\frak{k}_0.$ Then $\frak{t}_0 =  \{ \left(
\begin{array}{ccc}
H_1 & 0 \\
0 & H_2 \\
\end{array}
\right): H_1 \in \frak{so}(2), H_2 \in \frak{t}'_0 \}$ for some maximal abelian subalgebra of $\frak{so}(m),$ and the complexification $\frak{h}$ of $\frak{t}_0$ is a 
Cartan subalgebra of $\frak{k}$ as well as of $\frak{g}.$ Let $\Delta = \Delta(\frak{g}, \frak{h})$ be the set of all non-zero roots of $\frak{g}$ with respect to the Cartan subalgebra 
$\frak{h}.$ Since $\frak{h} \subset \frak{k},$ and each root subspace $\frak{g}^\alpha (\alpha \in \Delta)$ is one-dimensional, 
either $\frak{g}^\alpha \subset \frak{k}$ or $\frak{g}^\alpha \subset \frak{p}.$ We say that a root $\alpha \in \Delta$ is compact if $\frak{g}^\alpha \subset \frak{k},$ and 
non-compact if $\frak{g}^\alpha \subset \frak{p}.$ Let $\Delta_\frak{k}$ be the set of all compact roots in $\Delta,$ and $\Delta_n$ be the set of all non-compact roots in 
$\Delta.$ Then $\frak{k} = \frak{h} + \sum_{\alpha \in \Delta_\frak{k}} \frak{g}^\alpha, \frak{p} = \sum_{\alpha \in \Delta_n} \frak{g}^\alpha.$ 
Note that $\frak{k}$ is a reductive Lie algebra with one-dimensional center $\frak{so}(2, \mathbb{C}).$ 
For $m \neq 2,$ let $\Delta^+$ be a Borel-de Siebenthal positive root system of $\Delta,$ that is $\Delta^+$ is a 
positive root system such that the set $\Phi = \{\phi_1, \phi_2, \ldots , \phi_l \}$ of all simple roots contains exactly one non-compact root, say $\phi_1,$ and the coefficient 
$n_{\phi_1}(\delta)$ of $\phi_1$ in the highest root $\delta,$ when expressed as the sum of simple roots, is $1.$ Note that the highest root $\delta$ is non-compact. 
The Dynkin diagram of the semisimple part $[\frak{k}, \frak{k}]$ 
of $\frak{k}$ is the subdiagram of the Dynkin diagram of $\frak{g}$ consisting of the vertices $\{\phi_2, \ldots, \phi_l \}.$ If $m = 2,$ then $\frak{g}$ is not simple, and 
$\Delta = \Delta_n.$ In this case, let $\Delta^+$ be a positive root system of $\Delta,$ and the set of simple roots in $\Delta^+$ is $\Phi = \{\phi_1, \phi_2 \}.$

\begin{figure}[!h]
\begin{center} 
\begin{tikzpicture}

\filldraw [black] (0,0) circle [radius = 0.1]; 
\draw (1,0) circle [radius = 0.1]; 
\draw (2,0) circle [radius = 0.1]; 
\draw (3.5,0) circle [radius = 0.1]; 
\draw (4.5,0) circle [radius = 0.1]; 
\node [above] at (0.05,0.05) {$\phi_1$}; 
\node [above] at (1.05,0.05) {$\phi_2$}; 
\node [above] at (2.05,0.05) {$\phi_3$}; 
\node [above] at (3.75,0.05) {$\phi_{l-1}$}; 
\node [above] at (4.55,0.05) {$\phi_l$}; 
\node [left] at (-0.5,0) {$\frak{b}_l :$}; 
\node[ left] at (0.4,-0.5) {(\textrm{ if } $m=2l-1$)}; 
\draw (0.1,0) -- (0.9,0); 
\draw (1.1,0) -- (1.9,0);
\draw (2.1,0) -- (2.5,0); 
\draw [dotted] (2.5,0) -- (3,0); 
\draw (3,0) -- (3.4,0); 
\draw (4.4,0) -- (4.3,0.1); 
\draw (4.4,0) -- (4.3,-0.1); 
\draw (3.6,0.025) -- (4.35,0.025); 
\draw (3.6,-0.025) -- (4.35,-0.025); 

 \filldraw [black] (8,0) circle [radius = 0.1]; 
\draw (9,0) circle [radius = 0.1]; 
\draw (10,0) circle [radius = 0.1]; 
\draw (11.5,0) circle [radius = 0.1]; 
\draw (12.5,0) circle [radius = 0.1]; 
\draw (13.4,0.5) circle [radius = 0.1]; 
\draw (13.4,-0.5) circle [radius = 0.1]; 
\node [above] at (8.05,0.05) {$\phi_1$}; 
\node [above] at (9.05,0.05) {$\phi_2$}; 
\node [above] at (10.05,0.05) {$\phi_3$}; 
\node [above] at (11.35,0.05) {$\phi_{l-3}$}; 
\node [above] at (12.35,0.05) {$\phi_{l-2}$}; 
\node [above] at (13.45,0.55) {$\phi_{l-1}$}; 
\node [below] at (13.45,-0.55) {$\phi_l$}; 
\node [left] at (7.5,0) {$\frak{\delta}_l :$}; 
\node [left] at (8.8,-0.5) {(\textrm{ if } $m = 2l-2 \neq 2$)}; 
\draw (8.1,0) -- (8.9,0); 
\draw (9.1,0) -- (9.9,0); 
\draw (10.1,0) -- (10.5,0); 
\draw [dotted] (10.5,0) -- (11,0); 
\draw (11,0) -- (11.4,0); 
\draw (11.6,0) -- (12.4,0); 
\draw (12.6,0) -- (13.35,0.45); 
\draw (12.6,0) -- (13.35,-0.45); 

\filldraw [black] (0,-2) circle [radius = 0.1]; 
\filldraw [black] (1,-2) circle [radius = 0.1]; 
\node [above] at (0.05,-1.95) {$\phi_1$}; 
\node [above] at (1.05,-1.95) {$\phi_2$}; 
\node [left] at (-0.5,-2) {$\frak{\delta}_2 :$}; 
\node[ left] at (0.2,-2.5) {(\textrm{ if } $m=2$)}; 

\end{tikzpicture} 
\caption{Dynkin diagram of $\frak{g} \cong \frak{so}(m+2, \mathbb{C})$}\label{dynkin}
\end{center}
\end{figure}
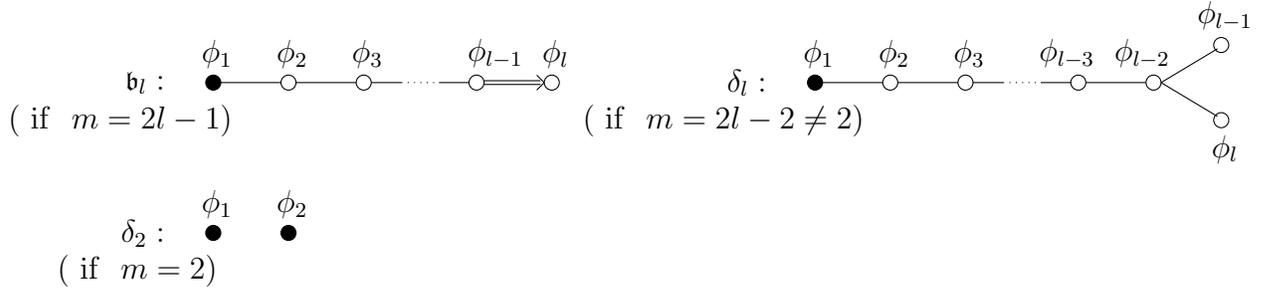 

In the diagram \ref{dynkin}, the non-compact roots are represented by black vertices. Let $\Delta_\frak{k}^+ = \Delta^+ \cap \Delta_\frak{k}, 
\Delta_n^+ = \Delta^+ \cap \Delta_n,$ and $\Delta_n^- = -\Delta_n^+$. Note that 
\[ \Delta_n^+ = 
\begin{cases}
\{\phi_1 + \cdots +\phi_i : 1 \le i \le l\} \cup \{\phi_1 + \cdots +\phi_{j-1}\\+2\phi_j+\cdots +2\phi_l : 2 \le j \le l\}  & \textrm{if } \frak{g} = \frak{b}_l (l \ge 1), \\
\{\phi_1 + \cdots +\phi_i : 1 \le i \le l-2\} \cup \{\xi_1, \xi_2, \phi_1+\cdots + \phi_l\}\\ \cup 
\{\phi_1 + \cdots +\phi_{j-1} +2\phi_j+\cdots +2\phi_{l-2}+\phi_{l-1}+\phi_l : 2 \le j \le l-2\}  & \textrm{if } \frak{g} = \frak{\delta}_l (l \ge 3), \\ 
\{\phi_1, \phi_2\} & \textrm{if } \frak{g} = \frak{\delta}_2; \\
\end{cases}
\] where $\xi_1 =\phi_1 + \cdots +\phi_{l-1}, \xi_2 = \phi_1 + \cdots +\phi_{l-2} + \phi_l.$ Define $\frak{p}_+=\sum_{\alpha \in \Delta_n^+}\frak{g}^\alpha$ and  
$\frak{p}_-=\sum_{\alpha \in \Delta_n^-} \frak{g}^\alpha.$ Then $\frak{p}=\frak{p}_+\oplus\frak{p}_-.$ In this article, we will continue with these notations, unless 
otherwise stated. 

Let $\Gamma$ be a torsion-free uniform lattice in $G.$ Then $\Gamma$ acts freely on the Hermitian symmetric space $X$ and $\Gamma \backslash X$ 
is a compact locally Hermitian symmetric space. A geometric cycle in $\Gamma \backslash X$ is a certain totally geodesic submanifold of $\Gamma \backslash X.$ 
Let $\sigma$ be an involutive automorphism of $G$ commuting with $\theta$ and $\sigma(\Gamma) = \Gamma.$ Let $G(\sigma)$ be the set of all fixed points of 
$\sigma$ in $G,\ K(\sigma) = K \cap G(\sigma),\ \Gamma(\sigma) = \Gamma \cap G(\sigma)$ and $X(\sigma) = G(\sigma) /K(\sigma).$ Then 
$\Gamma(\sigma)\backslash X(\sigma)$ is a totally geodesic submanifold of $\Gamma \backslash X,$ it is a geometric cycle in $\Gamma \backslash X$ and is called a 
special cycle. It is denoted by $C(\sigma, \Gamma).$ Let $X_u$ denote the compact dual of $X.$ 
The cohomology $H^* (X_u ; \mathbb{C})$ of $X_u$ can be identified with the cohomology 
$H^* (\Omega (X ; \mathbb{C})^G)$ of the complex $\Omega (X ; \mathbb{C})^G$ of $G$-invariant complex valued differential forms 
on $X.$ Again the cohomology $H^* (\Gamma \backslash X ; \mathbb{C})$ of $\Gamma \backslash X$ can be identified with the cohomology 
$H^* (\Omega (X ; \mathbb{C})^\Gamma)$ of the complex $\Omega (X ; \mathbb{C})^\Gamma.$
Since $\Gamma$ is a uniform lattice in $G,$ the inclusion $j_\Gamma : \Omega (X ; \mathbb{C})^G \hookrightarrow 
\Omega (X ; \mathbb{C})^\Gamma $ induces an injective map $j_\Gamma ^ * : H^*(\Omega (X ; \mathbb{C})^G) \hookrightarrow 
H^*(\Omega (X ; \mathbb{C})^\Gamma)$ (the so called Matsushima map). In this way, the elements of $H^* (X_u ; \mathbb{C})$ are represented by the $G$-invariant differential forms on $X.$ 

We have the irreducible decomposition 
\[L^2 (\Gamma \backslash G) \cong \widehat{\bigoplus}_{\pi \in \hat{G}} m(\pi , \Gamma ) H_\pi, \] 
due to Gelfand and Pyatetskii-Shapiro \cite{ggp}, \cite{gp}; where $L^2 (\Gamma \backslash G)$ is a Hilbert space of square integrable functions on 
$\Gamma \backslash G$ with respect to a $G$-invariant measure and it is a unitary representation of $G$ relative to the right translation action of $G$ on  
$L^2 (\Gamma \backslash G),\ H_\pi$ is the representation space of $\pi \in \hat{G},\ m(\pi , \Gamma ) \in \mathbb{N} \cup \{0\}$, the multiplicity of $\pi$ in $L^2 (\Gamma \backslash G)$. 
If $(\tau, \mathbb{C})$ is the trivial representation of $G$, then $m(\tau , \Gamma) = 1.$
A unitary representation $\pi \in \hat{G}$ is called an automorphic representation of $G$ with respect to $\Gamma$ if  $m(\pi , \Gamma ) > 0.$
The above decomposition leads to the Matsushima's isomorphism \cite{matsushima} 
\[
H^r (\Gamma \backslash X; \mathbb{C}) \cong \bigoplus_{\pi \in \hat{G}} m(\pi , \Gamma) H^r (\frak{g}, K; H_{\pi, K}),
\]
where $H_{\pi, K}$ is the $(\frak{g},K)$-module associated with $H_\pi.$ Now $H^* (\frak{g}, K; H_{\pi, K}) \neq 0$ {\it iff} $H_\pi$ is Vogan-Zuckerman's cohomologically 
induced module $A_\frak{q}$ with trivial infinitesimal character \cite{voganz}. Also the $(\frak{g},K)$-cohomology $H^* (\frak{g}, K; \mathbb{C})$ of the trivial representation 
$(\tau, \mathbb{C})$ is non-zero and this corresponds to $H^* (\Omega (X ; \mathbb{C})^G)$ via the Matsushima's isomorphism. Thus the elements of 
$H^* (\frak{g}, K; \mathbb{C})$ are represented by the $G$-invariant differential forms on $X.$ 

By a result of Millson and Raghunathan \cite{mira}, under certain conditions the 
Poincar\'{e} dual of the fundamental class $[C(\sigma, \Gamma)]$ of a special cycle $C(\sigma, \Gamma)$ is non-zero and is not represented by a $G$-invariant differential form 
on $X.$ This non-zero cohomology class may help to determine a non-trivial automorphic representation of $G$ via the  Matsushima's isomorphism. This technique is used 
by various authors for different semisimple Lie groups. See \cite{waldner}, \cite{sw}, \cite{schimpf}, \cite{mondal-sankaran1}, \cite{mondal-sankaran2}, \cite{paul}. In this 
article, we have considered all involutions of $G$ commuting with $\theta$ using \cite[Th. 1.3]{chuah} in the Table \ref{inv1}, Table \ref{inv2}. We have checked the conditions to get non-zero 
cohomology class in the Proposition \ref{prep} and the Proposition \ref{orp}. The main result of this article is as follows. 

\begin{theorem}\label{main} 
For each $\sigma \neq \tau_p (1 \le p \le l)$ in Table \ref{inv1} for $m = 2l-1,$ or for each $\sigma$ in Table \ref{inv2} for $m = 2l-2,$ 
there exists a torsion-free, $\langle \sigma , \theta \rangle$-stable, arithmetic, uniform lattice $\Gamma_\sigma$ of $G$ such that the cohomology 
classes defined by $[C(\sigma , \Gamma_\sigma)], [C(\sigma \theta, \Gamma_\sigma)]$ via Poincar\'e duality are non-zero and are not represented 
by $G$-invariant differential forms on $X.$
\end{theorem} 

Millson and Raghunathan \cite{mira}, Mondal and Sankaran \cite{mondal-sankaran2} have already considered the involutions $\sigma_p$ (in the Tables \ref{inv1}, \ref{inv2}), $\tau'_p$ (in the Table \ref{inv2}) 
to get non-zero cohomology classes associated with the corresponding special cycles. Jian-Shu Li \cite{li} has determined some non-zero cohomology classes associated with geometric cycles to 
identify certain $A_\frak{q}$ as automorphic representations of $O(p,q).$ The unitary equivalence classes of the representations $A_\frak{q}$ of $G$ and the Poincar\'{e}-Hodge polynomial of 
$H^*(\frak{g}, K; A_{\frak{q},K})$ are determined in \cite{pp}, using the general result in \cite{riba}, \cite{voganz}. Using these informations about $A_\frak{q},$ we have determined for each $A_\frak{q},$ the 
involutions $\sigma$ such that the corresponding cohomolgy class has no $A_\frak{q}$-component via Matsushima's isomorphism. The results are listed in the Table \ref{b-table} and Table \ref{d-table}.

\noindent
\section{Arithmetic uniform lattices of the Lie group $SO_0(2,m)$}\label{lattice} 

Let $B$ denote the Killing form of $\frak{g}$. 
For each linear function $\lambda$ on $\frak{h},$ there exists unique $H_\lambda \in \frak{h}$ such that 
\[ \lambda (H) = B (H, H_\lambda ) \textrm{ for all } H \in \frak{h}. \]
 Let $H_\alpha ^* = 2 H_\alpha /\alpha (H_\alpha)$ for all $\alpha \in \Delta.$ For each $\alpha \in \Delta$ there exists 
$E_\alpha \in \frak{g}^\alpha$ such that 
\begin{equation}\label{chevalley}
\begin{cases}
[H, E_\alpha] = \alpha (H) E_\alpha   \textrm{ for all } H \in \frak{h}, \\
[E_\alpha , E_{-\alpha} ] = H_\alpha ^* \textrm{ for all } \alpha \in \Delta , \\
[E_\alpha , E_\beta ] = 0 \textrm{ if } \alpha , \beta \in \Delta, \alpha + \beta \not\in \Delta , \alpha + \beta \neq 0 , \\
[E_\alpha , E_\beta ] = N_{\alpha , \beta} E_{\alpha + \beta} \textrm{ if } \alpha , \beta ,\alpha + \beta \in \Delta , \textrm{ where } \\
N_{\alpha , \beta} = -N_{-\alpha , -\beta} \in \mathbb{Z}, \textrm{ and } \\
\frak{k}_0 = \sum_{\phi \in \Phi} \mathbb{R} (i  H_\phi ^*) \oplus \sum_{\alpha \in \Delta_\frak{k}^+} (\mathbb{R} (E_\alpha - E_{-\alpha}) \oplus 
\mathbb{R} i(E_\alpha + E_{-\alpha})), \\ 
\frak{p}_0 = \sum_{\alpha \in \Delta_n^+} (\mathbb{R} i(E_\alpha - E_{-\alpha}) \oplus \mathbb{R} (E_\alpha + E_{-\alpha})). 
\end{cases}
\end{equation}
Note that $\{ H_\phi ^* : \phi \in \Phi \} \cup \{E_\alpha : \alpha \in \Delta\}$ is a Chevalley basis of $\frak{g}.$ 
We have $\theta(E_\alpha) = p_\alpha E_\alpha,$ where $p_\alpha = 1$ if $\alpha \in \Delta_\frak{k}$ and $p_\alpha = -1$ if $\alpha \in \Delta_n.$
Let $X_\alpha = E_\alpha - E_{-\alpha} , Y_\alpha = i(E_\alpha + E_{-\alpha})$ for all $\alpha \in \Delta^+,$ and $\frak{u} = \frak{k}_0 \oplus i\frak{p}_0.$ 
Then $\frak{u}$ is a compact real form of $\frak{g}$ dual to $\frak{g}_0.$ 

Since $G /Z \cong \textrm{Ad}(G)$ (where $Z$ denotes the centre of $G$), $Z$ is finite, $\textrm{Ad}(G)$ is the identity component of the Lie group $\textrm{Aut}(\frak{g}_0)$ and 
$\textrm{Aut}(\frak{g}_0)$ has finitely many components, it is sufficient to determine arithmetic uniform lattices of $\textrm{Aut}(\frak{g}_0)$ and 
we will follow the construction of Borel \cite{Borel} for this purpose. Let $F$ be a totally real algebraic number field of degree $>1,$ 
$\mathcal{O}$ be the ring of integers of $F$ and $S$ be the set of all infinite places of $F.$ We may choose $u \in F$ such that $u >0, s(u) < 0$ for all $s \in S \setminus \{id\}$
and $F = \mathbb{Q}(u).$ Let $v = \sqrt{u}$ and $v_s = \sqrt{-s(u)}$ for all $s \in S \setminus \{id\}.$ 
Note that $B_\frak{k} = \{ iH_\phi ^* : \phi \in \Phi \} \cup 
\{X_\alpha, Y_\alpha : \alpha \in \Delta_\frak{k}^+ \}$ is a basis of $\frak{k}_0$ and 
$B_n =  \{iX_\alpha, iY_\alpha : \alpha \in \Delta_n^+ \}$ is a basis of $\frak{p}_0.$ 
Then \begin{equation}\label{gbasis}
B = B_\frak{k} \cup B_n
\end{equation} 
is a basis of $\frak{g}_0$ consisting of vectors belonging to either $\frak{k}_0$ or to $\frak{p}_0$, 
with respect to which the structural constants are all rational numbers. 

Let $B_F = B_\frak{k} \cup vB_n,$ $\frak{m}$ be the vector space over $F$ spanned by $B_F$ and $\frak{m}^s$ be the vector space over $F^s = s(F)$ spanned by 
$B_\frak{k} \cup iv_sB_n,$ for all $s \in S-\{id\}$. Then $\frak{m}$ is a Lie algebra over $F,$ $\frak{m}^s$ is a Lie algebra over $F^s, \frak{m}^s$ is 
the conjugate of $\frak{m}$ by $s$ and $\frak{m} \otimes \mathbb{R} = \frak{g}_0 ,\  \frak{m}^s \otimes \mathbb{R} = \frak{u}$ 
for all $s \in S-\{id\}.$ We now identify $\textrm{Aut}(\frak{g})$ with an algebraic subgroup $G'$ of $ GL (N, \mathbb{C})$ 
($N = \frac{(m+1)(m+2)}{2} = \textrm{dim} (\frak{g}_0)$) defined over $F,$ via 
the basis $B_F.$ Then $\textrm{Aut}(\frak{g}_0)$ is identified with $G'_\mathbb{R},$ the group of real matrices in $G'.$ The group 
$(G'_\mathbb{R})^s$ is then $\textrm{Aut}(\frak{u}),$ hence compact,  for all $s \in S-\{id\}.$ Let $\Gamma = G'_{\mathcal{O}} = G' \cap GL(N, \mathcal{O}).$ As 
$(G'_\mathbb{R})^s$ is compact for all $s \in S-\{id\}, \Gamma$ is a cocompact arithmetic lattice in $\textrm{Aut}(\frak{g}_0),$ by  Weil's restriction of scalars. 

Since $\theta \in \Gamma, \theta \Gamma {\theta} = \Gamma.$ 
Also $\frak{g}_0 \cong \textrm{ad}(\frak{g}_0)$ and the involution of 
$\textrm{ad}(\frak{g}_0)$ corresponding to the Cartan involution $\theta$ of 
$\frak{g}_0$ is given by $\textrm{ad}(X) \mapsto \textrm{ad}(\theta X) = \theta \textrm{ad}(X)
{\theta},$ which is denoted by the same notation $\theta.$ Then $\theta$ is the differential at 
identity of the Lie group automorphism of $\textrm{Aut}(\frak{g}_0)$ 
given by $\sigma \mapsto \theta \sigma {\theta},$ which is also denoted by $\theta.$ We have $\textrm{Ad} 
\circ \theta = \theta \circ \textrm{Ad}.$ So if $\Gamma$ is a cocompact arithmetic lattice of 
$\textrm{Aut}(\frak{g}_0)$ constructed as before, then $\theta (\textrm{Ad}^{-1}(\Gamma)) 
= \textrm{Ad}^{-1}(\Gamma).$ In general, if $\sigma$ is an automorphism of $G$ whose differential at identity (denoted by the same notation) 
$\sigma \in \Gamma,$ then $\sigma$ is defined over $F$ and $\textrm{Ad}^{-1}(\Gamma)$ is invariant under $\sigma.$ 

Now we will give examples of maximal abelian subalgebras 
$\frak{t}_0$ (for any $m$), $\frak{t}'_0$ (if $m = 2l-2$) of $\frak{k}_0$ so that 
we will get Cartan subalgebras $\frak{h} = \frak{t}_0^\mathbb{C},\ \frak{h}' = {\frak{t}'}_0^\mathbb{C}$ of $\frak{g},$  
and the bases $B, B'$ of $\frak{g}_0$ as in \ref{gbasis}. Next we identify 
$\textrm{Aut}(\frak{g})$ with algebraic subgroups $G_1$ and $G_2$ of $ GL (N, \mathbb{C})$ via $B_F$ and $B'_F.$ Thus we get two $\theta$-invariant cocompact 
arithmetic lattices $\textrm{Ad}^{-1}(\Gamma),\ \textrm{Ad}^{-1}(\Gamma')$ of $G.$ Before defining $\frak{t}_0$ and $\frak{t}'_0,$ we will fix some notations. 
Let $E_{ij}$ be the elementary matrix whose $(i,j)$-th entry is $1$ and all other entries are $0,$ $F_{jk}=E_{jk}-E_{kj}$ for all $1\leq j < k \leq m+2,$ 
and $H_j=E_{2j-1,2j}-E_{2j,2j-1} = F_{2j-1,2j}$ for all $1\leq j \leq l,$ where $m = 2l-1$ or $2l-2.$ 

 Let $m = 2l-1$ or $2l-2,$ define $\frak{t}_0 =\sum_{1 \le j \le l} \mathbb{R}H_j$ (see \cite[Ch. III, \S 8]{helgason}). Then $\frak{t}_0$ is a maximal abelian subalgebra of $\frak{k}_0.$
Let $\frak{h} = \frak{t}_0^\mathbb{C}$ and for $1 \le j \le l,$  $e_j\in \frak{h}^*$ be defined by $e_j(H_k)=-i\delta_{jk},\ \delta_{jk}=1$ if $j=k$, otherwise $\delta_{jk}=0$.
Then \[\Delta(\frak{g}, \frak{h}) = 
\begin{cases}
\{\pm(e_j-e_k), \pm(e_j+e_k) : 1 \le j < k \le l \} \cup \{\pm e_j : 1 \le j \le l \} \textrm{ if } m = 2l-1,\\ 
\{\pm(e_j-e_k), \pm(e_j+e_k) : 1 \le j < k \le l \} \textrm{ if } m = 2l-2; 
\end{cases} \]
and 
\[\Delta^+ = 
\begin{cases}
\{e_j-e_k, e_j+e_k : 1 \le j < k \le l \} \cup \{e_j : 1 \le j \le l \} \textrm{ if }  m = 2l-1,\\
\{e_j-e_k, e_j+e_k : 1 \le j < k \le l \} \textrm{ if } m = 2l-2,
\end{cases} \] is a positive root system of $\Delta(\frak{g}, \frak{h})$ with the set of simple roots $\Phi = \{\phi_1, \phi_2, \ldots ,\phi_l \},$ where 
$\phi_j=e_j-e_{j+1}$ for all $1\leq j \leq l-1$ and $\phi_l=e_l$ if $m = 2l-1,\ \phi_l=e_{l-1} + e_l$ if $m = 2l-2.$ 
Define 
\[ G_{1k}^+ = f(F_{1,2k-1}+F_{2,2k}+i(F_{1,2k}-F_{2,2k-1})) \textrm{ for } 2 \le k \le l, \]
\[ G_{jk}^+ = F_{2j-1,2k-1}+F_{2j,2k}+i(F_{2j-1,2k}-F_{2j,2k-1}) \textrm{ for } 2 \le j < k \le l, \] 
\[ G_{1k}^- = f(-F_{1,2k-1}+F_{2,2k}+i(F_{1,2k}+F_{2,2k-1})) \textrm{ for } 2 \le k \le l, \]
\[G_{jk}^- =-F_{2j-1,2k-1}+F_{2j,2k}+i(F_{2j-1,2k}+F_{2j,2k-1}) \textrm{ for } 2 \le j < k \le l.\]
If $m = 2l-1,$ define 
\[D_1^+ = f(-F_{1,2l+1} + iF_{2,2l+1}) \textrm{ and } \]
\[D_j^+ = - F_{2j-1,2l+1} + iF_{2j,2l+1} \textrm{ for all } 2 \le j \le l. \]
We have $\frak{g}^{(e_j-e_k)} = \mathbb{C}G_{jk}^+,\ \frak{g}^{-(e_j-e_k)} = \mathbb{C} \overline{G_{jk}^+},\ \frak{g}^{(e_j+e_k)} = \mathbb{C}G_{jk}^-, 
\ \frak{g}^{-(e_j+e_k)} = \mathbb{C}\overline{G_{jk}^-}$ for $1 \le j < k \le l.$ If $m = 2l-1,\ \frak{g}^{e_j} = \mathbb{C}D_j^+,\ \frak{g}^{-e_j} = \mathbb{C}\overline{D_j^+}$ 
for all $1 \le j \le l.$ We have $\frak{g}^{\phi_j} \subset \frak{k}$ for $2 \le j \le l$ and 
$\frak{g}^{\phi_1} \subset \frak{p}.$ Also the highest root $\delta =e_1 + e_2$ and $n_{\phi_1}(\delta) = 1.$ 
Thus $\phi_1$ is the unique non-compact simple root and $n_{\phi_1}(\delta) = 1.$ Consequently, $\Delta^+$ is a 
Borel-de Siebenthal positive root system of $\Delta(\frak{g},\frak{h}).$ We have $H_{\phi_j}^*= i(H_j-H_{j+1})$ for all $1\leq j \leq l-1,$ and $H_{\phi_l}^*=2iH_l$ if 
$m = 2l-1,\ H_{\phi_l}^*=i(H_{l-1} + H_l)$ if $m = 2l-2.$ Also 
\[
\begin{cases} 
\{i(H_j-H_{j+1}), 2iH_l : 1 \le j \le l-1\} \cup \{\frac{1}{2}G_{jk}^+, - \frac{1}{2}\overline{G_{jk}^+}, \frac{1}{2}G_{jk}^-, - \frac{1}{2}\overline{G_{jk}^-} : 2 \le j < k \le l \} \\ 
\cup \{D_j^+, -\overline{D_j^+} : 2 \le j \le l \} \cup \{\frac{1}{2}G_{1k}^+,  \frac{1}{2}\overline{G_{1k}^+}, \frac{1}{2}G_{1k}^-,  \frac{1}{2}\overline{G_{jk}^-} : 2 \le k \le l \} \\ 
\cup \{D_1^+, \overline{D_1^+}\} \textrm{ if } m = 2l-1, \\ 
\{i(H_j-H_{j+1}), i(H_{l-1}+H_l) : 1 \le j \le l-1\} \cup \{\frac{1}{2}G_{jk}^+, - \frac{1}{2}\overline{G_{jk}^+}, \frac{1}{2}G_{jk}^-, - \frac{1}{2}\overline{G_{jk}^-} : \\ 2 \le j < k \le l \} 
 \cup \{\frac{1}{2}G_{1k}^+,  \frac{1}{2}\overline{G_{1k}^+}, \frac{1}{2}G_{1k}^-,  \frac{1}{2}\overline{G_{jk}^-} : 2 \le k \le l \}\textrm{ if } m = 2l-2,
\end{cases}
\] is a Chevalley basis of $\frak{g}$ satisfying \ref{chevalley}. 

Let $m = 2l-2$ and define $\frak{t}'_0 =\mathbb{R}H_1 \oplus \sum_{2 \le j \le l} \mathbb{R}F_{1+j,l+j}.$ Then $\frak{t}'_0$ is a maximal abelian subalgebra of $\frak{k}_0$ and 
$\frak{h}' = {\frak{t}'}_0^\mathbb{C}$ is a Cartan subalgebra of $\frak{g}.$ For $1 \le j \le l,$  $\epsilon_j \in (\frak{h}')^*$ be defined by $\epsilon_1(H_1)=-i, \epsilon_1(F_{1+k,l+k}) = 0$ and 
$\epsilon_j(H_1)= 0, \epsilon_j(F_{1+k,l+k}) = i\delta_{jk}$ for all $2 \le j,k \le l.$ Then $\Delta(\frak{g}, \frak{h}') = 
\{\pm(\epsilon_j-\epsilon_k), \pm(\epsilon_j+\epsilon_k): 1 \le j < k \le l \} \textrm{ and}$  
${\Delta'}^+ = \{\epsilon_j-\epsilon_k, \epsilon_j+\epsilon_k: 1 \le j < k \le l \}$
is a positive root system of $\Delta(\frak{g}, \frak{h}')$ with the set of simple roots $\Phi' = \{\phi'_1, \phi'_2, \ldots ,\phi'_l \},$ where 
$\phi'_j = \epsilon_j-\epsilon_{j+1}$ for all $1\leq j \leq l-1$ and $\phi'_l = \epsilon_{l-1} + \epsilon_l.$ 
Define 
\[ {G'}_{1k}^+ = f(F_{1,l+k} + F_{2,1+k}+i(F_{1,1+k}-F_{2,l+k})) \textrm{ for } 2 \le k \le l, \]
\[ {G'}_{jk}^+ = -F_{1+j,1+k}-F_{l+j, l+k}+i(F_{1+j,l+k} + F_{1+k,l+j}) \textrm{ for } 2 \le j < k \le l, \] 
\[ {G'}_{1k}^- = f(-F_{1,l+k} + F_{2,1+k}+i(F_{1,1+k}+F_{2,l+k})) \textrm{ for } 2 \le k \le l, \]
\[{G'}_{jk}^- = -F_{1+j,1+k} + F_{l+j, l+k}+i(-F_{1+j,l+k} + F_{1+k,l+j})\textrm{ for } 2 \le j < k \le l.\]
We have $\frak{g}^{(\epsilon_j-\epsilon_k)} = \mathbb{C}{G'}_{jk}^+,\ \frak{g}^{-(\epsilon_j-\epsilon_k)} = \mathbb{C} \overline{{G'}_{jk}^+},\ 
\frak{g}^{(\epsilon_j+\epsilon_k)} = \mathbb{C}{G'}_{jk}^-, \ \frak{g}^{-(\epsilon_j+\epsilon_k)} = \mathbb{C}\overline{{G'}_{jk}^-}$ for $1 \le j < k \le l.$ 
We have $\frak{g}^{\phi'_j} \subset \frak{k}$ for $2 \le j \le l$ and 
$\frak{g}^{\phi'_1} \subset \frak{p}.$ Also the highest root $\delta' =\epsilon_1 + \epsilon_2$ and $n_{\phi'_1}(\delta') = 1.$ 
Thus $\phi'_1$ is the unique non-compact simple root and $n_{\phi'_1}(\delta') = 1.$ Consequently, ${\Delta'}^+$ is a 
Borel-de Siebenthal positive root system of $\Delta(\frak{g},\frak{h}').$ We have $H_{\phi'_1}^*= i(H_1+F_{3,l+2}),$ 
$H_{\phi'_j}^*= -i(F_{1+j,l+j} - F_{2+j,l+1+j})$ for all $2\leq j \leq l-1$ and $H_{\phi'_l}^*=-i(F_{l,2l-1} + F_{l+1,2l})$. Also 
$\{ i(H_1+F_{3,l+2}), -i(F_{1+j,l+j} - F_{2+j,l+1+j}), -i(F_{l,2l-1} + F_{l+1,2l}) : 2 \le j \le l-1\} \cup $ \\ 
$\{\frac{1}{2}{G'}_{jk}^+, - \frac{1}{2}\overline{{G'}_{jk}^+}, \frac{1}{2}{G'}_{jk}^-, - \frac{1}{2}\overline{{G'}_{jk}^-} : 2 \le j < k \le l \}
\cup \{\frac{1}{2}{G'}_{1k}^+,  \frac{1}{2}\overline{{G'}_{1k}^+}, \frac{1}{2}{G'}_{1k}^-,  \frac{1}{2}\overline{{G'}_{jk}^-} : 2 \le k \le l \}$ 
is a Chevalley basis of $\frak{g}$ satisfying \ref{chevalley}. 

The associated bases $B, B'$ of $\frak{g}_0$ satisfying \ref{gbasis} and the corresponding cocompact arithmetic lattices 
$\Gamma, \Gamma'$ of Aut$(\frak{g}_0)$ will be used in the latter sections. 

\begin{remark} 
Note that $\frak{t}'_0$ is a maximal abelian subalgebra of $\frak{k}_0$ for $m = 2l-1$ also. But we need it only for $m = 2l-2$ and so 
defined it accordingly. 
\end{remark}

\noindent
\section{Involutions of $SO_0(2,m)$ commuting with $\theta$}\label{inv}
  
Let $\frak{g}_0$ be a real simple Lie algebra and $\frak{g}_0 = \frak{k}_0 \oplus \frak{p}_0$ be a 
Cartan decomposition with the Cartan involution $\theta.$ Assume that $\frak{k}_0$ has non-zero centre $\frak{z}.$ Then rank$(\frak{g}_0) =$ rank$(\frak{k}_0)$ and 
the Riemannian globally symmetric space associated with the pair $(\frak{g}_0, \frak{k}_0)$ is Hermitian symmetric. 
Let $\frak{t}_0$ be a maximal abelian subalgebra of $\frak{k}_0.$ Then $\frak{h} = \frak{t}_0^\mathbb{C}$ is a Cartan subalgebra of $\frak{k}=\frak{k}_0^\mathbb{C}$ 
as well as of $\frak{g} = \frak{g}_0^\mathbb{C}.$ Let $\Delta^+$ is a Borel-de Siebenthal positive root system of 
$\Delta = \Delta(\frak{g}, \frak{h})$ and $\Phi$ be the set of all simple roots in $\Delta^+.$ 
By an involution, we mean a Lie algebra automorphism of arder $2.$ 
The involutions of $\frak{g}_0$ commuting with $\theta$ (up to conjugation by an element of Aut$(\frak{g}_0)$) 
are classified in \cite{chuah}. We will describe briefly the process of the classification in \cite{chuah}. If $\frak{g}_0 = \frak{so}(2,m),$ up to conjugation, 
the involutions of $\frak{g}_0$ which are commuting with $\theta$ and not $\theta,$ are deduced in Tables \ref{inv1} and \ref{inv2}. 
  
Let $D^1$ be the extended dynkin diagram of $\frak{g}$ that is, it is a diagram where the extra vertex $-\delta$ has been added to the Dynkin diagram of $\frak{g}$ with the vertices $\phi \in \Phi$ 
for which $B(H_\delta, H_\phi) > 0.$ The vertices of $D^1$ are $\phi (\phi \in \Phi)$ and $-\delta.$ Each vertex $\psi \in \Phi \cup \{-\delta\}$ is associated with a positive integer 
$a_\psi,$ where $a_{-\delta} = 1$ and $a_\phi = n_\phi(\delta),$ the coefficient of $\phi$ in $\delta,$ when expressed as a sum of elements of $\Phi,$ for all $\phi \in \Phi.$ 
The {\it affine Vogan diagram} \cite{chuah} associated with $\theta$ is the extended  Dynkin diagram $D^1$ with the compact roots in $\Phi \cup \{-\delta\}$ are painted white and 
the non-compact roots in $\Phi \cup \{-\delta\}$ are painted black. Note that there are exactly two black vertices namely, $\nu$ and $-\delta.$ 
The centre $\frak{z}$ of the subalgebra $\frak{k}$ is one-dimensional and 
the Dynkin diagram $D_\frak{k}$ of the derived algebra $[\frak{k}, \frak{k}]$ is  the subdiagram 
of the affine Vogan diagram consisting of the vertices $\{\psi : \psi \textrm{ is painted white}\} = \Phi \setminus \{\nu\}.$ \cite[Th. 5.15(ii), Ch. X]{helgason}. 

Now each vertex $\psi$ of $D_\frak{k}$ is associated to a 
canonical positive integer $c_\psi$ such that on each component $C$ of $D_\frak{k},$  
$\sum_{\textrm{vertices of } C} c_\psi \psi =$ the highest root of the simple ideal associated with $C.$ A {\it Vogan diagram} 
of $D_\frak{k}$ is a diagram automorphism of order $1$ or $2$ on $D_\frak{k}$ such that the vertices fixed by the automorphism are uncircled or circled. 
The trivial Vogan diagram is the Vogan diagram of $D_\frak{k}$ with the trivial diagram automorphism and the vertices are all uncircled.  
The involutions of $\frak{k}_0$ (up to conjugation) are in bijective correspondence with the non-trivial Vogan diagrams of $D_\frak{k}$ with at most one circled vertex 
$\psi$ in each component of $D_\frak{k},$ where $c_\psi = 1$ or $2$ \cite{bds}, \cite[Th. 6.96]{knappb}. 

  Let inv$(\frak{g}_0)$ (resp. inv$(\frak{k}_0)$) be the set of all involutions of $\frak{g}_0$ (resp. $\frak{k}_0$). If 
$\sigma \in \textrm{inv}(\frak{g}_0)$ be such that $\sigma \theta = \theta \sigma$ and $\sigma \neq \theta,$ then $\sigma \in \textrm{inv}(\frak{k}_0).$
Conversely assume that $\sigma \in \textrm{inv}(\frak{k}_0).$ Corresponding to $\sigma,$ we have a non-trivial Vogan diagram of $D_\frak{k}$ with at most one circled vertex 
$\psi$ in each component of $D_\frak{k},$ where $c_\psi = 1$ or $2.$ Note that a vertex $\psi$ is uncircled (resp. circled) if $\sigma(\psi)=\psi$ and $\sigma = id$ 
(resp. $\sigma = -id$) on the corresponding root space. If the diagram automorphism of the Vogan diagram of $D_\frak{k}$ can be extended to the affine Vogan 
diagram of $\frak{g}$ corresponding to $\theta,$ then there are the following possibilities: \\ 
(a) If $\sigma|_{\frak{z}} = id,$ then the two black vertices are fixed by $\sigma.$  \\ 
(b)  If $\sigma|_{\frak{z}} = -id,$ then the two black vertices are interchanged by $\sigma.$ \\
See \cite[Prop. 2.4]{chuah}. An {\it almost double Vogan diagram} \cite{chuah} is the affine Vogan diagram $D^1$ with a diagram automorphism of order $1$ or $2$ preserving 
the vertex colors, such that the vertices fixed by the automorphism are uncircled or circled. Thus the diagram automorphism of the Vogan diagram of $D_\frak{k}$ can be extended to the 
affine Vogan diagram $D^1$ {\it iff}  the Vogan diagram of $D_\frak{k}$ can be extended to an almost double Vogan diagram. 
Now if the Vogan diagram of $D_\frak{k}$ can be extended to an almost double Vogan diagram, there are four possibilities for a vertex $\psi$ in the almost double 
Vogan diagram: \\
an uncircled vertex $\psi$ fixed by $\sigma;$ \\ 
a circled vertex $\psi$ fixed by $\sigma;$ \\ 
distinct vertices $\{\psi, \sigma(\psi)\}$ which are adjacent; \\ 
distinct vertices $\{\psi, \sigma(\psi)\}$ which are not adjacent. \\ 
Let $\mathcal{O}$ denote the vertices of the second and third type, that is $\mathcal{O} = \{\textrm{circled vertices}\} \cup \{\psi, \sigma (\psi) : 
\sigma(\psi) \neq \psi, \psi \textrm{ and } \sigma(\psi) \textrm{ are adjacent}  \}.$ We have an involution of $\frak{k}_0$ can be extended to an involution of $\frak{g}_0$ 
with $\sigma \theta = \theta \sigma$ {\it iff} the Vogan diagram of $D_\frak{k}$ can be extended to an almost double Vogan diagram  
and $\sum_{\mathcal{O}} a_\psi =$ even \cite[Th. 1.3]{chuah}. 

 Let $\sigma \in \textrm{inv}(\frak{g}_0)$ with $\sigma \theta = \theta \sigma.$  
Then $\sigma$ induces a holomorphic involution on the Hermitian symmetric space associated with the pair $(\frak{g}_0, \frak{k}_0)$ 
{\it iff} $\sigma|_{\frak{z}} = id,$ for the complex structure of this Hermitian symmetric space is induced from an element of $\frak{z}.$ Thus 
 $\sigma$ induces a holomorphic involution on the Hermitian symmetric space associated with the pair $(\frak{g}_0, \frak{k}_0)$ 
 {\it iff} $\sigma$ fixed the two black vertices $\nu$ and $-\delta.$ 
 
 If $\frak{g}_0 = \frak{so}(2,m)\ (m \neq 2),$ the highest root \[ \delta = 
\begin{cases} 
\phi_1 + 2\phi_2 + \cdots + 2\phi_l \textrm{ if } m = 2l-1, \\ 
\phi_1 + 2\phi_2 + \cdots + 2\phi_{l-2} + \phi_{l-1} + \phi_l  \textrm{ if } m = 2l-2. 
\end{cases} 
\]
In this case, the Vogan diagrams of involutions of $\frak{k}_0$ and the almost double Vogan diagrams of their 
extensions (if possible) to involutions of $\frak{g}_0$ are described in the Tables \ref{inv1} and \ref{inv2}.  

\begin{table}[!h]
\caption{Involutions of  $\frak{so}(2, 2l-1)$ which are commuting with $\theta$ but not $\theta$}\label{inv1}
\begin{tabular}{||c|c||}
\hline
inv$(\mathfrak{k}_0)$ & inv$(\mathfrak{g}_0)$  \\
    
\hline
\hline

\begin{tabular}{c}   
\begin{tikzpicture}
\draw (1,0) circle [radius = 0.1]; 
\draw (2,0) circle [radius = 0.1]; 
\draw  (3.25,0) circle [radius = 0.1];
\draw  (3.25,0) circle [radius = 0.15];
\draw (4.5,0) circle [radius = 0.1]; 
\draw (5.5,0) circle [radius = 0.1]; 
\node [above] at (1.05,0.05) {$\phi_2$}; 
\node [above] at (2.05,0.05) {$\phi_3$}; 
\node [above] at (3.30,0.05) {$\phi_p$};
\node [above] at (4.75,0.05) {$\phi_{l-1}$}; 
\node [above] at (5.55,0.05) {$\phi_l$}; 
\node [left] at (0.5,0) {$\sigma_p|_{[\mathfrak{k},\mathfrak{k}]} :$};
\draw (1.1,0) -- (1.9,0);
\draw (2.1,0) -- (2.5,0); 
\draw [dotted] (2.5,0) -- (3,0);
\draw (2.80,0) -- (3.1,0);
\draw (3.40,0) -- (3.67,0);
\draw [dotted](3.67,0) -- (3.95,0);
\draw (4,0) -- (4.4,0); 
\draw (5.4,0) -- (5.3,0.1); 
\draw (5.4,0) -- (5.3,-0.1); 
\draw (4.6,0.025) -- (5.35,0.025); 
\draw (4.6,-0.025) -- (5.35,-0.025);
 \end{tikzpicture}    \\
 
$\sigma_p|_{\mathfrak{z}} = id \textrm{ for all } 2\leq p \leq l,\ l \ge 2.$ \\
 
\end{tabular} & 

\begin{tabular}{c}  
\begin{tikzpicture}
\filldraw [black] (0,-.5) circle [radius = 0.1]; 
\filldraw[black] (0,.5) circle [radius = 0.1];
\draw (1,-0.08) -- (0,-.5);
\draw (1,0.08) -- (0,.5);
\node [above] at (-.34,0.3) {$\phi_1$};
\node [left] at (0,-.5) {$-\delta $};

\draw (1,0) circle [radius = 0.1]; 
\draw (2,0) circle [radius = 0.1]; 
\draw  (3.25,0) circle [radius = 0.1];
\draw  (3.25,0) circle [radius = 0.15];
\draw (4.5,0) circle [radius = 0.1]; 
\draw (5.5,0) circle [radius = 0.1]; 
\node [above] at (1.05,0.05) {$\phi_2$}; 
\node [above] at (2.05,0.05) {$\phi_3$}; 
\node [above] at (3.30,0.05) {$\phi_p$};
\node [above] at (4.75,0.05) {$\phi_{l-1}$}; 
\node [above] at (5.55,0.05) {$\phi_l$}; 
\node [left] at (-0.5,0) {$\sigma_p :$};
\draw (1.1,0) -- (1.9,0);
\draw (2.1,0) -- (2.5,0); 
\draw [dotted] (2.5,0) -- (3,0);
\draw (2.80,0) -- (3.1,0);
\draw (3.40,0) -- (3.67,0);
\draw [dotted](3.67,0) -- (3.95,0);
\draw (4,0) -- (4.4,0); 
\draw (5.4,0) -- (5.3,0.1); 
\draw (5.4,0) -- (5.3,-0.1); 
\draw (4.6,0.025) -- (5.35,0.025); 
\draw (4.6,-0.025) -- (5.35,-0.025);
 \end{tikzpicture}  
 \\
 \begin{tikzpicture}
 \filldraw [black] (0,-.5) circle [radius = 0.1]; 
 \draw [black] (0,-.5) circle [radius = 0.15];
\filldraw[black] (0,.5) circle [radius = 0.1];
\draw[black] (0,.5) circle [radius = 0.15];
\draw (1,-0.08) -- (0,-.5);
\draw (1,0.08) -- (0,.5);
\node [above] at (-.34,0.3) {$\phi_1$};
\node [left] at (-.04,-.5) {$-\delta $};

\draw (1,0) circle [radius = 0.1]; 
\draw (2,0) circle [radius = 0.1]; 
\draw  (3.25,0) circle [radius = 0.1];
\draw  (3.25,0) circle [radius = 0.15];
\draw (4.5,0) circle [radius = 0.1]; 
\draw (5.5,0) circle [radius = 0.1]; 
\node [above] at (1.05,0.05) {$\phi_2$}; 
\node [above] at (2.05,0.05) {$\phi_3$}; 
\node [above] at (3.30,0.05) {$\phi_p$};
\node [above] at (4.75,0.05) {$\phi_{l-1}$}; 
\node [above] at (5.55,0.05) {$\phi_l$}; 
\node [left] at (-0.5,0) {$\sigma_p\theta :$};
\draw (1.1,0) -- (1.9,0);
\draw (2.1,0) -- (2.5,0); 
\draw [dotted] (2.5,0) -- (3,0);
\draw (2.80,0) -- (3.1,0);
\draw (3.40,0) -- (3.67,0);
\draw [dotted](3.67,0) -- (3.95,0);
\draw (4,0) -- (4.4,0); 
\draw (5.4,0) -- (5.3,0.1); 
\draw (5.4,0) -- (5.3,-0.1); 
\draw (4.6,0.025) -- (5.35,0.025); 
\draw (4.6,-0.025) -- (5.35,-0.025);
\end{tikzpicture}

\\

\end{tabular} \\ 

\hline 

\begin{tabular}{c} 
\begin{tikzpicture}
 \draw (1,0) circle [radius = 0.1]; 
\draw (2,0) circle [radius = 0.1]; 
\draw  (3.25,0) circle [radius = 0.1];
\draw  (3.25,0) circle [radius = 0.15];
\draw (4.5,0) circle [radius = 0.1]; 
\draw (5.5,0) circle [radius = 0.1]; 
\node [above] at (1.05,0.05) {$\phi_2$}; 
\node [above] at (2.05,0.05) {$\phi_3$}; 
\node [above] at (3.30,0.05) {$\phi_p$};
\node [above] at (4.75,0.05) {$\phi_{l-1}$}; 
\node [above] at (5.55,0.05) {$\phi_l$}; 
\node [left] at (0.5,0) {$\tau_p|_{[\mathfrak{k},\mathfrak{k}]} :$};
\draw (1.1,0) -- (1.9,0);
\draw (2.1,0) -- (2.5,0); 
\draw [dotted] (2.5,0) -- (3,0);
\draw (2.80,0) -- (3.1,0);
\draw (3.40,0) -- (3.67,0);
\draw [dotted](3.67,0) -- (3.95,0);
\draw (4,0) -- (4.4,0); 
\draw (5.4,0) -- (5.3,0.1); 
\draw (5.4,0) -- (5.3,-0.1); 
\draw (4.6,0.025) -- (5.35,0.025); 
\draw (4.6,-0.025) -- (5.35,-0.025);
 \end{tikzpicture}  \\
 
$\tau_p|_\mathfrak{z} = - id \textrm{ for all } 1 \leq p \leq l,$ \\

 $\tau_1|_{[\mathfrak{k},\mathfrak{k}]} = id, l \ge 1.$ \\ 
 
 \end{tabular} & 
 
\begin{tabular}{c}  
\begin{tikzpicture}
\filldraw [black] (0,-.5) circle [radius = 0.1]; 
\filldraw[black] (0,.5) circle [radius = 0.1];
\draw (1,-0.08) -- (0,-.5);
\draw (1,0.08) -- (0,.5);
\node [above] at (-.34,0.3) {$\phi_1$};
\node [left] at (0,-.5) {$-\delta $};
\draw  (-0,-0.6) --(0,0.6) ;
\draw  (-.1,-0.3) --(0,-0.45) ;
\draw  (.1,-0.3) --(0,-0.45) ;
\draw  (-.1,0.3) --(0,0.45) ;
\draw  (.1,0.3) --(0,0.45) ;

\draw (1,0) circle [radius = 0.1]; 
\draw (2,0) circle [radius = 0.1]; 
\draw  (3.25,0) circle [radius = 0.1];
\draw  (3.25,0) circle [radius = 0.15];
\draw (4.5,0) circle [radius = 0.1]; 
\draw (5.5,0) circle [radius = 0.1]; 
\node [above] at (1.05,0.05) {$\phi_2$}; 
\node [above] at (2.05,0.05) {$\phi_3$}; 
\node [above] at (3.30,0.05) {$\phi_p$};
\node [above] at (4.75,0.05) {$\phi_{l-1}$}; 
\node [above] at (5.55,0.05) {$\phi_l$}; 
\node [left] at (-0.5,0) {$\tau_p, \tau_p \theta :$};

\draw (1.1,0) -- (1.9,0);
\draw (2.1,0) -- (2.5,0); 
\draw [dotted] (2.5,0) -- (3,0);
\draw (2.80,0) -- (3.1,0);
\draw (3.40,0) -- (3.67,0);
\draw [dotted](3.67,0) -- (3.95,0);
\draw (4,0) -- (4.4,0); 
\draw (5.4,0) -- (5.3,0.1); 
\draw (5.4,0) -- (5.3,-0.1); 
\draw (4.6,0.025) -- (5.35,0.025); 
\draw (4.6,-0.025) -- (5.35,-0.025);
\end{tikzpicture}
 
 \\
  
\begin{tikzpicture}
\filldraw [black] (0,-.5) circle [radius = 0.1]; 
\filldraw[black] (0,.5) circle [radius = 0.1];
\draw (1,-0.08) -- (0,-.5);
\draw (1,0.08) -- (0,.5);
\node [above] at (-.34,0.3) {$\phi_1$};
\node [left] at (0,-.5) {$-\delta $};
\draw  (-0,-0.6) --(0,0.6) ;
\draw  (-.1,-0.3) --(0,-0.45) ;
\draw  (.1,-0.3) --(0,-0.45) ;
\draw  (-.1,0.3) --(0,0.45) ;
\draw  (.1,0.3) --(0,0.45) ;

\draw (1,0) circle [radius = 0.1]; 
\draw (2,0) circle [radius = 0.1]; 
\draw (3.5,0) circle [radius = 0.1]; 
\draw (4.5,0) circle [radius = 0.1]; 
\node [above] at (1.05,0.05) {$\phi_2$}; 
\node [above] at (2.05,0.05) {$\phi_3$}; 
\node [above] at (3.75,0.05) {$\phi_{l-1}$}; 
\node [above] at (4.55,0.05) {$\phi_l$}; 
\node [left] at (-0.5,0) {$\tau_1,\tau_1\theta :$};
 
\draw (1.1,0) -- (1.9,0);
\draw (2.1,0) -- (2.5,0); 
\draw [dotted] (2.5,0) -- (3,0); 
\draw (3,0) -- (3.4,0); 
\draw (4.4,0) -- (4.3,0.1); 
\draw (4.4,0) -- (4.3,-0.1); 
\draw (3.6,0.025) -- (4.35,0.025); 
\draw (3.6,-0.025) -- (4.35,-0.025); 
\end{tikzpicture} 

\end{tabular} \\  
\hline
\end{tabular}

\end{table}

\begin{table}[!h] 
\caption{Involutions of $\frak{so}(2, 2l-2)$ which are commuting with $\theta$ but not $\theta$} \label{inv2}
\begin{tabular}{||c|c||}
\hline

inv$(\mathfrak{k}_0)$  &  inv$(\mathfrak{g}_0)$   \\

\hline
\hline

\begin{tabular}{c}  

\begin{tikzpicture}
\draw (1,0) circle [radius = 0.1]; 
\draw (2,0) circle [radius = 0.1]; 
\draw  (3.25,0) circle [radius = 0.1];
\draw  (3.25,0) circle [radius = 0.15];
\node [above] at (1.05,0.05) {$\phi_2$}; 
\node [above] at (2.05,0.05) {$\phi_3$}; 
\node [above] at (3.30,0.05) {$\phi_p$};
\draw (1.1,0) -- (1.9,0);
\draw (2.1,0) -- (2.5,0); 
\draw [dotted] (2.5,0) -- (3,0);
\draw (2.80,0) -- (3.1,0);
\draw (3.40,0) -- (3.67,0);
\draw [dotted](3.67,0) -- (3.95,0);
\draw (4,0) -- (4.4,0); 
\draw (4.5,0) circle [radius = 0.1]; 
\draw (5.5,0.5) circle [radius = 0.1]; 
\draw (5.5,-0.5) circle [radius = 0.1];
\node [above] at (4.5,0.05) {$\phi_{l-2}$}; 
\node [right] at (5.5,0.5) {$\phi_{l-1}$};
\node [right] at (5.5,-0.5) {$\phi_{l}$};
\node [left] at (.8,0) {$\sigma_p|_{[\mathfrak{k},\mathfrak{k}]} :$};
\draw (4.62,0) -- (5.4,0.5); 
\draw (4.62,0) -- (5.4,-0.5); 
\end{tikzpicture} 

\\

$\sigma_p|_\mathfrak{z}= id \textrm{ for all } 2 \leq p \leq l-2, l \ge 4.$ \\

\end{tabular} & 

\begin{tabular}{c}  
\begin{tikzpicture}
\filldraw [black] (0,0.5) circle [radius = 0.1]; 
\filldraw[black] (0,-.5) circle [radius = 0.1];
\node [left] at (0,0.5) {$\phi_1$};
\node [left] at (0,-.5) {$-\delta$};
\draw (0,.5) -- (.9,0);
\draw (0,-.5) -- (.9,0);

\draw (1,0) circle [radius = 0.1]; 
\draw (2,0) circle [radius = 0.1]; 
\draw  (3.25,0) circle [radius = 0.1];
\draw  (3.25,0) circle [radius = 0.15];

\node [above] at (1.05,0.05) {$\phi_2$}; 
\node [above] at (2.05,0.05) {$\phi_3$}; 
\node [above] at (3.30,0.05) {$\phi_p$};

\draw (1.1,0) -- (1.9,0);
\draw (2.1,0) -- (2.5,0); 
\draw [dotted] (2.5,0) -- (3,0);
\draw (2.80,0) -- (3.1,0);
\draw (3.40,0) -- (3.67,0);
\draw [dotted](3.67,0) -- (3.95,0);
\draw (4,0) -- (4.4,0); 
\draw (4.5,0) circle [radius = 0.1]; 
\draw (5.5,0.5) circle [radius = 0.1]; 
\draw (5.5,-0.5) circle [radius = 0.1];
 \node [above] at (4.5,0.05) {$\phi_{l-2}$}; 
\node [right] at (5.5,0.5) {$\phi_{l-1}$};
\node [right] at (5.5,-0.5) {$\phi_{l}$};
\node [left] at (-0.5,0) {$\sigma_p :$};
\draw (4.62,0) -- (5.4,0.5); 
\draw (4.62,0) -- (5.4,-0.5); 
\end{tikzpicture} 

\\

\begin{tikzpicture}
\filldraw [black] (0,.5) circle [radius = 0.1]; 
\filldraw [black] (0,-.5) circle [radius = 0.1];
\draw [black] (0,0.5) circle [radius = 0.15];
\draw [black] (0,-.5) circle [radius = 0.15];
\node [left] at (0,0.5) {$\phi_1$}; 
\node [left] at (-.02,-.5) {$-\delta$};
\draw (0,.5) -- (.9,0);
\draw (0,-.5) -- (.9,0);

\draw (1,0) circle [radius = 0.1]; 
\draw (2,0) circle [radius = 0.1]; 
\draw  (3.25,0) circle [radius = 0.1];
\draw  (3.25,0) circle [radius = 0.15];

\node [above] at (1.05,0.05) {$\phi_2$}; 
\node [above] at (2.05,0.05) {$\phi_3$}; 
\node [above] at (3.30,0.05) {$\phi_p$};

\draw (1.1,0) -- (1.9,0);
\draw (2.1,0) -- (2.5,0); 
\draw [dotted] (2.5,0) -- (3,0);
\draw (2.80,0) -- (3.1,0);
\draw (3.40,0) -- (3.67,0);
\draw [dotted](3.67,0) -- (3.95,0);
\draw (4,0) -- (4.4,0); 
\draw (4.5,0) circle [radius = 0.1]; 
\draw (5.5,0.5) circle [radius = 0.1]; 
\draw (5.5,-0.5) circle [radius = 0.1];
\node [above] at (4.5,0.05) {$\phi_{l-2}$}; 
\node [right] at (5.5,0.5) {$\phi_{l-1}$};
\node [right] at (5.5,-0.5) {$\phi_{l}$};
\node [left] at (-0.5,0) {$\sigma_p \theta:$};
\draw (4.62,0) -- (5.4,0.5); 
\draw (4.62,0) -- (5.4,-0.5); 
\end{tikzpicture} 

\end{tabular} \\  

\hline

\begin{tabular}{c} 

\begin{tikzpicture}
\draw (1,0) circle [radius = 0.1]; 
\draw (2,0) circle [radius = 0.1]; 
\draw  (3.25,0) circle [radius = 0.1];
\draw  (3.25,0) circle [radius = 0.15];
\draw (4.5,0) circle [radius = 0.1]; 

\node [above] at (1.05,0.05) {$\phi_2$}; 
\node [above] at (2.05,0.05) {$\phi_3$}; 
\node [above] at (3.30,0.05) {$\phi_p$};
   
\draw (1.1,0) -- (1.9,0);
\draw (2.1,0) -- (2.5,0); 
\draw [dotted] (2.5,0) -- (3,0);
\draw (2.80,0) -- (3.1,0);
\draw (3.40,0) -- (3.67,0);
\draw [dotted](3.67,0) -- (3.95,0);
\draw (4,0) -- (4.4,0); 
\draw (4.5,0) circle [radius = 0.1]; 
\draw (5.5,0.5) circle [radius = 0.1]; 
\draw (5.5,-0.5) circle [radius = 0.1];
\node [above] at (4.5,0.05) {$\phi_{l-2}$}; 
\node [right] at (5.5,0.5) {$\phi_{l-1}$};
\node [right] at (5.5,-0.5) {$\phi_{l}$};
\node [left] at (0.8,0) {$\tau_p|_{[\mathfrak{k},\mathfrak{k}]} :$};

\draw (4.62,0) -- (5.4,0.5); 
\draw (4.62,0) -- (5.4,-0.5); 
\end{tikzpicture}

\\

$\tau_p|_\mathfrak{z} = - id \textrm{ for all } 1 \leq p \leq l-2,$ \\

$\tau_1|_{[\mathfrak{k},\mathfrak{k}]} = id, l \ge 3.$

\end{tabular} & 

\begin{tabular}{c}  

\begin{tikzpicture}
\filldraw [black] (0,.5) circle [radius = 0.1];
\filldraw [black] (0,-.5) circle [radius = 0.1];
\node [left] at (0,0.5) {$\phi_1$};
\node [left] at (0,-.5) {$-\delta$};
\draw (0,-0.5) -- (0,.5);
\draw (0,-0.5) -- (0.9,0);
\draw (0,0.5) -- (0.9,0);
\draw (-.1,0.3) -- (0,.5);
\draw (.1,0.3) -- (0,.5);
\draw (-.1,-0.3) -- (0,-.5);
\draw (.1,0-.3) -- (0,-.5);

\draw (1,0) circle [radius = 0.1]; 
\draw (2,0) circle [radius = 0.1]; 
\draw  (3.25,0) circle [radius = 0.1];
\draw  (3.25,0) circle [radius = 0.15];
\draw (4.5,0) circle [radius = 0.1]; 

\node [above] at (1.05,0.05) {$\phi_2$}; 
\node [above] at (2.05,0.05) {$\phi_3$}; 
\node [above] at (3.30,0.05) {$\phi_p$};
   
\draw (1.1,0) -- (1.9,0);
\draw (2.1,0) -- (2.5,0); 
\draw [dotted] (2.5,0) -- (3,0);
\draw (2.80,0) -- (3.1,0);
\draw (3.40,0) -- (3.67,0);
\draw [dotted](3.67,0) -- (3.95,0);
\draw (4,0) -- (4.4,0); 

\draw (4.5,0) circle [radius = 0.1]; 
\draw (5.5,0.5) circle [radius = 0.1]; 
\draw (5.5,-0.5) circle [radius = 0.1];
\node [above] at (4.5,0.05) {$\phi_{l-2}$}; 
\node [right] at (5.5,0.5) {$\phi_{l-1}$};
\node [right] at (5.5,-0.5) {$\phi_{l}$};
\node [left] at (-0.5,0) {$\tau_p,\tau_p\theta :$};

\draw (4.62,0) -- (5.4,0.5); 
\draw (4.62,0) -- (5.4,-0.5); 
\end{tikzpicture} 

\\

\begin{tikzpicture}
\filldraw [black] (0,.5) circle [radius = 0.1];
\filldraw [black] (0,-.5) circle [radius = 0.1];
\node [left] at (0,0.5) {$\phi_1$};
\node [left] at (0,-.5) {$-\delta$};
\draw (0,-0.5) -- (0,.5);
\draw (0,-0.5) -- (0.9,0);
\draw (0,0.5) -- (0.9,0);
\draw (-.1,0.3) -- (0,.5);
\draw (.1,0.3) -- (0,.5);
\draw (-.1,-0.3) -- (0,-.5);
\draw (.1,0-.3) -- (0,-.5);
       
\draw (1,0) circle [radius = 0.1]; 
\draw (2,0) circle [radius = 0.1]; 

\node [above] at (1.05,0.05) {$\phi_2$}; 
\node [above] at (2.05,0.05) {$\phi_3$}; 
\node [above] at (3.4,0.05) {$\phi_{l-2}$};
 
\draw (1.1,0) -- (1.9,0);
\draw (2.1,0) -- (2.5,0); 
\draw [dotted] (2.5,0) -- (3,0); 
\draw (3,0) -- (3.4,0); 

\draw (3.5,0) circle [radius = 0.1];
\draw (4.5,0.5) circle [radius = 0.1];
\draw (4.5,-0.5) circle [radius = 0.1];

\node [right] at (4.5,0.5) {$\phi_{l-1}$}; 
\node [right] at (4.5,-0.5) {$\phi_l $}; 
\node [left] at (-0.8,0.2) {$\tau_1,\tau_1\theta :$};

\draw (3.6,0.025) -- (4.4,0.5); 
\draw (3.6,-0.025) -- (4.4,-0.5); 
\end{tikzpicture}

\end{tabular} \\  

\hline

\begin{tabular}{c}   

 \begin{tikzpicture}
\draw (1,0) circle [radius = 0.1]; 
\draw (2,0) circle [radius = 0.1]; 
\draw  (3.25,0) circle [radius = 0.1];
\draw  (3.25,0) circle [radius = 0.15];
\draw (4.5,0) circle [radius = 0.1]; 

\node [above] at (1.05,0.05) {$\phi_2$}; 
\node [above] at (2.05,0.05) {$\phi_3$}; 
\node [above] at (3.30,0.05) {$\phi_p$};
   
\draw (1.1,0) -- (1.9,0);
\draw (2.1,0) -- (2.5,0); 
\draw [dotted] (2.5,0) -- (3,0);
\draw (2.80,0) -- (3.1,0);
\draw (3.40,0) -- (3.67,0);
\draw [dotted](3.67,0) -- (3.95,0);
\draw (4,0) -- (4.4,0); 

\draw (4.5,0) circle [radius = 0.1]; 
\draw (5.5,0.5) circle [radius = 0.1]; 
\draw (5.5,-0.5) circle [radius = 0.1];
 \node [above] at (4.5,0.05) {$\phi_{l-2}$}; 
\node [right] at (5.5,0.5) {$\phi_{l-1}$};
\node [right] at (5.5,-0.5) {$\phi_{l}$};
\node [left] at (0.9,0) {$\mu_p|_{[\mathfrak{k},\mathfrak{k}]} :$};

\draw (4.62,0) -- (5.4,0.5); 
\draw (4.62,0) -- (5.4,-0.5); 
\draw (5.5,0.4) -- (5.5,-0.4);
\draw (5.4,0.3) -- (5.5,0.4);
\draw (5.6,0.3) -- (5.5,0.4);
\draw (5.4,-0.3) -- (5.5,-0.4);
\draw (5.6,-0.3) -- (5.5,-0.4);
\end{tikzpicture}

 \\

$\mu_p|_\mathfrak{z}= -id, \textrm{ for all } 1\leq p \leq l-2, l \ge 3,$ \\

\begin{tikzpicture}
\draw (1,0) circle [radius = 0.1]; 
\draw (2,0) circle [radius = 0.1]; 
\draw  (3.25,0) circle [radius = 0.1];
\draw (4.5,0) circle [radius = 0.1]; 

\node [above] at (1.05,0.05) {$\phi_2$}; 
\node [above] at (2.05,0.05) {$\phi_3$}; 
\node [above] at (3.30,0.05) {$\phi_p$};
   
\draw (1.1,0) -- (1.9,0);
\draw (2.1,0) -- (2.5,0); 
\draw [dotted] (2.5,0) -- (3,0);
\draw (2.80,0) -- (3.1,0);
\draw (3.40,0) -- (3.67,0);
\draw [dotted](3.67,0) -- (3.95,0);
\draw (4,0) -- (4.4,0); 

\draw (4.5,0) circle [radius = 0.1]; 
\draw (5.5,0.5) circle [radius = 0.1]; 
\draw (5.5,-0.5) circle [radius = 0.1];
 \node [above] at (4.5,0.05) {$\phi_{l-2}$}; 
\node [right] at (5.5,0.5) {$\phi_{l-1}$};
\node [right] at (5.5,-0.5) {$\phi_{l}$};
\node [left] at (0.9,0) {$\mu_1|_{[\mathfrak{k},\mathfrak{k}]} :$};

\draw (4.62,0) -- (5.4,0.5); 
\draw (4.62,0) -- (5.4,-0.5); 
\draw (5.5,0.4) -- (5.5,-0.4);
\draw (5.4,0.3) -- (5.5,0.4);
\draw (5.6,0.3) -- (5.5,0.4);
\draw (5.4,-0.3) -- (5.5,-0.4);
\draw (5.6,-0.3) -- (5.5,-0.4);
\end{tikzpicture}

\end{tabular} & 

\begin{tabular}{c}  
     
\begin{tikzpicture}
\filldraw [black] (0,0.5) circle [radius = 0.1]; 
\filldraw[black] (0,-.5) circle [radius = 0.1];
\node [left] at (0,0.5) {$\phi_1$};
\node [left] at (0,-.5) {$-\delta$};
\draw (0,0.5) -- (0.9,0);
\draw (0,-0.5) -- (0.9,0);
\draw (0,-0.5) --(0,0.5) ;
\draw (0,-0.5) --(0,0.5) ;
\draw (-.12,-0.3) --(0,-0.5) ;
\draw (.12,-0.3) --(0,-0.5) ;
\draw (-.12,0.3) --(0,0.5) ;
\draw (.12,0.3) --(0,0.5) ;

\draw (1,0) circle [radius = 0.1]; 
\draw (2,0) circle [radius = 0.1]; 
\draw  (3.25,0) circle [radius = 0.1];
\draw  (3.25,0) circle [radius = 0.15];
\draw (4.5,0) circle [radius = 0.1]; 

\node [above] at (1.05,0.05) {$\phi_2$}; 
\node [above] at (2.05,0.05) {$\phi_3$}; 
\node [above] at (3.30,0.05) {$\phi_p$};
   
\draw (1.1,0) -- (1.9,0);
\draw (2.1,0) -- (2.5,0); 
\draw [dotted] (2.5,0) -- (3,0);
\draw (2.80,0) -- (3.1,0);
\draw (3.40,0) -- (3.67,0);
\draw [dotted](3.67,0) -- (3.95,0);
\draw (4,0) -- (4.4,0); 

\draw (4.5,0) circle [radius = 0.1]; 
\draw (5.5,0.5) circle [radius = 0.1]; 
\draw (5.5,-0.5) circle [radius = 0.1];
 
\node [above] at (4.5,0.05) {$\phi_{l-2}$}; 
\node [right] at (5.5,0.5) {$\phi_{l-1}$};
\node [right] at (5.5,-0.5) {$\phi_{l}$};
\node [left] at (-0.5,0) {$\mu_p,\mu_p\theta :$};

\draw (4.62,0) -- (5.4,0.5); 
\draw (4.62,0) -- (5.4,-0.5); 
\draw (5.5,0.4) -- (5.5,-0.4);
\draw (5.4,0.3) -- (5.5,0.4);
\draw (5.6,0.3) -- (5.5,0.4);
\draw (5.4,-0.3) -- (5.5,-0.4);
\draw (5.6,-0.3) -- (5.5,-0.4);
\end{tikzpicture} 

 \\

\begin{tikzpicture}
\filldraw [black] (0,.5) circle [radius = 0.1];
\filldraw [black] (0,-.5) circle [radius = 0.1];
\node [left] at (0,0.5) {$\phi_1$};
\node [left] at (0,-.5) {$-\delta$};
\draw (1,0) circle [radius = 0.1]; 
\draw (2,0) circle [radius = 0.1]; 
\draw (0,-0.5) -- (0,.5);
\draw (0,-0.5) -- (0.9,0);
\draw (0,0.5) -- (0.9,0);
\draw (-.1,0.3) -- (0,.5);
\draw (.1,0.3) -- (0,.5);
\draw (-.1,-0.3) -- (0,-.5);
\draw (.1,0-.3) -- (0,-.5);

\draw (1,0) circle [radius = 0.1]; 
\draw (2,0) circle [radius = 0.1]; 

\node [above] at (1.05,0.05) {$\phi_2$}; 
\node [above] at (2.05,0.05) {$\phi_3$}; 
\node [above] at (3.4,0.05) {$\phi_{l-2}$}; 
 
\draw (1.1,0) -- (1.9,0);
\draw (2.1,0) -- (2.5,0); 
\draw [dotted] (2.5,0) -- (3,0); 
\draw (3,0) -- (3.4,0); 
\draw (3.5,0) circle [radius = 0.1];
\draw (4.5,0.5) circle [radius = 0.1];
\draw (4.5,-0.5) circle [radius = 0.1];
\node [right] at (4.5,0.5) {$\phi_{l-1}$}; 
\node [right] at (4.5,-0.5) {$\phi_l $}; 
\node [left] at (-0.8,0.2) {$\mu_1,\mu_1\theta :$};

\draw (3.6,0.025) -- (4.4,0.5); 
\draw (3.6,-0.025) -- (4.4,-0.5);
\draw (4.5,0.4) -- (4.5,-0.4);
\draw (4.4,0.2) -- (4.5,0.4);
\draw (4.6,0.2) -- (4.5,0.4);
\draw (4.4,-0.2) -- (4.5,-0.4);
\draw (4.6,-0.2) -- (4.5,-0.4);
\end{tikzpicture}
 
\end{tabular} \\ 
\hline 

\end{tabular}
\end{table}

\begin{table}[!h]
\begin{tabular}{||c|c||}
\hline

\begin{tabular}{c} 

\begin{tikzpicture}
\draw (1,0) circle [radius = 0.1]; 
\draw (2,0) circle [radius = 0.1]; 
\draw  (3.25,0) circle [radius = 0.1];
\draw  (3.25,0) circle [radius = 0.15];
\draw (4.5,0) circle [radius = 0.1]; 

\node [above] at (1.05,0.05) {$\phi_2$}; 
\node [above] at (2.05,0.05) {$\phi_3$}; 
\node [above] at (3.30,0.05) {$\phi_p$};
   
\draw (1.1,0) -- (1.9,0);
\draw (2.1,0) -- (2.5,0); 
\draw [dotted] (2.5,0) -- (3,0);
\draw (2.80,0) -- (3.1,0);
\draw (3.40,0) -- (3.67,0);
\draw [dotted](3.67,0) -- (3.95,0);
\draw (4,0) -- (4.4,0); 

\draw (4.5,0) circle [radius = 0.1]; 
\draw (5.5,0.5) circle [radius = 0.1]; 
\draw (5.5,-0.5) circle [radius = 0.1];
 \node [above] at (4.5,0.05) {$\phi_{l-2}$}; 
\node [right] at (5.5,0.5) {$\phi_{l-1}$};
\node [right] at (5.5,-0.5) {$\phi_{l}$};
\node [left] at (0.9,0) {$\tau_p'|_{[\mathfrak{k},\mathfrak{k}]} :$};

\draw (4.62,0) -- (5.4,0.5); 
\draw (4.62,0) -- (5.4,-0.5); 
\draw (5.5,0.4) -- (5.5,-0.4);
\draw (5.4,0.3) -- (5.5,0.4);
\draw (5.6,0.3) -- (5.5,0.4);
\draw (5.4,-0.3) -- (5.5,-0.4);
\draw (5.6,-0.3) -- (5.5,-0.4);
\end{tikzpicture}   \\
$\tau_p'|_\mathfrak{z}= id, \textrm{ for all } 1\leq p \leq l-2, l \ge 3,$ \\

\begin{tikzpicture}
\draw (1,0) circle [radius = 0.1]; 
\draw (2,0) circle [radius = 0.1]; 
\draw  (3.25,0) circle [radius = 0.1];
\draw (4.5,0) circle [radius = 0.1]; 

\node [above] at (1.05,0.05) {$\phi_2$}; 
\node [above] at (2.05,0.05) {$\phi_3$}; 
\node [above] at (3.30,0.05) {$\phi_p$};
   
\draw (1.1,0) -- (1.9,0);
\draw (2.1,0) -- (2.5,0); 
\draw [dotted] (2.5,0) -- (3,0);
\draw (2.80,0) -- (3.1,0);
\draw (3.40,0) -- (3.67,0);
\draw [dotted](3.67,0) -- (3.95,0);
\draw (4,0) -- (4.4,0); 

\draw (4.5,0) circle [radius = 0.1]; 
\draw (5.5,0.5) circle [radius = 0.1]; 
\draw (5.5,-0.5) circle [radius = 0.1];
 
\node [above] at (4.5,0.05) {$\phi_{l-2}$}; 
\node [right] at (5.5,0.5) {$\phi_{l-1}$};
\node [right] at (5.5,-0.5) {$\phi_{l}$};
\node [left] at (0.9,0) {$\tau_1'|_{[\mathfrak{k},\mathfrak{k}]} :$};

\draw (4.62,0) -- (5.4,0.5); 
\draw (4.62,0) -- (5.4,-0.5); 
\draw (5.5,0.4) -- (5.5,-0.4);
\draw (5.4,0.3) -- (5.5,0.4);
\draw (5.6,0.3) -- (5.5,0.4);
\draw (5.4,-0.3) -- (5.5,-0.4);
\draw (5.6,-0.3) -- (5.5,-0.4);
\end{tikzpicture} \\
 
 \end{tabular} & 

\begin{tabular}{c}  

\begin{tikzpicture}
\filldraw [black] (0,0.5) circle [radius = 0.1]; 
\filldraw[black] (0,-.5) circle [radius = 0.1];
\node [left] at (0,0.5) {$\phi_1$};
\node [left] at (0,-.5) {$-\delta$};
\draw (0,.5) -- (.9,0);
\draw (0,-.5) -- (.9,0);

\draw (1,0) circle [radius = 0.1]; 
\draw (2,0) circle [radius = 0.1]; 
\draw  (3.25,0) circle [radius = 0.1];
\draw  (3.25,0) circle [radius = 0.15];
\draw (4.5,0) circle [radius = 0.1]; 

\node [above] at (1.05,0.05) {$\phi_2$}; 
\node [above] at (2.05,0.05) {$\phi_3$}; 
\node [above] at (3.30,0.05) {$\phi_p$};
   
\draw (1.1,0) -- (1.9,0);
\draw (2.1,0) -- (2.5,0); 
\draw [dotted] (2.5,0) -- (3,0);
\draw (2.80,0) -- (3.1,0);
\draw (3.40,0) -- (3.67,0);
\draw [dotted](3.67,0) -- (3.95,0);
\draw (4,0) -- (4.4,0); 

\draw (4.5,0) circle [radius = 0.1]; 
\draw (5.5,0.5) circle [radius = 0.1]; 
\draw (5.5,-0.5) circle [radius = 0.1];
 \node [above] at (4.5,0.05) {$\phi_{l-2}$}; 
\node [right] at (5.5,0.5) {$\phi_{l-1}$};
\node [right] at (5.5,-0.5) {$\phi_{l}$};
\node [left] at (-0.5,0) {$\tau_p ' :$};

\draw (4.62,0) -- (5.4,0.5); 
\draw (4.62,0) -- (5.4,-0.5); 
\draw (5.5,0.4) -- (5.5,-0.4);
\draw (5.4,0.3) -- (5.5,0.4);
\draw (5.6,0.3) -- (5.5,0.4);
\draw (5.4,-0.3) -- (5.5,-0.4);
\draw (5.6,-0.3) -- (5.5,-0.4);
\end{tikzpicture}   \\

\begin{tikzpicture}
\filldraw [black] (0,.5) circle [radius = 0.1]; 
\filldraw [black] (0,-.5) circle [radius = 0.1];
\draw [black] (0,0.5) circle [radius = 0.15];
\draw [black] (0,-.5) circle [radius = 0.15];

\node [left] at (0,0.5) {$\phi_1$}; 
\node [left] at (-.02,-.5) {$-\delta$};
\draw (0,.5) -- (.9,0);
\draw (0,-.5) -- (.9,0);

\draw (1,0) circle [radius = 0.1]; 
\draw (2,0) circle [radius = 0.1]; 
\draw  (3.25,0) circle [radius = 0.1];
\draw  (3.25,0) circle [radius = 0.15];
\draw (4.5,0) circle [radius = 0.1]; 

\node [above] at (1.05,0.05) {$\phi_2$}; 
\node [above] at (2.05,0.05) {$\phi_3$}; 
\node [above] at (3.30,0.05) {$\phi_p$};
   
\draw (1.1,0) -- (1.9,0);
\draw (2.1,0) -- (2.5,0); 
\draw [dotted] (2.5,0) -- (3,0);
\draw (2.80,0) -- (3.1,0);
\draw (3.40,0) -- (3.67,0);
\draw [dotted](3.67,0) -- (3.95,0);
\draw (4,0) -- (4.4,0); 

\draw (4.5,0) circle [radius = 0.1]; 
\draw (5.5,0.5) circle [radius = 0.1]; 
\draw (5.5,-0.5) circle [radius = 0.1];
 
\node [above] at (4.5,0.05) {$\phi_{l-2}$}; 
\node [right] at (5.5,0.5) {$\phi_{l-1}$};
\node [right] at (5.5,-0.5) {$\phi_{l}$};
\node [left] at (-0.5,0) {$\tau_p '\theta :$};

\draw (4.62,0) -- (5.4,0.5); 
\draw (4.62,0) -- (5.4,-0.5); 
\draw (5.5,0.4) -- (5.5,-0.4);
\draw (5.4,0.3) -- (5.5,0.4);
\draw (5.6,0.3) -- (5.5,0.4);
\draw (5.4,-0.3) -- (5.5,-0.4);
\draw (5.6,-0.3) -- (5.5,-0.4);
\end{tikzpicture}    \\
 
\begin{tikzpicture}
\filldraw [black] (0,.5) circle [radius = 0.1];
\filldraw [black] (0,-.5) circle [radius = 0.1];
\node [left] at (0,0.5) {$\phi_1$}; 
\node [left] at (0,-.5) {$-\delta$};
\draw (0,-0.5) -- (.9,0);
\draw (0,0.5) -- (.9,0);
       
\draw (1,0) circle [radius = 0.1]; 
\draw (2,0) circle [radius = 0.1]; 

\node [above] at (1.05,0.05) {$\phi_2$}; 
\node [above] at (2.05,0.05) {$\phi_3$}; 
\node [above] at (3.4,0.05) {$\phi_{l-2}$}; 
 
\draw (1.1,0) -- (1.9,0);
\draw (2.1,0) -- (2.5,0); 
\draw [dotted] (2.5,0) -- (3,0); 
\draw (3,0) -- (3.4,0); 
\draw (3.5,0) circle [radius = 0.1];
\draw (4.5,0.5) circle [radius = 0.1];
\draw (4.5,-0.5) circle [radius = 0.1];

\node [right] at (4.5,0.5) {$\phi_{l-1}$}; 
\node [right] at (4.5,-0.5) {$\phi_l $}; 
\node [left] at (-0.5,0) {$\tau_1': $}; 

\draw (3.6,0.025) -- (4.4,0.5); 
\draw (3.6,-0.025) -- (4.4,-0.5);
\draw (4.5,0.4) -- (4.5,-0.4);
\draw (4.4,0.2) -- (4.5,0.4);
\draw (4.6,0.2) -- (4.5,0.4);
\draw (4.4,-0.2) -- (4.5,-0.4);
\draw (4.6,-0.2) -- (4.5,-0.4);
\end{tikzpicture}   \\

\begin{tikzpicture}
\filldraw [black] (0,.5) circle [radius = 0.1];
\draw [black] (0,.5) circle [radius = 0.15];
\filldraw [black] (0,-.5) circle [radius = 0.1];
\draw [black] (0,-.5) circle [radius = 0.15];
\node [left] at (-.02,-.5) {$-\delta$};
\node [left] at (0,0.5) {$\phi_1$};
\draw (0,.5) -- (.9,-0);
\draw (0,-.5) -- (.9,-0);
       
\draw (1,0) circle [radius = 0.1]; 
\draw (2,0) circle [radius = 0.1]; 

\node [above] at (1.05,0.05) {$\phi_2$}; 
\node [above] at (2.05,0.05) {$\phi_3$}; 
\node [above] at (3.4,0.05) {$\phi_{l-2}$}; 
 
\draw (1.1,0) -- (1.9,0);
\draw (2.1,0) -- (2.5,0); 
\draw [dotted] (2.5,0) -- (3,0); 
\draw (3,0) -- (3.4,0); 
\draw (3.5,0) circle [radius = 0.1];
\draw (4.5,0.5) circle [radius = 0.1];
\draw (4.5,-0.5) circle [radius = 0.1];
\node [right] at (4.5,0.5) {$\phi_{l-1}$}; 
\node [right] at (4.5,-0.5) {$\phi_l $}; 
\node [left] at (-0.5,0) {$\tau_1'\theta: $}; 
\draw (3.6,0.025) -- (4.4,0.5); 
\draw (3.6,-0.025) -- (4.4,-0.5);
\draw (4.5,0.4) -- (4.5,-0.4);
\draw (4.4,0.2) -- (4.5,0.4);
\draw (4.6,0.2) -- (4.5,0.4);
\draw (4.4,-0.2) -- (4.5,-0.4);
\draw (4.6,-0.2) -- (4.5,-0.4);
\end{tikzpicture}
\end{tabular} \\  
\hline
\begin{tabular}{c}   

\begin{tikzpicture}
\draw (1,0) circle [radius = 0.1]; 
\draw (2,0) circle [radius = 0.1]; 
\draw (3.5,0) circle [radius = 0.1]; 
\draw (4.5,0.5) circle [radius = 0.1];
\draw (4.5,-0.5) circle [radius = 0.15];
\draw (4.5,-0.5) circle [radius = 0.1];

\node [above] at (1.05,0.05) {$\phi_2$}; 
\node [above] at (2.05,0.05) {$\phi_3$}; 
\node [above] at (3.4,0.05) {$\phi_{l-2}$}; 
\node [right] at (4.5,0.5) {$\phi_{l-1}$}; 
\node [right] at (4.5,-0.5) {$\phi_l $}; 
\node [left] at (0.7,0) {$\sigma_{l}|_{[\mathfrak{k},\mathfrak{k}]}: $}; 

\draw (1.1,0) -- (1.9,0);
\draw (2.1,0) -- (2.5,0); 
\draw [dotted] (2.5,0) -- (3,0); 
\draw (3,0) -- (3.4,0); 
\draw (3.6,0.025) -- (4.4,0.5); 
\draw (3.6,-0.025) -- (4.35,-0.5); 
\end{tikzpicture}

 \\

$\sigma_l|_\mathfrak{z} = id, l \ge 3.$   \\

\end{tabular} & 

\begin{tabular}{c}  

\begin{tikzpicture}
\filldraw [black] (0,0.5) circle [radius = 0.1]; 
\filldraw[black] (0,-.5) circle [radius = 0.1];
\node [left] at (0,0.5) {$\phi_1$};
\node [left] at (0,-.5) {$-\delta$};
\draw [black] (0,0.5) circle [radius = 0.15];
\draw (0,.5) -- (.9,0);
\draw (0,-.5) -- (.9,0);

\draw (1,0) circle [radius = 0.1]; 
\draw (2,0) circle [radius = 0.1]; 
\draw (3.5,0) circle [radius = 0.1]; 
\draw (4.5,0.5) circle [radius = 0.1];
\draw (4.5,-0.5) circle [radius = 0.15];
\draw (4.5,-0.5) circle [radius = 0.1];

\node [above] at (1.05,0.05) {$\phi_2$}; 
\node [above] at (2.05,0.05) {$\phi_3$}; 
\node [above] at (3.4,0.05) {$\phi_{l-2}$}; 
\node [right] at (4.5,0.5) {$\phi_{l-1}$}; 
\node [right] at (4.5,-0.5) {$\phi_l $}; 
\node [left] at (-0.5,0) {$\sigma_{l}: $}; 
\draw (1.1,0) -- (1.9,0);
\draw (2.1,0) -- (2.5,0); 
\draw [dotted] (2.5,0) -- (3,0); 
\draw (3,0) -- (3.4,0); 
\draw (3.6,0.025) -- (4.4,0.5); 
\draw (3.6,-0.025) -- (4.35,-0.5); 
\end{tikzpicture}

\\

\begin{tikzpicture}
\filldraw [black] (0,0.5) circle [radius = 0.1]; 
\filldraw[black] (0,-.5) circle [radius = 0.1];
\node [left] at (0,0.5) {$\phi_1$};
\node [left] at (-.02,-.5) {$-\delta$};
\draw [black] (0,-0.5) circle [radius = 0.15];
\draw (0,.5) -- (.9,0);
\draw (0,-.5) -- (.9,0);

\draw (1,0) circle [radius = 0.1]; 
\draw (2,0) circle [radius = 0.1]; 
\draw (3.5,0) circle [radius = 0.1]; 
\draw (4.5,0.5) circle [radius = 0.1];
\draw (4.5,-0.5) circle [radius = 0.15];
\draw (4.5,-0.5) circle [radius = 0.1];
\node [above] at (1.05,0.05) {$\phi_2$}; 
\node [above] at (2.05,0.05) {$\phi_3$}; 
\node [above] at (3.4,0.05) {$\phi_{l-2}$}; 
\node [right] at (4.5,0.5) {$\phi_{l-1}$}; 
\node [right] at (4.5,-0.5) {$\phi_l $}; 
\node [left] at (-0.5,0) {$\sigma_{l}\theta: $}; 

\draw (1.1,0) -- (1.9,0);
\draw (2.1,0) -- (2.5,0); 
\draw [dotted] (2.5,0) -- (3,0); 
\draw (3,0) -- (3.4,0); 
\draw (3.6,0.025) -- (4.4,0.5); 
\draw (3.6,-0.025) -- (4.35,-0.5); 
\end{tikzpicture}
 
\end{tabular} \\ 

\hline 

\begin{tabular}{c} 

\begin{tikzpicture}
\draw (1,0) circle [radius = 0.1];
\draw (2,0) circle [radius = 0.1]; 
\draw (3.5,0) circle [radius = 0.1]; 
\draw (4.5,0.5) circle [radius = 0.1];
\draw (4.5,0.5) circle [radius = 0.15];
\draw (4.5,-0.5) circle [radius = 0.1];

\node [above] at (1.05,0.05) {$\phi_2$}; 
\node [above] at (2.05,0.05) {$\phi_3$}; 
\node [above] at (3.4,0.05) {$\phi_{l-2}$}; 
\node [right] at (4.5,0.5) {$\phi_{l-1}$}; 
\node [right] at (4.5,-0.5) {$\phi_l $}; 
\node [left] at (0.7,0) {$\sigma_{l-1}|_{[\mathfrak{k},\mathfrak{k}]}: $}; 

\draw (1.1,0) -- (1.9,0);
\draw (2.1,0) -- (2.5,0); 
\draw [dotted] (2.5,0) -- (3,0); 
\draw (3,0) -- (3.4,0);
\draw (3.5,0) circle [radius = 0.1];
\draw (3.6,0.025) -- (4.35,0.45); 
\draw (3.6,-0.025) -- (4.4,-0.5); 
\end{tikzpicture} 

\\

$\sigma_{l-1}|_\mathfrak{z} = id, l \ge 3.$\\

\end{tabular} & 

\begin{tabular}{c}  

\begin{tikzpicture}
\filldraw [black] (0,0.5) circle [radius = 0.1]; 
\filldraw[black] (0,-.5) circle [radius = 0.1];
\node [left] at (0,0.5) {$\phi_1$};
\node [left] at (0,-.5) {$-\delta$};
\draw [black] (0,0.5) circle [radius = 0.15];
\draw (0,.5) -- (.9,0);
\draw (0,-.5) -- (.9,0);
   
\draw (1,0) circle [radius = 0.1];
\draw (2,0) circle [radius = 0.1]; 
\draw (3.5,0) circle [radius = 0.1]; 
\draw (4.5,0.5) circle [radius = 0.1];
\draw (4.5,0.5) circle [radius = 0.15];
\draw (4.5,-0.5) circle [radius = 0.1];

\node [above] at (1.05,0.05) {$\phi_2$}; 
\node [above] at (2.05,0.05) {$\phi_3$}; 
\node [above] at (3.4,0.05) {$\phi_{l-2}$}; 
\node [right] at (4.5,0.5) {$\phi_{l-1}$}; 
\node [right] at (4.5,-0.5) {$\phi_l $}; 
\node [left] at (-0.5,0) {$\sigma_{l-1}: $}; 

\draw (1.1,0) -- (1.9,0);
\draw (2.1,0) -- (2.5,0); 
\draw [dotted] (2.5,0) -- (3,0); 
\draw (3,0) -- (3.4,0);

\draw (3.5,0) circle [radius = 0.1];
\draw (3.6,0.025) -- (4.35,0.45); 
\draw (3.6,-0.025) -- (4.4,-0.5); 
\end{tikzpicture} 

\\

\begin{tikzpicture}
\filldraw [black] (0,0.5) circle [radius = 0.1]; 
\filldraw[black] (0,-.5) circle [radius = 0.1];
\node [left] at (0,0.5) {$\phi_1$};
\node [left] at (-.02,-.5) {$-\delta$};
\draw [black] (0,-0.5) circle [radius = 0.15];
\draw (0,.5) -- (.9,0);
\draw (0,-.5) -- (.9,0);
   
\draw (1,0) circle [radius = 0.1]; 
\draw (2,0) circle [radius = 0.1]; 
\draw (3.5,0) circle [radius = 0.1]; 
\draw (4.5,0.5) circle [radius = 0.1];
\draw (4.5,0.5) circle [radius = 0.15];
\draw (4.5,-0.5) circle [radius = 0.1];

\node [above] at (1.05,0.05) {$\phi_2$}; 
\node [above] at (2.05,0.05) {$\phi_3$}; 
\node [above] at (3.4,0.05) {$\phi_{l-2}$}; 
\node [right] at (4.5,0.5) {$\phi_{l-1}$}; 
\node [right] at (4.5,-0.5) {$\phi_l $}; 
\node [left] at (-0.5,0) {$\sigma_{l-1}\theta: $}; 

\draw (1.1,0) -- (1.9,0);
\draw (2.1,0) -- (2.5,0); 
\draw [dotted] (2.5,0) -- (3,0); 
\draw (3,0) -- (3.4,0);

\draw (3.5,0) circle [radius = 0.1];
\draw (3.6,0.025) -- (4.35,0.45); 
\draw (3.6,-0.025) -- (4.4,-0.5); 
\end{tikzpicture} 

\end{tabular} \\  

\hline

\begin{tabular}{c}

\begin{tikzpicture}
\draw (1,0) circle [radius = 0.1]; 
\draw (2,0) circle [radius = 0.1]; 
\draw (3.5,0) circle [radius = 0.1]; 
\draw (4.5,0.5) circle [radius = 0.1];
\draw (4.5,-0.5) circle [radius = 0.15];
\draw (4.5,-0.5) circle [radius = 0.1];

\node [above] at (1.05,0.05) {$\phi_2$}; 
\node [above] at (2.05,0.05) {$\phi_3$}; 
\node [above] at (3.4,0.05) {$\phi_{l-2}$}; 
\node [right] at (4.5,0.5) {$\phi_{l-1}$}; 
\node [right] at (4.5,-0.5) {$\phi_l $}; 
\node [left] at (.6,0) {$\tau_{l}|_{[\mathfrak{k},\mathfrak{k}]}: $}; 

\draw (1.1,0) -- (1.9,0);
\draw (2.1,0) -- (2.5,0); 
\draw [dotted] (2.5,0) -- (3,0); 
\draw (3,0) -- (3.4,0); 
\draw (3.6,0.025) -- (4.4,0.5); 
\draw (3.6,-0.025) -- (4.35,-0.5); 
\node [left] at (3.5,-0.5) {$\tau_{l}|_\mathfrak{z} = - id, l \ge 3.$}; 
\end{tikzpicture} 

\end{tabular} & 

\begin{tabular}{c} 

can not be extended.

\end{tabular}  \\  

\hline

\begin{tabular}{c}

\begin{tikzpicture}
\draw (1,0) circle [radius = 0.1];
\draw (2,0) circle [radius = 0.1]; 
\draw (3.5,0) circle [radius = 0.1]; 
\draw (4.5,0.5) circle [radius = 0.1];
\draw (4.5,0.5) circle [radius = 0.15];
\draw (4.5,-0.5) circle [radius = 0.1];

\node [above] at (1.05,0.05) {$\phi_2$}; 
\node [above] at (2.05,0.05) {$\phi_3$}; 
\node [above] at (3.4,0.05) {$\phi_{l-2}$}; 
\node [right] at (4.5,0.5) {$\phi_{l-1}$}; 
\node [right] at (4.5,-0.5) {$\phi_l $}; 
\node [left] at (0.5,0) {$\tau_{l-1}|_{[\mathfrak{k},\mathfrak{k}]}: $}; 

\draw (1.1,0) -- (1.9,0);
\draw (2.1,0) -- (2.5,0); 
\draw [dotted] (2.5,0) -- (3,0); 
\draw (3,0) -- (3.4,0);
\draw (3.5,0) circle [radius = 0.1];
\draw (3.6,0.025) -- (4.35,0.45); 
\draw (3.6,-0.025) -- (4.4,-0.5); 
\node [left] at (3.5,-0.5) {$\tau_{{l-1}}|_\mathfrak{z} = - id, l \ge 3.$}; 
\end{tikzpicture} 

\end{tabular} & 

\begin{tabular}{c} 

can not be extended.

\end{tabular}  \\

\hline

\begin{tabular}{c}   

\begin{tikzpicture}
\draw (1,0.4) circle [radius = 0.1]; 
\draw (1,0.4) circle [radius = 0.15]; 
\draw (1,-0.4) circle [radius = 0.1]; 
\draw (1,-0.4) circle [radius = 0.15]; 
\node [right] at (1.1,0.4) {$\phi_2$}; 
\node [right] at (1.1,-0.4) {$\phi_3$};
\node [left] at (0,0) {$\sigma_0 :$};
\node [below] at (-0.5,-0.05) {$(l=3)$};
\node [left] at (1.5,-0.8) {$\sigma_{0}|_\mathfrak{z} = id.$}; 
\end{tikzpicture}  
\end{tabular} & 

\begin{tabular}{c}  

\begin{tikzpicture}
\filldraw [black] (0,0) circle [radius = 0.1]; 
\draw (1,0.5) circle [radius = 0.1]; 
\draw (1,-0.5) circle [radius = 0.1]; 
\draw (1,0.5) circle [radius = 0.15]; 
\draw (1,-0.5) circle [radius = 0.15];
\filldraw[black] (2,0) circle [radius = 0.1];

\node [right] at (1.05,.6) {$\phi_2 $};
\node [right] at (1.05,-.55) {$\phi_3 $};
\node [left] at (0,0) {$\phi_1 $};
\node [right] at (2,0) {$-\delta $};
\node [left] at (-0.5,0) {$\sigma_0: $};

\draw (0.1,0) -- (.85,0.5); 
\draw (1.15,0.5) -- (1.9,0);
\draw (0.1,0) -- (.85,-0.5);
\draw (1.15,-0.5) -- (2,0); 
\end{tikzpicture}  
 
\begin{tikzpicture}
\filldraw [black] (0,0) circle [radius = 0.1]; 
\draw (1,0.5) circle [radius = 0.1]; 
\draw (1,-.5) circle [radius = 0.1]; 
\draw (1,0.5) circle [radius = 0.15]; 
\draw (1,-.5) circle [radius = 0.15];
\draw (0,0) circle [radius = 0.15];
\filldraw[black] (2,0) circle [radius = 0.1];
\draw (2,0) circle [radius = 0.15];

\node [right] at (1.05,.6) {$\phi_2 $};
\node [right] at (1.05,-.55) {$\phi_3 $};
\node [left] at (0,0) {$\phi_1 $};
\node [right] at (2,0) {$-\delta $};
\node [left] at (-.6,0) {$\sigma_0\theta: $};

\draw (0.15,0) -- (.85,0.5); 
\draw (1.15,0.5) -- (1.85,0);
\draw (0.15,0) -- (.85,-0.5);
\draw (1.15,-0.5) -- (1.85,0); 
\end{tikzpicture}     \\

\end{tabular} \\ 
 
\hline
 
\begin{tabular}{c} 
  
\begin{tikzpicture}
\draw (1,0.4) circle [radius = 0.1]; 
\draw (1,0.4) circle [radius = 0.15]; 
\draw (1,-0.4) circle [radius = 0.1]; 
\draw (1,-0.4) circle [radius = 0.15]; 
\node [right] at (1.1,0.4) {$\phi_2$}; 
\node [right] at (1.1,-0.4) {$\phi_3$};
\node [left] at (0,0) {$\sigma_0' :$};
\node [below] at (-0.5,-0.05) {$(l=3)$};
\node [left] at (1.5,-1) {$\sigma'_{0}|_\mathfrak{z} = - id.$}; 
\end{tikzpicture}
 
\end{tabular} & 

\begin{tabular}{c}  

\begin{tikzpicture}
\filldraw [black] (0,0) circle [radius = 0.1]; 
\draw (1,0.5) circle [radius = 0.1]; 
\draw (1,-0.5) circle [radius = 0.1]; 
\draw (1,0.5) circle [radius = 0.15]; 
\draw (1,-0.5) circle [radius = 0.15];
\filldraw[black] (2,0) circle [radius = 0.1];

\node [right] at (1.05,0.6) {$\phi_2 $};
\node [right] at (1.05,-0.55) {$\phi_3 $};
\node [left] at (0,0) {$\phi_1 $};
\node [right] at (2,0) {$-\delta $};
\node [left] at (-.5,0) {$\sigma_0' , \sigma_0'\theta: $};

\draw (0.1,0) -- (.85,0.5); 
\draw (1.15,0.5) -- (1.9,0);
\draw (0.1,0) -- (.85,-0.5);
\draw (1.15,-0.5) -- (1.9,0); 
\draw (0.1,0) -- (1.9,0); 
\draw (0.1,0) -- (.35,0.1); 
\draw (0.1,0) -- (.35,-0.1);
\draw (1.65,0.1) -- (1.9,0);
\draw (1.65,-0.1) -- (1.9,0);
\end{tikzpicture}  
 
\end{tabular} \\  
\hline 

\end{tabular}

\end{table} 

\begin{table} 

\begin{tabular}{||c|c||}
\hline
\begin{tabular}{c}   
$l=2; \sigma_1|_{\mathfrak{z}_1} = id, \sigma_1|_{\mathfrak{z}_2} = id.$ 
\end{tabular} & 
\begin{tabular}{c} 

\begin{tikzpicture} 
\filldraw [black] (0,0.5) circle [radius = 0.15]; 
\draw (0,0.5) circle [radius = 0.22]; 
\filldraw [black] (0,-0.5) circle [radius = 0.15]; 
\draw [black] (0,-0.5) circle [radius = 0.22];
\filldraw [black] (1.5,0.5) circle [radius = 0.15];   
\filldraw [black] (1.5,-0.5) circle [radius = 0.15]; 

\node [left] at (-0.05,0.5) {$-\phi_1 $};
\node [left] at (-0.05,-0.5) {$\phi_1$};
\node [right] at (1.5,.5) {$-\phi_2 $};
\node [right] at (1.5,-.5) {$\phi_2 $};
\node [left] at (-.3,0) {$\sigma_1 :$};

\draw (0.04,0.3) -- (0.04,-0.3);
\draw (-0.04,0.3) -- (-0.04,-0.3);
\draw (-.12,0.3) -- (-.12,-0.3);
\draw (0.12,0.3) -- (0.12,-0.3);

\draw (1.37,0.5) -- (1.37,-0.5);
\draw (1.45,0.5) -- (1.45,-0.5);
\draw (1.54,0.5) -- (1.54,-0.5);
\draw (1.63,0.5) -- (1.63,-0.5);
\end{tikzpicture} 
 
\begin{tikzpicture}
\filldraw [black] (0,0.5) circle [radius = 0.15]; 
\filldraw [black] (0,-0.5) circle [radius = 0.15]; 
\filldraw [black] (1.5,0.5) circle [radius = 0.15]; 
\draw (1.5,0.5) circle [radius = 0.22]; 
\filldraw [black] (1.5,-0.5) circle [radius = 0.15]; 
\draw [black] (1.5,-0.5) circle [radius = 0.22];

\node [left] at (0,0.5) {$-\phi_1 $};
\node [left] at (0,-.5) {$\phi_1$};
\node [right] at (1.6,.5) {$-\phi_2 $};
\node [right] at (1.6,-.5) {$\phi_2 $};
\node [left] at (-.3,0) {$\sigma_1\theta :$};

\draw (0.04,0.5) -- (0.04,-0.5);
\draw (-0.04,0.5) -- (-0.04,-0.5);
\draw (-.12,0.5) -- (-.12,-0.5);
\draw (0.12,0.5) -- (0.12,-0.5);

\draw (1.37,0.3) -- (1.37,-0.3);
\draw (1.45,0.3) -- (1.45,-0.3);
\draw (1.54,0.3) -- (1.54,-0.3);
\draw (1.63,0.3) -- (1.63,-0.3);
\end{tikzpicture} 
 
\end{tabular}  \\  
\hline
\begin{tabular}{c} 
$l=2; \eta_1|_{\mathfrak{z}_1} = -id, \eta_1|_{\mathfrak{z}_2} = id.$ 
\end{tabular} & 
\begin{tabular}{c} 

\begin{tikzpicture}
\filldraw [black] (0,0.5) circle [radius = 0.15]; 
\filldraw [black] (0,-0.5) circle [radius = 0.15]; 
\filldraw [black] (1.5,0.5) circle [radius = 0.15];   
 \filldraw [black] (1.5,-0.5) circle [radius = 0.15]; 

\node [left] at (-0.1,0.5) {$-\phi_1 $};
\node [left] at (-0.1,-.5) {$\phi_1$};
\node [right] at (1.5,.5) {$-\phi_2 $};
\node [right] at (1.5,-.5) {$\phi_2 $};
\node [left] at (-.3,0) {$\eta_1 :$};

\draw (0.04,0.5) -- (0.04,-0.5);
\draw (-0.04,0.5) -- (-0.04,-0.5);
\draw (-.12,0.5) -- (-.12,-0.5);
\draw (0.12,0.5) -- (0.12,-0.5);

\draw (1.37,0.5) -- (1.37,-0.5);
\draw (1.45,0.5) -- (1.45,-0.5);
\draw (1.54,0.5) -- (1.54,-0.5);
\draw (1.63,0.5) -- (1.63,-0.5);

\draw (0,-0.5) to[out=180, in=180] (0,0.5);

\draw (-0.1,0.5) -- (-0.3,0.39);
\draw (-0.1,0.5) -- (-0.2,0.2);
\draw (-0.1,-0.5) -- (-.2, -.2);
\draw (-0.1,-0.5) -- (-.3, -.4);
\end{tikzpicture}

\begin{tikzpicture}
\filldraw [black] (0,0.5) circle [radius = 0.15]; 
\filldraw [black] (0,-0.5) circle [radius = 0.15]; 
\filldraw [black] (1.5,0.5) circle [radius = 0.15]; 
\draw (1.5,0.5) circle [radius = 0.22]; 
\filldraw [black] (1.5,-0.5) circle [radius = 0.15]; 
\draw [black] (1.5,-0.5) circle [radius = 0.22];

\node [left] at (-0.1,0.5) {$-\phi_1 $};
\node [left] at (-0.1,-.5) {$\phi_1$};
\node [right] at (1.6,.5) {$-\phi_2 $};
\node [right] at (1.6,-.5) {$\phi_2 $};
\node [left] at (-.3,0) {$\eta_1\theta :$};

\draw (0.04,0.5) -- (0.04,-0.5);
\draw (-0.04,0.5) -- (-0.04,-0.5);
\draw (-.12,0.5) -- (-.12,-0.5);
\draw (0.12,0.5) -- (0.12,-0.5);

\draw (1.37,0.3) -- (1.37,-0.3);
\draw (1.45,0.3) -- (1.45,-0.3);
\draw (1.54,0.3) -- (1.54,-0.3);
\draw (1.63,0.3) -- (1.63,-0.3);
 
\draw (0,-0.5) to[out=180, in=180] (0,0.5);

\draw (-0.1,0.5) -- (-0.3,0.39);
\draw (-0.1,0.5) -- (-0.2,0.2);
\draw (-0.1,-0.5) -- (-.2, -.2); 
\draw (-0.1,-0.5) -- (-.3, -.4);
\end{tikzpicture} 

\end{tabular}  \\
\hline
\begin{tabular}{c} 
$l=2 ; \eta_2|_{\mathfrak{z}_1} = id, \eta_2|_{\mathfrak{z}_2} = -id.$ 
\end{tabular} & 
\begin{tabular}{c} 

\begin{tikzpicture}
\filldraw [black] (0,0.5) circle [radius = 0.15]; 
\filldraw [black] (0,-0.5) circle [radius = 0.15]; 
\filldraw [black] (1.5,0.5) circle [radius = 0.15];   
\filldraw [black] (1.5,-0.5) circle [radius = 0.15]; 

\node [left] at (0,0.5) {$-\phi_1 $};
\node [left] at (0,-.5) {$\phi_1$};
\node [right] at (1.6,.5) {$-\phi_2 $};
\node [right] at (1.6,-.5) {$\phi_2 $};
\node [left] at (-.3,0) {$\eta_2 :$};

\draw (0.04,0.5) -- (0.04,-0.5);
\draw (-0.04,0.5) -- (-0.04,-0.5);
\draw (-.12,0.5) -- (-.12,-0.5);
\draw (0.12,0.5) -- (0.12,-0.5);

\draw (1.37,0.5) -- (1.37,-0.5);
\draw (1.45,0.5) -- (1.45,-0.5);
\draw (1.54,0.5) -- (1.54,-0.5);
\draw (1.63,0.5) -- (1.63,-0.5);

\draw (1.5,-0.55) to[out=360, in=360] (1.5,0.55);

\draw (1.65,0.52)-- (1.9,.33);
\draw (1.65,0.52)-- (1.72,.23);
\draw (1.65,-0.52)-- (1.9,-.33);
\draw (1.65,-0.52)-- (1.72,-.23);
\end{tikzpicture}
 
\begin{tikzpicture}
\filldraw [black] (0,0.5) circle [radius = 0.15]; 
\draw (0,0.5) circle [radius = 0.22]; 
\filldraw [black] (0,-0.5) circle [radius = 0.15]; 
\draw [black] (0,-0.5) circle [radius = 0.22];
\filldraw [black] (1.5,0.5) circle [radius = 0.15];   
\filldraw [black] (1.5,-0.5) circle [radius = 0.15]; 

\node [left] at (-0.1,0.5) {$-\phi_1 $};
\node [left] at (-0.1,-.5) {$\phi_1$};
\node [right] at (1.6,.5) {$-\phi_2 $};
\node [right] at (1.6,-.5) {$\phi_2 $};
\node [left] at (-.3,0) {$\eta_2\theta :$};

\draw (0.04,0.3) -- (0.04,-0.3);
\draw (-0.04,0.3) -- (-0.04,-0.3);
\draw (-.12,0.3) -- (-.12,-0.3);
\draw (0.12,0.3) -- (0.12,-0.3);

\draw (1.37,0.5) -- (1.37,-0.5);
\draw (1.45,0.5) -- (1.45,-0.5);
\draw (1.54,0.5) -- (1.54,-0.5);
\draw (1.63,0.5) -- (1.63,-0.5);

\draw (1.5,-0.55) to[out=360, in=360] (1.5,0.55);

\draw (1.65,0.52)-- (1.9,.33);
\draw (1.65,0.52)-- (1.72,.23);
\draw (1.65,-0.52)-- (1.9,-.33);
\draw (1.65,-0.52)-- (1.72,-.23);
\end{tikzpicture} 

\end{tabular}  \\
\hline
\begin{tabular}{c} 
$l=2 ; \mu_1|_{\mathfrak{z}_1} = -id, \mu_1|_{\mathfrak{z}_2} = -id.$ 
\end{tabular} & 
\begin{tabular}{c} 

\begin{tikzpicture}
\filldraw [black] (0,0.5) circle [radius = 0.15]; 
\filldraw [black] (0,-0.5) circle [radius = 0.15]; 
\filldraw [black] (1.5,0.5) circle [radius = 0.15];   
\filldraw [black] (1.5,-0.5) circle [radius = 0.15]; 

\node [left] at (-0.1,0.5) {$-\phi_1 $};
\node [left] at (-0.1,-.5) {$\phi_1$};
\node [right] at (1.6,.5) {$-\phi_2 $};
\node [right] at (1.6,-.5) {$\phi_2 $};
\node [left] at (-.3,0) {$\mu_1 , \mu_1\theta :$};

\draw (0.04,0.5) -- (0.04,-0.5);
\draw (-0.04,0.5) -- (-0.04,-0.5);
\draw (-.12,0.5) -- (-.12,-0.5);
\draw (0.12,0.5) -- (0.12,-0.5);

\draw (1.37,0.5) -- (1.37,-0.5);
\draw (1.45,0.5) -- (1.45,-0.5);
\draw (1.54,0.5) -- (1.54,-0.5);
\draw (1.63,0.5) -- (1.63,-0.5);

\draw (0,-0.5) to[out=180, in=180] (0,0.5);

\draw (-0.1,0.5) -- (-0.3,0.39);
\draw (-0.1,0.5) -- (-0.2,0.2);
\draw (-0.1,-0.5) -- (-.2, -.2);
\draw (-0.1,-0.5) -- (-.3, -.4);
    
\draw (1.5,-0.55) to[out=360, in=360] (1.5,0.55);

\draw (1.65,0.52)-- (1.9,.33);
\draw (1.65,0.52)-- (1.72,.23);
\draw (1.65,-0.52)-- (1.9,-.33);
\draw (1.65,-0.52)-- (1.72,-.23);
\end{tikzpicture}

\end{tabular}  \\
\hline
\begin{tabular}{c}
$l=2 ; \tau'_1(H_{\omega_1})=H_{\omega_2}, \tau'_1(H_{\omega_2})=H_{\omega_1}.$
\end{tabular} & 
\begin{tabular}{c} 

\begin{tikzpicture}
\filldraw [black] (0,0.5) circle [radius = 0.15]; 
\filldraw [black] (0,-0.5) circle [radius = 0.15]; 
\filldraw [black] (1.5,0.5) circle [radius = 0.15];   
\filldraw [black] (1.5,-0.5) circle [radius = 0.15]; 

\node [left] at (0,0.5) {$-\phi_1 $};
\node [left] at (0,-.5) {$\phi_1$};
\node [right] at (1.5,.5) {$-\phi_2 $};
\node [right] at (1.5,-.5) {$\phi_2 $};
\node [left] at (-.3,0) {$\tau'_1, \tau'_1\theta :$};

\draw (0.04,0.5) -- (0.04,-0.5);
\draw (-0.04,0.5) -- (-0.04,-0.5);
\draw (-.12,0.5) -- (-.12,-0.5);
\draw (0.12,0.5) -- (0.12,-0.5);

\draw (1.37,0.5) -- (1.37,-0.5);
\draw (1.45,0.5) -- (1.45,-0.5);
\draw (1.54,0.5) -- (1.54,-0.5);
\draw (1.63,0.5) -- (1.63,-0.5);

\draw  (0.04,0.5) -- (1.37,0.5);
\draw (-0.04,-0.5) -- (1.37,-0.5);

\draw (0.08,0.5) -- (.25, 0.6);
\draw (0.08,0.5) -- (.25, 0.4);
\draw (1.42,0.5) -- (1.25,0.6);
\draw (1.42,0.5) -- (1.25,0.4);

\draw (0.08,-0.5) -- (.25,-0.4);
\draw (0.08,-0.5) -- (.25,-0.6);
\draw (1.42,-.5) -- (1.25,-.4);
\draw (1.42,-.5) -- (1.25,-.6);
\end{tikzpicture}

\end{tabular}  \\
\hline
\begin{tabular}{c}  
$l=2 ; \tau_1(H_{\omega_1})= -H_{\omega_2}, \tau_1(H_{\omega_2})= - H_{\omega_1}.$ 
\end{tabular} & 
\begin{tabular}{c} 

\begin{tikzpicture}
\filldraw [black] (0,0.5) circle [radius = 0.15]; 
\filldraw [black] (0,-0.5) circle [radius = 0.15]; 
\filldraw [black] (1.5,0.5) circle [radius = 0.15];   
\filldraw [black] (1.5,-0.5) circle [radius = 0.15]; 

\node [left] at (0,0.5) {$-\phi_1 $};
\node [left] at (0,-.5) {$\phi_1$};
\node [right] at (1.5,.5) {$-\phi_2 $};
\node [right] at (1.5,-.5) {$\phi_2 $};
\node [left] at (-.3,0) {$\tau_1, \tau_1\theta :$};

\draw (0.04,0.5) -- (0.04,-0.5);
\draw (-0.04,0.5) -- (-0.04,-0.5);
\draw (-.12,0.5) -- (-.12,-0.5);
\draw (0.12,0.5) -- (0.12,-0.5);

\draw (1.37,0.5) -- (1.37,-0.5);
\draw (1.45,0.5) -- (1.45,-0.5);
\draw (1.54,0.5) -- (1.54,-0.5);
\draw (1.63,0.5) -- (1.63,-0.5);

\draw (0,.55) -- (1.5,-0.55);
\draw (0,-.55) -- (1.5,0.55);

\draw (.1,.48)-- (.2,.25);
\draw (.1,.48)-- (.3,.48);
\draw (1.4,-.45) -- ( 1.2,-.5);
\draw (1.4,-.45) -- ( 1.3,-.25);

\draw (0.1,-.45) -- (.2,-.25);
\draw (0.1,-.45) -- (.3,-.45);
\draw (1.4,.45) -- (1.2,.5);
\draw (1.4,.45) -- (1.3,.25);
\end{tikzpicture}

\end{tabular}  \\
\hline
\end{tabular}

\end{table}

Let $\omega_1, \omega_2, \ldots , \omega_l$ be the fundamental weights associated with $\phi_1, \phi_2, \ldots , \phi_l$ respectively 
relative to the set of simple roots $\Phi = \{\phi_1, \phi_2, \ldots , \phi_l\}.$ If $m = 2$ that is, $\frak{g}_0 = \frak{so}(2,2),$ then $\frak{k} = \frak{k}_1 \oplus \frak{k}_2,$ 
where $\frak{k}_i = \mathbb{C}H_{\omega_i} \cong \frak{so}(2,\mathbb{C}) (i = 1,2)$ are abelian ideals in $\frak{k}.$ Clearly $\frak{z} = \frak{z}_1 \oplus \frak{z}_2,$ where 
$\frak{z}_i = \frak{k}_i,$ the centre of $\frak{k}_i (i = 1,2).$ 
Now assume that $m \ge 3.$ Then $\frak{z} = \mathbb{C}(H_{\omega_1}).$ If $\frak{g}_0 = \frak{so}(2, 2l-1),$ then $\omega_1 = \gamma,$ where 
$\gamma = \phi_1 + \phi_2 + \cdots + \phi_l,$ a non-compact root. If $\frak{g}_0 = \frak{so}(2, 2l-2),$ then $2\omega_1 = \gamma_1 + \gamma_2,$ where 
$\gamma_1 =  \phi_1 + \phi_2 + \cdots + \phi_{l-1}, \gamma_2 =  \phi_1 + \phi_2 + \cdots +\phi_{l-2} +\phi_l.$ The roots $\gamma_1$ and $\gamma_2$ are non-compact and 
$\gamma_1 \pm \gamma_2 \notin \Delta.$ Recall that for $\alpha \in \Delta^+, X_\alpha = E_\alpha - E_{-\alpha}, Y_\alpha = i(E_\alpha + E_{-\alpha}).$ 
Consider the Cayley transform \[c = 
\begin{cases} 
\textrm{Ad}(\textrm{exp}(\frac{\pi}{4} X_\gamma)) \textrm{ if } \frak{g}_0 = \frak{so}(2, 2l-1), \\
\textrm{Ad}(\textrm{exp}(\frac{\pi}{4}(X_{\gamma_1}+X_{\gamma_2})))  \textrm{ if } \frak{g}_0 = \frak{so}(2, 2l-2). 
\end{cases}
\] Then $c(Y_{\gamma}) = iH_{\gamma}^*$ and $c|_{Ker(\gamma)} = id,$ if $\frak{g}_0 = \frak{so}(2, 2l-1).$ 
If $\frak{g}_0 = \frak{so}(2, 2l-2),$ $c(Y_{\gamma_j}) = iH_{\gamma_j}^* (j = 1,2)$ and $c|_{Ker(\gamma_1)\cap Ker(\gamma_2)} = id.$
Let $\sigma \in \textrm{inv}(\frak{g}_0)$ be such that $\sigma|_\frak{z} = id$ and \\ 
$\sigma|_{\frak{g}^\gamma} = -id$ if $\frak{g}_0 = \frak{so}(2, 2l-1),$ \\
$\sigma|_{\frak{g}^{\gamma_j}} = -id (j = 1,2)$ if $\frak{g}_0 = \frak{so}(2, 2l-2).$ \\ 
Then $c\sigma c^{-1}(H^*_{\omega_1}) = -H^*_{\omega_1}$ that is, $c\sigma c^{-1}|_\frak{z} = - id.$ Hence $c\sigma c^{-1}(H_{\phi_1}) = H_{-\delta}$ and 
$c\sigma c^{-1}$ is an involution. Keeping this in mind, we are ready to determine a representative for each involution of $\frak{g}_0$ whose almost double Vogan diagram 
is given in Tables \ref{inv1} and \ref{inv2}. We have the following proposition: 

\begin{proposition}\label{prep}
For each almost double Vogan diagram $\sigma$ given in Table \ref{inv1}or \ref{inv2}, there is an involution of $G$ commuting with $\theta$ such that 
the differential at identity is in $\Gamma \textrm{ or } \Gamma',$ and has the almost double Vogan diagram $\sigma.$ 
\end{proposition} 

To prove this proposition, it is sufficient to prove that there is an involution of $G$ commuting with $\theta$ whose differential at identity has 
the almost double Vogan diagram $\sigma$ and for some simple roots $\phi$ and $\psi,$ if $\sigma(\phi) = \psi,$ then $\sigma (E_\phi) = p_\phi E_\psi$ where $p_\phi = \pm1.$ 
We have proved this below. 

\subsection{$G = SO_0(2,2l-1)\ (l \ge 1)$:} Then $\frak{g}_0 = \frak{so}(2, 2l-1).$ First assume that $l \ge 2.$ For $2 \le p \le l,$ let $\sigma_p(X) = 
I_{2p,2l+1-2p} X I_{2p,2l+1-2p}$ for all $X \in M_{2l+1 \times 2l+1}(\mathbb{C}).$ Then $\sigma_p : G \longrightarrow G$ is an involution commuting with $\theta$ and 
the differential of $\sigma_p$ at the identity is $\sigma_p : \frak{g}_0 \longrightarrow \frak{g}_0.$ 
If $\frak{t}_0 =\sum_{1 \le j \le l} \mathbb{R}H_j$ as in \S \ref{lattice}, then $\sigma_p|_{\frak{t}_0} = id.$ 
We have $\sigma_p(\phi_j) = \phi_j$ for all $1\le j \le l,$ and 
\[\sigma_p|_{\frak{g}^{\phi_j}} = 
\begin{cases}
id \textrm{ for } j \neq p, \\ 
-id \textrm{ for } j = p. 
\end{cases} \]
Thus $\sigma_p, \sigma_p \theta$ are involutions of $\frak{g}_0$ whose 
almost double Vogan diagrams are given in Table \ref{inv1} as $\sigma_p, \sigma_p\theta$ respectively. 

Let $\gamma =  \phi_1 + \phi_2 \cdots + \phi_l = e_1,\ X_\gamma = D_1^+- \overline{D_1^+} = -2i(E_{1,2l+1} + E_{2l+1,1})$ and for $2 \le p \le l,\ \tilde{\tau}_p = 
a_{\textrm{exp}(\frac{\pi}{4}X_\gamma)} \sigma_p a_{\textrm{exp}(-\frac{\pi}{4}X_\gamma)},$ where for $g \in M_{2l+1 \times 2l+1}(\mathbb{C}),\ a_g(X) = gXg^{-1}$ for 
all $X \in  M_{2l+1 \times 2l+1}(\mathbb{C}).$ So $\tilde{\tau}_p(X) = \textrm{exp}(\frac{\pi}{4}X_\gamma) \sigma_p(\textrm{exp}(-\frac{\pi}{4}X_\gamma) X 
\textrm{exp}(\frac{\pi}{4}X_\gamma)) \textrm{exp}(-\frac{\pi}{4}X_\gamma) = \textrm{exp}(\frac{\pi}{2}X_\gamma) I_{2p,2l+1-2p} X I_{2p,2l+1-2p} 
\textrm{exp} (-\frac{\pi}{2}X_\gamma).$ Since $\textrm{exp}(\frac{\pi}{2}X_\gamma) = I + (\cos \pi - 1)(E_{11} + E_{2l+1,2l+1}) -i \sin \pi (E_{1,2l+1} + E_{2l+1,1}) = 
I - 2(E_{11} + E_{2l+1,2l+1}),$ we have $\tilde{\tau}_p(X) =(I - 2(E_{11} + E_{2l+1,2l+1}))$ \\ $I_{2p,2l+1-2p} X I_{2p,2l+1-2p} (I - 2(E_{11} + E_{2l+1,2l+1})) = 
\textrm{diag}(1, -I_{2p-1}, I_{2l-2p}, -1) X \textrm{diag}(1, -I_{2p-1},$ \\ $I_{2l-2p}, -1).$ Thus for $2 \le p \le l-1$ and $2 \le j \le l,$ 
\[\tilde{\tau}_p|_{\frak{g}^{\phi_j}} = 
\begin{cases}
id \textrm{ for } j \neq p,l, \\ 
-id \textrm{ for } j = p,l; \textrm{ and }
\end{cases} \]
\[\tilde{\tau}_l|_{\frak{g}^{\phi_j}} = id \textrm{ for all } 2 \le j \le l. \]
 For $2 \le p \le l,$ define $\tau_p = \sigma_l \tilde{\tau}_p$ so that $\tau_p(X) = \textrm{diag}(-1, I_{2p-1}, -I_{2l-2p+1}) X \textrm{diag}(-1, I_{2p-1},$ \\ $-I_{2l-2p+1})$ for all 
$X \in  M_{2l+1 \times 2l+1}(\mathbb{C}).$ For $l \ge 1,$ define $\tau_1(X) = \textrm{diag}(-1, 1, -I_{2l-1}) X \textrm{diag}(-1,$ \\ $1,-I_{2l-1})$ for all 
$X \in  M_{2l+1 \times 2l+1}(\mathbb{C}).$
Then for $1 \le p \le l,$ $\tau_p (H_{\phi_1}) = H_{-\delta}, \tau_p(H_{\phi_j}) = H_{\phi_j},$ 
\[\tau_p|_{\frak{g}^{\phi_j}} = 
\begin{cases}
id \textrm{ if } j \neq p, \\ 
-id \textrm{ if } j = p, 
\end{cases} \]
where $2 \le j \le l;$ and $\tau_p(G_{12}^+) = -\overline{G_{12}^-}$ for $2 \le p \le l,\ \tau_1(G_{12}^+) = \overline{G_{12}^-},\ \tau_1(D_1^+) = -\overline{D_1^+}.$ 
Thus $\tau_p : G \longrightarrow G$ is an involution commuting with $\theta,$ 
the differential of $\tau_p$ at the identity is $\tau_p : \frak{g}_0 \longrightarrow \frak{g}_0,\ \tau_p, \tau_p \theta$ are involutions of $\frak{g}_0$ whose 
almost double Vogan diagram is given in Table \ref{inv1} as $\tau_p$ (or $\tau_p\theta).$ 

\subsection{$G = SO_0(2,2l-2)\ (l \ge 3)$:} Then $\frak{g}_0 = \frak{so}(2, 2l-2).$ First assume that $l \ge 4.$ For $2 \le p \le l-2,$ let $\sigma_p(X) = 
I_{2p,2l-2p} X I_{2p,2l-2p}$ for all $X \in M_{2l \times 2l}(\mathbb{C}).$ Then $\sigma_p : G \longrightarrow G$ is an involution commuting with $\theta$ and 
the differential of $\sigma_p$ at the identity is $\sigma_p : \frak{g}_0 \longrightarrow \frak{g}_0.$ 
If $\frak{t}_0 =\sum_{1 \le j \le l} \mathbb{R}H_j$ as in \S \ref{lattice}, then $\sigma_p|_{\frak{t}_0} = id.$ 
We have $\sigma_p(\phi_j) = \phi_j$ for all $1\le j \le l,$ and 
\[\sigma_p|_{\frak{g}^{\phi_j}} = 
\begin{cases}
id \textrm{ for } j \neq p, \\ 
-id \textrm{ for } j = p. 
\end{cases} \]
Thus $\sigma_p, \sigma_p \theta$ are involutions of $\frak{g}_0$ whose 
almost double Vogan diagrams are given in Table \ref{inv2} as $\sigma_p, \sigma_p\theta$ respectively. 

Let $\gamma_1 =  \phi_1 + \phi_2 + \cdots + \phi_{l-2} + \phi_{l-1} = e_1 - e_l,\  \gamma_2 =  \phi_1 + \phi_2 + \cdots + \phi_{l-2} + \phi_l = e_1 + e_l, 
\ X_{\gamma_1} = \frac{1}{2}(G_{1l}^+- \overline{G_{1l}^+}) = i(E_{1,2l-1} + E_{2,2l} + E_{2l-1,1} + E_{2l,2}),
\ X_{\gamma_2} = \frac{1}{2}(G_{1l}^- - \overline{G_{1l}^-}) = i(-E_{1,2l-1} + E_{2,2l} - E_{2l-1,1} + E_{2l,2})$
and for $2 \le p \le l-2,\ \mu_p = a_{\textrm{exp}(\frac{\pi}{4}(X_{\gamma_1} + X_{\gamma_2}))} \sigma_p a_{\textrm{exp}(-\frac{\pi}{4}(X_{\gamma_1} + X_{\gamma_2}))},$ 
where for $g \in M_{2l \times 2l}(\mathbb{C}),\ a_g(X) = gXg^{-1}$ for 
all $X \in  M_{2l \times 2l}(\mathbb{C}).$ So $\mu_p(X) = \textrm{exp}(\pi i(E_{2,2l} + E_{2l,2})) I_{2p,2l-2p} X I_{2p,2l-2p} 
\textrm{exp} (-\pi i(E_{2,2l} + E_{2l,2})).$ Since $\textrm{exp}(\pi i(E_{2,2l} + E_{2l,2})) = I + (\cos \pi - 1)(E_{22} + E_{2l,2l}) + i \sin \pi (E_{2,2l} + E_{2l,2}) = 
I - 2(E_{22} + E_{2l,2l}),$ we have $\mu_p(X) = \textrm{diag}(-1,1, -I_{2p-2},$ \\ $I_{2l-2p-1}, -1) X \textrm{diag}(-1,1, -I_{2p-2},I_{2l-2p-1}, -1)$ for all 
$X \in  M_{2l \times 2l}(\mathbb{C}).$ For $l \ge 3,$ define $\mu_1(X) = \textrm{diag}(-1,I_{2l-2},-1) X \textrm{diag}(-1,I_{2l-2},-1)$ for all 
$X \in  M_{2l \times 2l}(\mathbb{C}).$
Then for $1 \le p \le l-2,$ $\mu_p (H_{\phi_1}) = H_{-\delta}, \mu_p(H_{\phi_{l-1}}) = H_{\phi_l}, \mu_p(H_{\phi_j}) = H_{\phi_j},$ 
\[\mu_p|_{\frak{g}^{\phi_j}} = 
\begin{cases}
id \textrm{ if } j \neq p, \\ 
-id \textrm{ if } j = p, 
\end{cases} \]
where $2 \le j \le l-2;$ and $\mu_p(G_{12}^+) = \overline{G_{12}^-},\ \mu_p(G_{(l-1)l}^+) = -G_{(l-1)l}^-$ for $2 \le p \le l-2,\ \mu_1(G_{12}^+) = -\overline{G_{12}^-},\ 
\mu_1(G_{(l-1)l}^+) = -G_{(l-1)l}^-.$ Thus $\mu_p : G \longrightarrow G$ is an involution commuting with $\theta,$ 
the differential of $\mu_p$ at the identity is $\mu_p : \frak{g}_0 \longrightarrow \frak{g}_0, \ \mu_p, \mu_p \theta$ are involutions of $\frak{g}_0$ whose 
almost double Vogan diagram is given in Table \ref{inv2} as $\mu_p$ (or $\mu_p\theta).$ 

Let $l \ge 3.$ For $1 \le p \le l-2,$ define $\tau_p(X) = \textrm{diag}(-1,I_{2p-1},-I_{2l-2p}) X \textrm{diag}(-1,I_{2p-1},-I_{2l-2p})$ for all 
$X \in M_{2l \times 2l}(\mathbb{C}).$ Then for $1 \le p \le l-2,$ $\tau_p (H_{\phi_1}) = H_{-\delta}, \tau_p(H_{\phi_j}) = H_{\phi_j},$ 
\[\tau_p|_{\frak{g}^{\phi_j}} = 
\begin{cases}
id \textrm{ if } j \neq p, \\ 
-id \textrm{ if } j = p,
\end{cases} \]
where $2 \le j \le l;$ and $\tau_p(G_{12}^+) = -\overline{G_{12}^-}$ for $2 \le p \le l,\ \tau_1(G_{12}^+) = \overline{G_{12}^-}.$ 
Thus $\tau_p : G \longrightarrow G$ is an involution commuting with $\theta,$ 
the differential of $\tau_p$ at the identity is $\tau_p : \frak{g}_0 \longrightarrow \frak{g}_0,\ \tau_p, \tau_p \theta$ are involutions of $\frak{g}_0$ whose 
almost double Vogan diagram is given in Table \ref{inv2} as $\tau_p$ (or $\tau_p\theta).$ 

Let $l \ge 3.$ For $1 \le p \le l-2,$ define $\tau'_p(X) = \textrm{diag}(-I_{2p}, I_{2l-2p-1},-1) X \textrm{diag}(-I_{2p}, I_{2l-2p-1},-1)$ for all 
$X \in M_{2l \times 2l}(\mathbb{C}).$ Then for $1 \le p \le l-2,$ $\tau'_p (H_{\phi_{l-1}}) = H_{\phi_l}, \tau'_p(H_{\phi_j}) = H_{\phi_j},$ 
\[\tau'_p|_{\frak{g}^{\phi_j}} = 
\begin{cases}
id \textrm{ if } j \neq p \textrm{ or } p = 1, \\ 
-id \textrm{ if } j = p \neq 1, 
\end{cases} \]
where $1 \le j \le l-2;$ and $\tau'_p(G_{(l-1)l}^+) = -G_{(l-1)l}^-.$ 
Thus $\tau'_p : G \longrightarrow G$ is an involution commuting with $\theta,$ 
the differential of $\tau'_p$ at the identity is $\tau'_p : \frak{g}_0 \longrightarrow \frak{g}_0,\ \tau'_p, \tau'_p \theta$ are involutions of $\frak{g}_0$ whose 
almost double Vogan diagrams are given in Table \ref{inv2} as $\tau'_p$ and $\tau'_p\theta$ respectively. 

Assume that $l \ge 3.$ For any two non-negative integers $p, q,$ define $J_{p,q} = \left(
\begin{array}{ccc}
0 & I_{p,q} \\
-I_{p,q} & 0 \\
\end{array}
\right) \in M_{2(p+q) \times 2(p+q)}(\mathbb{R}).$ Define $\sigma_l(X) = J_0 X J_0^{-1}$ for all $X \in M_{2l \times 2l}(\mathbb{C}),$ where $J_0 = \textrm{diag}(J_{0,1}, J_{0,l-1}).$ 
Then $\sigma_l : G \longrightarrow G$ is an involution commuting with $\theta$ and the differential of $\sigma_l$ at the identity is $\sigma_l : \frak{g}_0 \longrightarrow \frak{g}_0.$ 
We have $\frak{g}_0(\sigma_l) = \{X \in \frak{g}_0 : \sigma_l(X) = X \} = 
\{\left(
\begin{array}{ccccccc}
0 & c & X'_2 & X''_2 \\
-c & 0 & -X''_2 & X'_2 \\
 {X'_2}^t & -{X''_2}^t & A & B \\ 
 {X''_2}^t & {X'_2}^t & -B & A 
 \end{array}
\right): c \in \mathbb{R}, A \in \frak{so}(l-1), B \in M_{(l-1) \times (l-1)}(\mathbb{R}) \textrm{ is symmetric}, X'_2, X''_2 \in M_{1 \times (l-1)}(\mathbb{R})\},$ 
$\frak{g}_0(-\sigma_l) = 
\{\left(
\begin{array}{ccccccc}
0 & 0 & X'_2 & X''_2 \\
0 & 0 & X''_2 & -X'_2 \\
 {X'_2}^t & {X''_2}^t & A & B \\ 
 {X''_2}^t & -{X'_2}^t & B & -A 
 \end{array}
\right): A, B \in \frak{so}(l-1), X'_2, X''_2 \in M_{1 \times (l-1)}(\mathbb{R})\}.$ 
Consider the maximal abelian subalgebra $\frak{t}'_0$ of $\frak{k}_0$ defined in \S \ref{lattice}. Since $\sigma_l|_{\frak{t}'_0} = id,$ 
we have $\sigma_l(\phi'_j) = \phi'_j$ for all $1\le j \le l.$ Also
\[\sigma_l|_{\frak{g}^{\phi'_j}} = 
\begin{cases}
id \textrm{ for } j \neq 1,l, \\ 
-id \textrm{ for } j = 1,l. 
\end{cases} \]
Thus $\sigma_l, \sigma_l \theta$ are involutions of $\frak{g}_0$ whose 
almost double Vogan diagrams are given in Table \ref{inv2} as $\sigma_l, \sigma_l\theta$ respectively. 

Assume that $l \ge 3.$ Define $\sigma_{l-1}(X) = J'_0 X {J'}_0^{-1}$ for all $X \in M_{2l \times 2l}(\mathbb{C}),$ where $J'_0 = \textrm{diag}(J_{0,1}, J_{l-2,1}).$ 
Then $\sigma_{l-1} : G \longrightarrow G$ is an involution commuting with $\theta$ and the differential of $\sigma_{l-1}$ at the identity is $\sigma_{l-1} : \frak{g}_0 \longrightarrow \frak{g}_0.$ 
We have $\frak{g}_0(\sigma_{l-1}) = \{X \in \frak{g}_0 : \sigma_{l-1}(X) = X \} = 
\{\left(
\begin{array}{ccccccccccc}
0 & c & X'_2 & x'_2 & X''_2 & x''_2 \\
-c & 0 & -X''_2 & x''_2 & X'_2 & -x'_2 \\
 {X'_2}^t & -{X''_2}^t & A_1 & A_2 & B_1 & B_2 \\ 
x'_2 & x''_2 & -A_2^t & 0 & -B_2^t & b_3 \\
 {X''_2}^t & {X'_2}^t & -B_1 & B_2 & A_1 & -A_2 \\ 
x''_2 & -x'_2 & -B_2^t & -b_3 & A_2^t & 0 \\ 
 \end{array}
\right): c, x'_2, x''_2, b_3 \in \mathbb{R}, A_1 \in \frak{so}(l-2), B_1 \in M_{(l-2) \times (l-2)}(\mathbb{R}) \textrm{ is symmetric}, X'_2, X''_2,A_2^t, B_2^t \in M_{1 \times (l-2)}(\mathbb{R})\},$ 
$\frak{g}_0(-\sigma_{l-1}) = 
\{\left(
\begin{array}{ccccccccccc}
0 & 0 & X'_2 & x'_2 & X''_2 & x''_2 \\
0 & 0 & X''_2 & -x''_2 & -X'_2 & x'_2 \\
 {X'_2}^t & {X''_2}^t & A_1 & A_2 & B_1 & B_2 \\ 
x'_2 & -x''_2 & -A_2^t & 0 & B_2^t & 0 \\
 {X''_2}^t & -{X'_2}^t & B_1 & -B_2 & -A_1 & A_2 \\ 
x''_2 & x'_2 & -B_2^t & 0 & -A_2^t & 0 \\ 
 \end{array}
\right): x'_2, x''_2 \in \mathbb{R}, A_1, B_1 \in \frak{so}(l-2), X'_2, X''_2, A_2^t, B_2^t \in M_{1 \times (l-2)}(\mathbb{R})\}.$ 
Consider the maximal abelian subalgebra $\frak{t}'_0$ of $\frak{k}_0$ defined in \S \ref{lattice}. Since $\sigma_{l-1}|_{\frak{t}'_0} = id,$ 
we have $\sigma_{l-1}(\phi'_j) = \phi'_j$ for all $1\le j \le l.$ Also
\[\sigma_{l-1}|_{\frak{g}^{\phi'_j}} = 
\begin{cases}
id \textrm{ for } j \neq 1,l-1, \\ 
-id \textrm{ for } j = 1,l-1. 
\end{cases} \]
Thus $\sigma_{l-1}, \sigma_{l-1} \theta$ are involutions of $\frak{g}_0$ whose 
almost double Vogan diagrams are given in Table \ref{inv2} as $\sigma_{l-1}, \sigma_{l-1}\theta$ respectively. 

Assume that $l = 3.$ Define $\sigma_0 = \sigma_{l-1} \sigma_l.$ Since $\sigma_{l-1} \sigma_l = \sigma_l \sigma_{l-1},$ 
$\sigma_0 : G \longrightarrow G$ is an involution commuting with $\theta$ and the differential of $\sigma_0$ at the identity is $\sigma_0 : \frak{g}_0 \longrightarrow \frak{g}_0.$ 
Clearly $\sigma_0, \sigma_0 \theta$ are involutions of $\frak{g}_0$ whose 
almost double Vogan diagrams are given in Table \ref{inv2} as $\sigma_0, \sigma_0 \theta$ respectively. 

Let $l = 3.$ Define $\sigma'_0(X) = \textrm{diag}(-1,1,-1,1,-1,1) X \textrm{diag}(-1,1,-1,1,-1,1)$ for all 
$X \in M_{6 \times 6}(\mathbb{C}).$ Then $\sigma'_0 (H_{\phi'_1}) = H_{-\delta'}, \sigma'_0 (H_{\phi'_j}) = H_{\phi'_j},$ and 
\[{\sigma'_0}|_{\frak{g}^{\phi'_j}} = -id \textrm{ if } j = 2,3. \]
 Also $\sigma'_0({G'}_{12}^+) = \overline{{G'}_{12}^-}.$ 
Thus $\sigma'_0 : G \longrightarrow G$ is an involution commuting with $\theta,$ 
the differential of $\sigma'_0$ at the identity is $\sigma'_0 : \frak{g}_0 \longrightarrow \frak{g}_0,\ \sigma'_0, \sigma'_0 \theta$ are involutions of $\frak{g}_0$ whose 
almost double Vogan diagram is given in Table \ref{inv2} as $\sigma'_0$ (or $\sigma'_0\theta).$ 

\subsection{$G = SO_0(2,2)$:} Define $\sigma_1(X) = \textrm{diag}(-J_{0,1}, J_{0,1}) X \textrm{diag}(J_{0,1}, -J_{0,1})$ for all 
$X \in M_{4 \times 4}(\mathbb{C}).$ Then $\sigma_1 (H_{\phi_j}) = H_{\phi_j}$ for $j = 1,2,$ and 
${\sigma_1}|_{\frak{g}^{\pm \phi_1}} = -id,\ {\sigma_1}|_{\frak{g}^{\pm \phi_2}} = id.$ 
Thus $\sigma_1 : G \longrightarrow G$ is an involution commuting with $\theta,$ 
the differential of $\sigma_1$ at the identity is $\sigma_1 : \frak{g}_0 \longrightarrow \frak{g}_0,\ \sigma_1, \sigma_1 \theta$ are involutions of $\frak{g}_0$ whose 
almost double Vogan diagrams are given in Table \ref{inv2} as $\sigma_1$ and $\sigma_1\theta$ respectively. 

Define $\eta_1(X) = J'_2 X J'_2$ for all $X \in M_{4 \times 4}(\mathbb{C}),$ where $J'_2 = \left(
\begin{array}{ccc}
0 & I_2 \\
I_2 & 0 \\
\end{array}
\right).$ 
Then $\eta_1 (H_{\phi_1}) = -H_{\phi_1},\ \eta_1 (H_{\phi_2}) = H_{\phi_2}$ and 
$\eta_1(G_{12}^+) = -\overline{G_{12}^+},\ {\eta_1}|_{\frak{g}^{\pm \phi_2}} = id.$ 
Thus $\eta_1 : G \longrightarrow G$ is an involution commuting with $\theta,$ 
the differential of $\eta_1$ at the identity is $\eta_1 : \frak{g}_0 \longrightarrow \frak{g}_0,\ \eta_1, \eta_1 \theta$ are involutions of $\frak{g}_0$ whose 
almost double Vogan diagrams are given in Table \ref{inv2} as $\eta_1$ and $\eta_1\theta$ respectively. 

Define $\eta_2(X) = J'_{1,1} X J'_{1,1}$ for all $X \in M_{4 \times 4}(\mathbb{C}),$ where $J'_{1,1} = \left(
\begin{array}{ccc}
0 & I_{1,1} \\
I_{1,1} & 0 \\
\end{array}
\right).$ 
Then $\eta_2 (H_{\phi_1}) = H_{\phi_1},\ \eta_2 (H_{\phi_2}) = -H_{\phi_2}$ and 
${\eta_2}|_{\frak{g}^{\pm \phi_1}} = id,\ \eta_2(G_{12}^-) = -\overline{G_{12}^-}.$
Thus $\eta_2 : G \longrightarrow G$ is an involution commuting with $\theta,$ 
the differential of $\eta_2$ at the identity is $\eta_2 : \frak{g}_0 \longrightarrow \frak{g}_0,\ \eta_2, \eta_2 \theta$ are involutions of $\frak{g}_0$ whose 
almost double Vogan diagrams are given in Table \ref{inv2} as $\eta_2$ and $\eta_2\theta$ respectively. 

Define $\mu_1(X) = \textrm{diag}(-1,I_2,-1) X \textrm{diag}(-1,I_2,-1)$ for all 
$X \in M_{4 \times 4}(\mathbb{C}).$ Then $\mu_1 (H_{\phi_j}) = -H_{\phi_j}$ for $j = 1,2,$ and 
$\mu_1(G_{12}^+) = \overline{G_{12}^+},\ \mu_1(G_{12}^-) = \overline{G_{12}^-}.$ 
Thus $\mu_1 : G \longrightarrow G$ is an involution commuting with $\theta,$ 
the differential of $\mu_1$ at the identity is $\mu_1 : \frak{g}_0 \longrightarrow \frak{g}_0,\ \mu_1, \mu_1 \theta$ are involutions of $\frak{g}_0$ whose 
almost double Vogan diagram is given in Table \ref{inv2} as $\mu_1$ (or $\mu_1\theta$). 

Define $\tau'_1(X) = \textrm{diag}(-I_2,1,-1) X \textrm{diag}(-I_2,1,-1)$ for all 
$X \in M_{4 \times 4}(\mathbb{C}).$ Then $\tau'_1 (H_{\phi_1}) = H_{\phi_2},\ \tau'_1 (H_{\phi_2}) = H_{\phi_1}$ and 
$\tau'_1(G_{12}^+) = G_{12}^-,\ \tau'_1(G_{12}^-) = G_{12}^+.$ 
Thus $\tau'_1 : G \longrightarrow G$ is an involution commuting with $\theta,$ 
the differential of $\tau'_1$ at the identity is $\tau'_1 : \frak{g}_0 \longrightarrow \frak{g}_0,\ \tau'_1, \tau'_1 \theta$ are involutions of $\frak{g}_0$ whose 
almost double Vogan diagram is given in Table \ref{inv2} as $\tau'_1$ (or $\tau'_1\theta$). 

Define $\tau_1(X) = \textrm{diag}(-1,1,-I_2) X \textrm{diag}(-1,1,-I_2)$ for all 
$X \in M_{4 \times 4}(\mathbb{C}).$ Then $\tau_1 (H_{\phi_1}) = -H_{\phi_2},\ \tau_1 (H_{\phi_2}) = -H_{\phi_1}$ and 
$\tau_1(G_{12}^+) = \overline{G_{12}^-},\ \tau_1(G_{12}^-) = \overline{G_{12}^+}.$ 
Thus $\tau_1 : G \longrightarrow G$ is an involution commuting with $\theta,$ 
the differential of $\tau_1$ at the identity is $\tau_1 : \frak{g}_0 \longrightarrow \frak{g}_0,\ \tau_1, \tau_1 \theta$ are involutions of $\frak{g}_0$ whose 
almost double Vogan diagram is given in Table \ref{inv2} as $\tau_1$ (or $\tau_1\theta$). 

\noindent 
\section{Orientation preserving action of some reductive subgroups of $SO_0(2,m)$}\label{or} 

Let $\sigma (\neq \theta)$ be an involution of $G$ commuting with $\theta$ and the differential of $\sigma$ at identity is an involution of $\frak{g}_0,$ 
denoted by the same notation $\sigma.$ Define $G(\sigma) = \{g \in G : \sigma(g) = g\}, K(\sigma) = K \cap G(\sigma), \frak{g}_0(\sigma) = \{X \in \frak{g}_0 : \sigma (X) = X \}, 
\frak{k}_0(\sigma) = \frak{k}_0 \cap \frak{g}_0(\sigma)$ and $\frak{p}_0(\sigma) = \frak{p}_0 \cap \frak{g}_0(\sigma).$ Then $G(\sigma)$ is a closed reductive subgroup of $G$ 
with Lie algebra $\frak{g}_0(\sigma), K(\sigma)$ is a maximal compact subgroup of $G(\sigma)$ and $X(\sigma) = G(\sigma)/K(\sigma)$ is a Riemannian globally symmetric 
space of non-compact type. We need to know that for which involutions $\sigma,$ the canonical action of $G(\sigma)$ is orientation preserving on $X(\sigma).$ 
Since $G(\sigma) = K(\sigma)\textrm{exp}(\frak{p}_0(\sigma)),$ it is sufficient to check whether the canonical action of $K(\sigma)$ on $X(\sigma)$ is 
orientation preserving. If $\it o = K(\sigma)$ is the identity coset in $X(\sigma)$, then $K(\sigma)(\it o) = \it o$ and the differential  
of this action is given by $\textrm{Ad} :  K(\sigma) \longrightarrow T_{\it o}(X(\sigma))$. Hence it is sufficient to check whether 
\begin{equation}\label{oreq}
\textrm{det}(\textrm{Ad}(k)|_{\frak{p}_0(\sigma)}) = 1 \textrm{ for all } k \in  K(\sigma).  
\end{equation}
Since $\frak{p}_0(\sigma\theta) = \frak{p}_0(-\sigma)$ and 
det$(\textrm{Ad}(k)|_{\frak{p}_0}) = 1$ for all $k \in  K,$ we have condition \ref{oreq} holds for $\sigma$ {\it iff} it holds for $\sigma \theta.$ 
That is, $G(\sigma)$ acts orientation preservingly on $X(\sigma)$ {\it iff} 
 $G(\sigma \theta)$ acts orientation preservingly on $X(\sigma \theta).$ In general, for all $k \in  K(\sigma),$  Ad$(k) : \frak{p}_0(\sigma) \longrightarrow \frak{p}_0(\sigma)$ is a 
linear isometry, where the innere product on $\frak{p}_0(\sigma)$ is induced from the Killing form $B$ of $\frak{g}.$ So 
det$(\textrm{Ad}(k)|_{\frak{p}_0(\sigma)}) = \pm 1$ for all $k \in  K(\sigma).$
Now the complex structure of the Hermitian symmetric space $X = G/K$ is given by ad$(H_0)$ 
for some $H_0$ in the centre of $\frak{k}_0.$ If $\sigma|_\frak{z} = id$ that is, if $\sigma$ induces a holomorphic diffeomorphism on $X,$ 
then $\frak{p}_0(\sigma)$ is a complex subspace of $\frak{p}_0$ and $X(\sigma)$ is a Hermitian symmetric space. 
Also Ad$(k) : \frak{p}_0(\sigma) \longrightarrow \frak{p}_0(\sigma)$ is a complex linear map for all $k \in  K(\sigma).$ So 
condition \ref{oreq} holds for involutions $\sigma$ with $\sigma|_\frak{z} = id,$ by the following lemma. 

\begin{lemma} 
Let $V$ is a finite dimensional complex vector space and $T : V \longrightarrow V$ be a complex linear map. Let $V^\mathbb{R}$ be the 
real vector space $V$ and $T^\mathbb{R}$ be the real linear map $T$ on $V^\mathbb{R}.$ Then det$(T^\mathbb{R}) > 0.$
\end{lemma}

\begin{proof} 
Since $T$ is upper triangulable, there exists a basis $\{v_1, v_2, \dots ,v_n \}$ of $V$ such that $T(v_l) = \sum_{1 \le k \le l} z_{kl} v_k,$ where 
$z_{kl} \in \mathbb{C},$ for all $1 \le k, l \le n.$ Let $z_{kl} = a_{kl} + i b_{kl},$ for some $a_{kl}, b_{kl} \in \mathbb{C},$ for all $k, l.$ Note that 
$\{v_1, iv_1, v_2, iv_2, \ldots , v_n, iv_n \}$ is a basis of $V^\mathbb{R}$ and $T^\mathbb{R}(v_l) = \sum_{1 \le k \le l} (a_{kl} v_k + b_{kl}(iv_k)),\ 
T^\mathbb{R}(iv_l) = \sum_{1 \le k \le l} ( - b_{kl}v_k + a_{kl} (iv_k))$ for all $k, l.$ So det$(T^\mathbb{R}) = \prod_{1 \le k \le n} \textrm{det}\left(
\begin{array}{ccc}
a_{kk} & -b_{kk} \\
b_{kk} & a_{kk} \\
\end{array}
\right) = \prod_{1 \le k \le n} (a^2_{kk} + b^2_{kk}) > 0.$ 
\end{proof} 

So if $\sigma|_\frak{z} = id,$ then $G(\sigma)$ acts orientation preservingly on $X(\sigma).$ Actually we have the following proposition. 

\begin{proposition}\label{orp}
If $m = 2l-1,$ then $G(\sigma)$ acts orientation preservingly on $X(\sigma)$ for $\sigma = \sigma_p, \sigma_p \theta (2 \le p \le l, l \ge 2)$ in the Table \ref{inv1}; and 
$G(\sigma)$ acts orientation reversingly on $X(\sigma)$ for $\sigma = \tau_p, \tau_p \theta (1 \le p \le l, l \ge 1)$ in the Table \ref{inv1}. \\ 
If $m = 2l-2 (l \ge 2),$ then $G(\sigma)$ acts orientation preservingly on $X(\sigma)$ for all $\sigma$ in the Table \ref{inv2}. 
\end{proposition} 

To prove this proposition, it remains to check the condition \ref{oreq} for involutions $\sigma$ with $\sigma|_\frak{z} \neq id,$ which is done below. 

Consider the map $\tau_p : G \longrightarrow G$ given by  $\tau_p(X) = \textrm{diag}(-1,I_{2p-1},-I_{m+2-2p}) X \textrm{diag}(-1,I_{2p-1},$ \\ $-I_{m+2-2p})$ for all 
$X \in G,$ where $1 \le p \le \frac{m+1}{2}$ if $m$ is odd, $1 \le p \le \frac{m-2}{2}$ if $m > 2$ is even, $p =1$ if $m = 2.$ Then $K(\tau_p) = \{\textrm{diag}(\pm I_2, X, Y) : 
X \in O(2p-2), Y \in O(m+2 -2p),\textrm{det}(X)\textrm{det}(Y) = 1 \}$ whose identity component is $K(\tau_p)_0 = \{\textrm{diag}(I_2, X, Y) : X \in SO(2p-2), Y \in SO(m+2 -2p) \}.$ 
Therefore the components of $K(\tau_p)$ are 
\[\begin{cases}
K(\tau_p)_0, Y_2 K(\tau_p)_0, Y_3 K(\tau_p)_0, Y_4 K(\tau_p)_0 \textrm{ if } p > 1, \\
K(\tau_p)_0, Y_2 K(\tau_p)_0 \textrm{ if } p = 1; 
\end{cases}\] 
where $Y_2 = \textrm{diag}(-I_2, I_m), Y_3 = \textrm{diag}(I_2, -1, I_{2p-3},-1, I_{m+1-2p}), Y_4 = \textrm{diag}(-I_2, -1, I_{2p-3},-1, $ \\ $I_{m+1-2p}).$ 
We have $\frak{p}_0(\tau_p) = \{\left(
\begin{array}{ccccccc}
0 & 0 & 0 & X_1 \\
0 & 0 & X_2 & 0  \\
0 & X_2^t & 0 & 0 \\ 
X_1^t & 0 & 0 & 0 \\ 
\end{array}
\right) : X_1 \in M_{1 \times (m+2-2p)}(\mathbb{R}), X_2 \in M_{1 \times (2p-2)}(\mathbb{R}) \}.$ Thus $\{E_{1k} +E_{k1}, E_{2r} + E_{r2} : 2p+1 \le k \le m+2, 3 \le r \le 2p \}$ 
is a basis of $\frak{p}_0(\tau_p),$ if $p > 1;$ and $\{E_{1k} +E_{k1} : 3 \le k \le m+2 \}$ is a basis of $\frak{p}_0(\tau_1).$ 
Now $\textrm{Ad}(Y_2)(E_{1k} + E_{k1}) = - (E_{1k} + E_{k1})$ for all $2p+1 \le k \le m+2,\ \textrm{Ad}(Y_2)(E_{2r} + E_{r2}) = - (E_{2r} + E_{r2})$ for all $3 \le r \le 2p (p > 1).$ 
Thus for all $p \ge 1,$ det$(\textrm{Ad}(Y_2)|_{\frak{p}_0(\tau_p)}) = 1$ if $m$ is even, and det$(\textrm{Ad}(Y_2)|_{\frak{p}_0(\tau_p)}) = -1$ if $m$ is odd. Hence if $m$ is odd, 
then $G(\tau_p)$ acts orientation reversingly on $X(\tau_p)$ for all $p \ge 1$ and if $m$ is even, $G(\tau_1)$ acts orientation preservingly on $X(\tau_1).$ Now assume that 
$m$ is even and $p > 1.$ Then $\textrm{Ad}(Y_3)(E_{1k} + E_{k1}) = - (E_{1k} + E_{k1})$ if $k = 2p+1,$ and $\textrm{Ad}(Y_3)(E_{1k} + E_{k1}) = (E_{1k} + E_{k1})$ if 
$k > 2p+1,\ \textrm{Ad}(Y_3)(E_{2r} + E_{r2}) = - (E_{2r} + E_{r2})$ if $r =3,$ and $\textrm{Ad}(Y_3)(E_{2r} + E_{r2}) = (E_{2r} + E_{r2})$ if $r > 3,$ 
$\textrm{Ad}(Y_4)(E_{1k} + E_{k1}) = (E_{1k} + E_{k1})$ if $k = 2p+1,$ and $\textrm{Ad}(Y_3)(E_{1k} + E_{k1}) = -(E_{1k} + E_{k1})$ if 
$k > 2p+1,\ \textrm{Ad}(Y_3)(E_{2r} + E_{r2}) = (E_{2r} + E_{r2})$ if $r=3,$ and $\textrm{Ad}(Y_3)(E_{2r} + E_{r2}) = -(E_{2r} + E_{r2})$ if $r > 3.$ Thus 
det$(\textrm{Ad}(Y_3)|_{\frak{p}_0(\tau_p)}) = 1,$ det$(\textrm{Ad}(Y_4)|_{\frak{p}_0(\tau_p)}) = 1$ and $G(\tau_p)$ acts orientation preservingly on $X(\tau_p).$

Let $m = 2l-2 (l \ge 2)$ and consider the map $\mu_p : G \longrightarrow G$ given by  $\mu_p(X) = \textrm{diag}(-1,1,-I_{2p-2},I_{2l-2p-1},-1) X 
\textrm{diag}(-1,1,-I_{2p-2},I_{2l-2p-1},-1)$ for all $X \in G,$ where $1 \le p \le l-2$ if $l > 2,$ $p =1$ if $l = 2.$ Then $K(\mu_p) = \{\left(
\begin{array}{ccccccc}
\pm I_2 & 0 & 0 & 0 \\
0 & X_{22} & 0 & X_{24}  \\
0 & 0 & X_{33} & 0 \\ 
0 & X_{42} & 0 & x_{44} \\ 
\end{array}
\right) : A = X_{33} \in O(2l-2p-1), X_{22} \in M_{(2p-2) \times (2p-2)}(\mathbb{R}), X_{24} \in M_{(2p-2) \times 1}(\mathbb{R}), 
X_{42} \in M_{1 \times (2p-2)}(\mathbb{R}), x_{44} \in \mathbb{R}, B = \left(
\begin{array}{ccc}
X_{22} & X_{24} \\
X_{42} & x_{44} \\
\end{array}
\right) \in O(2p-1), \textrm{det}(A)\textrm{det}(B) = 1\}$
whose identity component is $K(\mu_p)_0 = \{\left(
\begin{array}{ccccccc}
I_2 & 0 & 0 & 0 \\
0 & X_{22} & 0 & X_{24}  \\
0 & 0 & X_{33} & 0 \\ 
0 & X_{42} & 0 & x_{44} \\ 
\end{array}
\right)  : X_{33} \in SO(2l-2p-1), X_{22} \in M_{(2p-2) \times (2p-2)}(\mathbb{R}), X_{24} \in M_{(2p-2) \times 1}(\mathbb{R}), 
X_{42} \in M_{1 \times (2p-2)}(\mathbb{R}), x_{44} \in \mathbb{R}, \left(
\begin{array}{ccc}
X_{22} & X_{24} \\
X_{42} & x_{44} \\
\end{array}
\right) \in SO(2p-1)\}.$
Therefore the components of $K(\mu_p)$ are 
$K(\mu_p)_0, Y_2 K(\mu_p)_0, Y_3 K(\mu_p)_0, Y_4 K(\mu_p)_0$ for $p \ge 1,$ 
where $Y_2 = \textrm{diag}(-I_2, I_{2l-2}), Y_3 = \textrm{diag}(I_{2l-2},-I_2), Y_4 = \textrm{diag}(-I_2, I_{2l-4},-I_2).$ 
We have $\frak{p}_0(\mu_p) = \{\left(
\begin{array}{ccccccccc}
0 & 0 & X_{11} & 0 & x_{13} \\
0 & 0 & 0 & X_{22} & 0  \\
X_{11}^t & 0 & 0 & 0 & 0 \\ 
0 & X_{22}^t & 0 & 0 & 0 \\
x_{13} & 0 & 0 & 0 & 0 \\ 
\end{array}
\right) : X_{11} \in M_{1 \times (2p-2)}(\mathbb{R}), X_{22} \in M_{1 \times (2l-2p-1)}(\mathbb{R}), x_{13} \in \mathbb{R}\}.$ Thus $\{E_{1k} +E_{k1}, E_{2r} + E_{r2} : 
3 \le k \le 2p \textrm{ or } k = 2l, 2p+1 \le r \le 2l-1 \}$ is a basis of $\frak{p}_0(\mu_p).$ 
Now $\textrm{Ad}(Y_2)(E_{1k} + E_{k1}) = - (E_{1k} + E_{k1})$ for all $3 \le k \le 2p$ or $k = 2l,\ \textrm{Ad}(Y_2)(E_{2r} + E_{r2}) = - (E_{2r} + E_{r2})$ for all $2p+1 \le r \le 2l-1;$ 
$\textrm{Ad}(Y_3)(E_{1k} + E_{k1}) = (E_{1k} + E_{k1})$ for all $3 \le k \le 2p,\ \textrm{Ad}(Y_3)(E_{2r} + E_{r2}) = (E_{2r} + E_{r2})$ for all $2p+1 \le r \le 2l-2,\ 
\textrm{Ad}(Y_3)(E_{2,2l-1} + E_{2l-1,2}) = - (E_{2,2l-1} + E_{2l-1,2}),\ \textrm{Ad}(Y_3)(E_{1,2l} + E_{2l,1}) = - (E_{1,2l} + E_{2l,1});$
$\textrm{Ad}(Y_4)(E_{1k} + E_{k1}) = -(E_{1k} + E_{k1})$ for all $3 \le k \le 2p,\ \textrm{Ad}(Y_4)(E_{2r} + E_{r2}) = - (E_{2r} + E_{r2})$ for all $2p+1 \le r \le 2l-2,\ 
\textrm{Ad}(Y_4)(E_{2,2l-1} + E_{2l-1,2}) = (E_{2,2l-1} + E_{2l-1,2}),\ \textrm{Ad}(Y_4)(E_{1,2l} + E_{2l,1}) = (E_{1,2l} + E_{2l,1}).$
Thus det$(\textrm{Ad}(Y_i)|_{\frak{p}_0(\mu_p)}) = 1,$ for $2 \le i \le 4$ and $G(\mu_p)$ acts orientation preservingly on $X(\mu_p).$

Let $m = 4$ that is, $l = 3$ and consider the map $\sigma'_0 : G \longrightarrow G$ given by  $\sigma'_0(X) = \textrm{diag}(-1,1,-1,1,-1,1) X 
\textrm{diag}(-1,1,-1,1,-1,1)$ for all $X \in G.$ Then $K(\sigma'_0)$ has the identity component $K(\sigma'_0)_0 = \{\left(
\begin{array}{ccccccccc}
I_2 & 0 & 0 & 0 & 0 \\
0 & x_{33} & 0 & x_{35} & 0 \\
0 & 0 & x_{44} & 0 & x_{46} \\ 
0 & x_{53} & 0 & x_{55} & 0 \\ 
0 & 0 & x_{64} & 0 & x_{66} \\ 
\end{array}
\right) : \left(  
\begin{array}{ccc}
x_{33} & x_{35} \\
x_{53} & x_{55} \\
\end{array}
\right), \left(
\begin{array}{ccc}
x_{44} & x_{46} \\
x_{64} & x_{66} \\
\end{array}
\right) \in SO(2)\}.$
Therefore the components of $K(\sigma'_0)$ are 
$K(\sigma'_0)_0, Y_2 K(\sigma'_0)_0, Y_3 K(\sigma'_0)_0, Y_4 K(\sigma'_0)_0,$ 
where $Y_2 = \textrm{diag}(-I_2, I_4), Y_3 = \textrm{diag}(I_4,-I_2), Y_4 = \textrm{diag}(-I_2,I_2,-I_2).$ 
Also $\{E_{13} +E_{31}, E_{15} +E_{51}, E_{24} + E_{42}, E_{26} + E_{62}\}$ is a basis of $\frak{p}_0(\sigma'_0).$ 
Now $\textrm{Ad}(Y_2)(E_{1k} + E_{k1}) = - (E_{1k} + E_{k1})$ for $k = 3,5,\ \textrm{Ad}(Y_2)(E_{2r} + E_{r2}) = - (E_{2r} + E_{r2})$ for $r = 4,6;$ 
$\textrm{Ad}(Y_3)(E_{13} + E_{31}) = (E_{13} + E_{31}),\ \textrm{Ad}(Y_3)(E_{24} + E_{42}) = (E_{24} + E_{42}),\ 
\textrm{Ad}(Y_3)(E_{15} + E_{51}) = - (E_{15} + E_{51}),\ \textrm{Ad}(Y_3)(E_{26} + E_{62}) = - (E_{26} + E_{62});$
$\textrm{Ad}(Y_4)(E_{13} + E_{31}) = - (E_{13} + E_{31}),\ \textrm{Ad}(Y_4)(E_{24} + E_{42}) = - (E_{24} + E_{42}),\ 
\textrm{Ad}(Y_4)(E_{15} + E_{51}) = (E_{15} + E_{51}),\ \textrm{Ad}(Y_4)(E_{26} + E_{62}) = (E_{26} + E_{62}).$
Thus det$(\textrm{Ad}(Y_i)|_{\frak{p}_0(\sigma'_0)}) = 1,$ for $2 \le i \le 4$ and $G(\sigma'_0)$ acts orientation preservingly on $X(\sigma'_0).$

Assume that $m = 2$ that is, $l = 2.$ Consider the maps $\eta_1, \eta_2, \tau'_1 : G \longrightarrow G$ given by 
$\eta_1(X) = J'_2 X J'_2,\ \eta_2(X) = J'_{1,1} X J'_{1,1},\ \tau'_1(X) = \textrm{diag}(-I_2,1,-1) X \textrm{diag}(-I_2,1,-1)$ 
for all $X \in G,$ where $J'_2 = \left(
\begin{array}{ccc}
0 & I_2 \\
I_2 & 0 \\
\end{array}
\right)$ and $J'_{1,1} = \left(
\begin{array}{ccc}
0 & I_{1,1} \\
I_{1,1} & 0 \\
\end{array}
\right).$ Then $K(\eta_1) = \{\left(
\begin{array}{ccc}
X & 0 \\
0 & X \\
\end{array}
\right) : X \in SO(2) \}$ and $K(\eta_2) = \{\left(
\begin{array}{ccccccc}
x_{11} & x_{12} & 0 & 0 \\ 
x_{21} & x_{22} & 0 & 0 \\ 
0 & 0 & x_{11} & -x_{12} \\ 
0 & 0 & -x_{21} & x_{22} \\ 
\end{array}
\right) : \left(
\begin{array}{ccc}
x_{11} & x_{12} \\
x_{21} & x_{22} \\
\end{array}
\right) \in SO(2) \}$ are connected, but $K(\tau'_1)$ has two components $K(\tau'_1)_0,\ Y_2K(\tau'_1)_0,$ where 
$K(\tau'_1)_0 = \{\left(
\begin{array}{ccc}
X & 0 \\
0 & I_2 \\
\end{array}
\right) : X \in SO(2) \}, Y_2 = \left(
\begin{array}{ccc}
I_2 & 0 \\
0 & -I_2 \\
\end{array}
\right).$ Now $\frak{p}_0(\tau'_1)$ has basis $\{E_{14} + E_{41}, E_{24} + E_{42} \}$ and Ad$(Y_2)(E_{14} + E_{41}) = 
- (E_{14} + E_{41}),\ \textrm{Ad}(Y_2)(E_{24} + E_{42}) = - (E_{24} + E_{42})$ so that det$(\textrm{Ad}(Y_2)|_{\frak{p}_0(\tau'_1)}) = 1.$
Hence $G(\sigma)$ acts orientation preservingly on $X(\sigma)$ for $\sigma = \eta_1, \eta_2, \tau'_1.$

\noindent 
\section{Proof of Theorem \ref{main}}\label{proof} 

Let $F$ be a totally real algebraic number field of degree $>1$ and $\mathcal{O}$ be the ring of integers of $F.$ Then $\bar{G} = \textrm{Ad}(G)$ is a linear connected  
semisimple Lie group defined over $F$ via the basis $B_F$ (or $B'_F$) and $\bar{K} = \textrm{Ad}(K)$ is a maximal compact subgroup of $\bar{G}$ associated with 
the Cartan involution $\textrm{Ad}(g) \mapsto \theta \textrm{Ad}(g) \theta,$ denoted by the same notation $\theta.$ 
The Cartan involution $\theta$ of $\bar{G}$ is defined over $F.$ Consider 
the arithmetic uniform lattice $\Gamma$ (or $\Gamma'$) of Aut$(\frak{g}_0).$ Then $\Gamma$ (or $\Gamma'$) is invariant under $\theta.$  
By a result of Millson and Raghunathan \cite[Th. 2.1]{mira} (see also \cite[Th. 4.11]{rs}), 
if $\sigma$ is an involution of $\bar{G}$ defined over $F$ with $\sigma \theta = \theta \sigma$ and $\Gamma \subset \bar{G}_{\mathcal{O}}$ be 
a torsion-free, $\langle \sigma , \theta \rangle$-stable, arithmetic uniform lattice of $\bar{G}$ such that  
the Lie groups $\bar{G}(\sigma), \bar{G}(\sigma \theta)$ act orientation preservingly on $\bar{G}(\sigma)/\bar{K}(\sigma)$ and $\bar{G}(\sigma \theta)/\bar{K}(\sigma \theta)$ respectively,
then there exists a $\langle \sigma , \theta \rangle$-stable subgroup $\Gamma_1$ of $\Gamma$ of finite index such that the cohomology 
classes defined by $[C(\sigma , \Gamma_1)], [C(\sigma \theta, \Gamma_1)]$ via Poincar\'e duality are non-zero and are not represented 
by $\bar{G}$-invariant differential forms on $\bar{G}/\bar{K}$. Let $\sigma \neq \tau_p (1 \le p \le l)$ be an involution given in Table \ref{inv1} 
if $m = 2l-1,$ or $\sigma$ be an involution given in Table \ref{inv2} if $m = 2l-2.$ Then $\sigma \theta = \theta \sigma,$ 
$\sigma \in \Gamma$ or $\sigma \in \Gamma'$ (Proposition \ref{prep}), and $G(\sigma), G(\sigma \theta)$ act orientation preservingly on $X(\sigma), X(\sigma \theta)$ 
respectively (Proposition \ref{orp}). The involution $\textrm{Ad}(g) \mapsto \sigma \textrm{Ad}(g) \sigma$ of $\bar{G}$ corresponding to $\sigma,$ 
is denoted by the same notation $\sigma.$ Then $\sigma$ is defined over $F.$ Let $\Gamma_0$ (respectively, $\Gamma'_0$) be the set of all torsion-free elements of 
$\Gamma \cap \bar{G}$ (respectively, $\Gamma' \cap \bar{G}$). 
Then $\Gamma_0 \subset \bar{G}_\mathcal{O}$ (or $\Gamma'_0 \subset \bar{G}_\mathcal{O}$) is a torsion-free, $\langle \sigma , \theta \rangle$-stable, 
arithmetic uniform lattice of $\bar{G}.$ We have $\bar{G}(\sigma), \bar{G}(\sigma \theta)$ act orientation preservingly on $\bar{G}(\sigma)/\bar{K}(\sigma)$ and 
$\bar{G}(\sigma \theta)/\bar{K}(\sigma \theta)$ respectively. So there exists a $\langle \sigma , \theta \rangle$-stable subgroup $\bar{\Gamma}_\sigma$ of $\Gamma_0$ 
or $\Gamma'_0$ of finite index such that the cohomology classes defined by $[C(\sigma , \bar{\Gamma}_\sigma)], [C(\sigma \theta, \bar{\Gamma}_\sigma)]$ via 
Poincar\'e duality are non-zero and are not represented by $\bar{G}$-invariant differential forms on $\bar{G}/\bar{K}.$ Since $G$ is a covering group of $\bar{G},$ 
the cohomology classes defined by $[C(\sigma , \Gamma_\sigma)], [C(\sigma \theta, \Gamma_\sigma)]$ via Poincar\'e duality are also not represented by $G$-invariant 
differential forms on $X,$ where $\Gamma_\sigma = \textrm{Ad}^{-1}(\bar{\Gamma}_\sigma).$ This completes the proof.

\noindent
\section{Cohomologically induced representations of $SO_0(2,m)$ with the trivial infinitesimal character}\label{rep} 

A $\theta$-stable parabolic subalgebra of $\frak{g}_0$ is a parabolic subalgebra $\frak{q}$ of $\frak{g}$ such that 
(a) $\theta(\frak{q}) = \frak{q}$, and (b) $\bar{\frak{q}} \cap \frak{q}= \frak{l}$ is a Levi subalgebra of $\frak{q}$; 
where $\bar{\ }$ denotes the conjugation of $\frak{g}$ with respect to $\frak{g}_0$. Let $\frak{u}$ be the nilradical of $\frak{q}$ so that $\frak{q} = \frak{l} \oplus \frak{u}.$ 
Then $\frak{l}, \frak{u}$ are $\theta$-stable. By (b), $\frak{l}$ is the complexification of a real subalgebra $\frak{l}_0$ of $\frak{g}_0$. 
Now corresponding to a $\theta$-stable parabolic subalgebra $\frak{q}$, we have an irreducible unitary representation $A_\frak{q}$ of $G$ 
whose $(\frak{g}, K)$-module $A_{\frak{q}, K} = \mathcal{R}^S _\frak{q} (\mathbb{C}),$ where $S = \textrm{dim}
(\frak{u} \cap \frak{k}).$ The representation $A_\frak{q}$ has trivial infinitesimal character. 
The $(\frak{g}, K)$-modules $A_{\frak{q} , K}$ were first constructed, in general, by Parthasarathy \cite{parthasarathy1}. 
Vogan and Zuckerman \cite{voganz} gave a construction of the $(\frak{g}, K)$-modules $A_{\frak{q} , K}$ via cohomological induction and 
Vogan \cite{vogan} proved that these are unitarizable.
If $\frak{q}$ is a $\theta$-stable parabolic subalgebra, then so is $Ad(k)(\frak{q})$ and 
$A_\frak{q}$ is unitarily equivalent to $A_{Ad(k)(\frak{q})}$ for all $k \in K.$ So we may assume that $\frak{q}$ contains the Borel subalgebra 
$\frak{h} \oplus \sum_{\alpha \in \Delta_\frak{k}^+} \frak{g}^\alpha$ of $\frak{k}.$ Now define 
$\Delta(\frak{u} \cap \frak{p}_+) = \{\alpha \in \Delta_n^+ : \frak{g}^\alpha \subset \frak{u} \},\ \Delta(\frak{u} \cap \frak{p}_-) = \{\alpha \in \Delta_n^- : \frak{g}^\alpha \subset \frak{u} \}$ 
and $\Delta(\frak{u} \cap \frak{p}) = \Delta(\frak{u} \cap \frak{p}_+) \cup \Delta(\frak{u} \cap \frak{p}_-).$ If $\frak{q}, \frak{q}'$ be two $\theta$-stable parabolic subalgebras of 
$\frak{g}_0$ containing $\frak{h} \oplus \sum_{\alpha \in \Delta_\frak{k}^+} \frak{g}^\alpha$ of $\frak{k},$ then 
$A_\frak{q}$ is unitarily equivalent to $A_{\frak{q}'}$ {\it if and only if} 
\begin{equation} \label{cong}
\Delta(\frak{u} \cap \frak{p}) = \Delta(\frak{u'} \cap \frak{p}).
\end{equation}
See \cite{riba}. If $\frak{q}$ is a $\theta$-stable 
Borel subalgebra, then $A_\frak{q}$ is a discrete series representation and if $\frak{q} = \frak{g},$ then $A_\frak{q}$ is the trivial representation. 

The representations $A_\frak{q}$ are the only irreducible unitary representations with non-zero $(\frak{g}, K)$-cohomology and 
\[  H^r (\frak{g}, K; A_{\frak{q}, K}) \cong H^{r-R(\frak{q})} (Y_\frak{q} ; \mathbb{C}),\] 
where $Y_\frak{q}$ is the compact dual of the Riemannian globally symmetric space $L/{L\cap K},\ L = \{g \in G : \textrm{Ad}(g) (\frak{q}) = 
\frak{q} \}$ is the connected Lie subgroup of $G$ with Lie algebra $\frak{l}_0,\ L \cap K$ is a maximal compact subgroup of $L$ and 
$R(\frak{q}) = |\Delta(\frak{u} \cap \frak{p})|,$ the cardinality of $\Delta(\frak{u} \cap \frak{p}).$ See \cite{voganz}. 
Since $G/K$ is Hermitian symmetric, for each $\theta$-stable parabolic subalgebra $\frak{q}$ of $\frak{g}_0,$ 
the space $Y_\frak{q}$ is Hermitian symmetric. Let $P(Y_\frak{q}; x,t),\ P_\frak{q}(x,t)$ denote the Poincar\'{e}-Hodge polynomials of $H^*(Y_\frak{q} ; \mathbb{C})$ and 
$H^* (\frak{g}, K; A_{\frak{q}, K})$ respectively. We have 
\begin{equation}\label{cohomology}
P_\frak{q}(x,t) = x^{R_+(\frak{q})}t^{R_-(\frak{q})}P(Y_\frak{q}; x,t), 
\end{equation} 
where $R_+(\frak{q}) = |\Delta(\frak{u} \cap \frak{p})|,\ R_-(\frak{q}) = |\Delta(\frak{u} \cap \frak{p})|.$ Using the condition \ref{cong} and the formula \ref{cohomology}, 
the unitary equivalence classes of representations $A_\frak{q}$ of $G$ and the Poincar\'{e}-Hodge polynomial of $H^* (\frak{g}, K; A_{\frak{q}, K})$ are determined in \cite{pp}. 
This result is described in the Table \ref{b-table} and Table \ref{d-table}. 

If $\Gamma \subset G$ is a uniform lattice, we have 
\[L^2 (\Gamma \backslash G) \cong \widehat{\bigoplus}_{\pi \in \hat{G}} m(\pi , \Gamma ) H_\pi, \] 
due to Gelfand and Pyatetskii-Shapiro \cite{ggp}, \cite{gp}; where $L^2 (\Gamma \backslash G)$ is a Hilbert space of square integrable functions on 
$\Gamma \backslash G$ with respect to a $G$-invariant measure and it is a unitary representation of $G$ relative to the right traslation action of $G$ on  
$L^2 (\Gamma \backslash G),\ H_\pi$ is the representation space of $\pi \in \hat{G},\ m(\pi , \Gamma ) \in \mathbb{N} \cup \{0\}$, the multiplicity of $\pi$ in $L^2 (\Gamma \backslash G)$. 
If $(\tau, \mathbb{C})$ is the trivial representation of $G$, then $m(\tau , \Gamma) = 1.$
A unitary representation $\pi \in \hat{G}$ such that $m(\pi , \Gamma ) > 0$ for some uniform lattice $\Gamma,$ is called an automorphic representation of $G$ with respect to $\Gamma$.  
If in addition, the lattice $\Gamma$ is torsion-free, we have the Matsushima's isomorphism \cite{matsushima} 
\begin{equation} \label{matsushima} 
H^r (\Gamma \backslash X; \mathbb{C}) \cong \bigoplus_{[\frak{q}] \in \mathcal{Q}} m(\frak{q}, \Gamma) H^r (\frak{g}, K; A_{\frak{q},K});
\end{equation} 
where $\mathcal{Q}$ is the set of equivalence classes of $\theta$-stable parabolic subalgebras of $\frak{g}_0$ containing $\frak{h} \oplus \sum_{\alpha \in \Delta_\frak{k}^+} \frak{g}^\alpha$ 
under the equivalence relation $\frak{q} \sim \frak{q}'$ {\it iff} $\Delta(\frak{u} \cap \frak{p}) = \Delta(\frak{u'} \cap \frak{p}),$ and $m(\frak{q},\Gamma)$ is the multiplicity of $A_\frak{q}$ in 
$L^2 (\Gamma \backslash G).$ Since $\Gamma \backslash X$ is a compact K\"{a}hler manifold, we have the Hodge decomposition 
\[H^r (\Gamma \backslash X; \mathbb{C}) \cong \bigoplus_{p+q = r} H^{p,q}(\Gamma \backslash X; \mathbb{C}). \] 

Now from Theorem \ref{main}, we have for each $\sigma \neq \tau_p (1 \le p \le l)$ in Table \ref{inv1} for $m = 2l-1,$ or for each $\sigma$ in Table \ref{inv2} for $m = 2l-2,$ 
there exists a torsion-free, $\langle \sigma , \theta \rangle$-stable, arithmetic, uniform lattice $\Gamma_\sigma$ of $G$ and non-zero cohomology 
classes in $H^{d(\sigma)} (\Gamma_\sigma \backslash X; \mathbb{C}) (d(\sigma) = \textrm{dim}(X(\sigma)))$ and 
$H^{d(\sigma \theta)}(\Gamma_\sigma \backslash X; \mathbb{C}) (d(\sigma \theta) = \textrm{dim}(X(\sigma \theta)))$ 
defined by $[C(\sigma , \Gamma_\sigma)]$ and $[C(\sigma \theta, \Gamma_\sigma)]$ respectively and these do not lie in $H^* (\frak{g}, K; A_{\frak{g},K})$ 
via the isomorphism \ref{matsushima}. We say that {\it an involution $\sigma$ mentioned in Theorem \ref{main} has no $A_\frak{q}$-component} if 
the Poincar\'{e} duals of the fundamental classes $[C(\sigma , \Gamma_\sigma)]$ and $[C(\sigma \theta, \Gamma_\sigma)]$ have zero component on 
$m_\frak{q} H^* (\frak{g}, K; A_{\frak{q},K})$ via the isomorphism \ref{matsushima}. The dimensions $d(\sigma)$ and $d(\sigma \theta)$ for each 
involution $\sigma$ mentioned in Theorem \ref{main} are listed in the Table \ref{dimension}.
If in addition, $\sigma|_\frak{z} = id,$ then $C(\sigma , \Gamma_\sigma),\ C(\sigma \theta, \Gamma_\sigma)$ are compact locally 
Hermitian symmetric spaces and the Poincar\'{e} duals of the fundamental classes $[C(\sigma , \Gamma_\sigma)]$ and $[C(\sigma \theta, \Gamma_\sigma)]$ 
lie in $H^{\frac{d(\sigma)}{2}, \frac{d(\sigma)}{2}} (\Gamma_\sigma \backslash X; \mathbb{C})$ and 
$H^{\frac{d(\sigma \theta)}{2}, \frac{d(\sigma \theta)}{2}}(\Gamma_\sigma \backslash X; \mathbb{C})$ respectively. Combining all these, we have determined the involutions 
$\sigma$ which have no $A_\frak{q}$-component for each $[\frak{q}] \in \mathcal{Q},$ and these are listed in the Table \ref{b-table} and Table \ref{d-table}. 

\begin{table}[!h]
\caption{Dimensions of $X(\sigma)$ and $X(\sigma \theta):$}\label{dimension}
\begin{tabular}{||c|c|c||}
\hline
$\sigma$ &  $d(\sigma)$ & $d(\sigma \theta)$  \\
\hline
\hline
$\sigma_p (2 \le p \le l, m = 2l-1, l \ge 2)$ & $2(2p-2)$ & $2(2l-2p+1)$ \\
$\sigma_p (2 \le p \le l-2, m = 2l-2, l \ge 4)$ & $2(2p-2)$ & $2(2l-2p)$ \\
$\sigma_{l-1}, \sigma_l (m = 2l-2, l \ge 3)$ & $2(l-1)$ & $2(l-1)$ \\
$\tau'_p (1 \le p \le l-2, m = 2l-2, l \ge 3)$ & $2(2p-1)$ & $2(2l-2p-1)$ \\ 
$\tau_p (1 \le p \le l-2, m = 2l-2, l \ge 3)$ & $2(l-1)$ & $2(l-1)$ \\
$\mu_p (1 \le p \le l-2, m = 2l-2, l \ge 3)$ & $2(l-1)$ & $2(l-1)$ \\
$\sigma_0, \sigma'_0 (m = 4, l = 3)$ & $4$ & $4$ \\
$\sigma_1, \tau'_1, \tau_1, \mu_1 (m = 2, l = 2)$ & $2$ & $2$ \\
$\eta_1, \eta_2 (m = 2, l = 2)$ & $3$ & $1$ \\
\hline
\end{tabular} 
\end{table}

\begin{table}[!h] 
\caption{Involutions with no $A_\frak{q}$-component for $G = SO_0(2,2l-1)(l \ge 2):$}\label{b-table}
\begin{tabular}{||c|c|c|c||}
\hline
$\Delta(\frak{u} \cap \frak{p}_-)$ & $\Delta(\frak{u} \cap \frak{p}_+)$ & $P_\frak{q}(x,t)$ & 
\begin{tabular}{c} $\sigma$ with no \\ $A_\frak{q}$-component \end{tabular} \\ 
\hline
\hline
empty & 
\begin{tabular}{c} $\{\beta \in \Delta_n^+ :\beta \ge $\\$ \phi_1\}$\\ \\ 
$\{\beta \in \Delta_n^+ : \beta \ge $\\$\phi_1+\cdots+\phi_i\}$\\$(2 \le i \le l)$\\ \\empty \end{tabular} & 
\begin{tabular}{c} $x^{2l-1}$\\ \\$x^{2l-i}+x^{2l-i+1}t$\\$+\cdots+x^{2l-1}t^{i-1}$\\$(2 \le i \le l)$\\ \\$1+xt+\cdots$\\$+x^{2l-1}t^{2l-1}$   \end{tabular} & 
\begin{tabular}{c} $\sigma_p$\\  \\  \\$\sigma_p$\\    \\    \\ $-$  \\ \end{tabular}  \\  
\hline 
\begin{tabular}{c} $-\{\beta \in \Delta_n^+ : \beta \le $\\$\phi_1+\cdots +\phi_i\}$\\$(1 \le i \le l-1)$ \end{tabular} & 
\begin{tabular}{c} $\{\beta \in \Delta_n^+ : \beta \ge $\\$\phi_1+\cdots +\phi_{i+1}\}$\\ \\$\{\beta \in \Delta_n^+ : \beta \ge $\\$\phi_1+\cdots +\phi_j\}$\\$(i+2 \le j \le l)$\\ \\
$\{\beta \in \Delta_n^+ : \beta \ge $\\$\phi_1+\cdots +\phi_i+$\\$2\phi_{i+1}+\cdots+2\phi_l\}$ \end{tabular} & 
\begin{tabular}{c} $x^{2l-1-i}t^i$\\ \\$x^{2l-j}t^i+x^{2l-j+1}t^{i+1}$\\$+\cdots+x^{2l-i-1}t^{j-1}$\\$(i+2 \le j \le l)$\\ \\$x^it^i+x^{i+1}t^{i+1}+$\\$\cdots+x^{2l-i-1}t^{2l-i-1}$   \end{tabular} & 
\begin{tabular}{c} $\sigma_p$\\  \\  \\$\sigma_p$\\    \\    \\ $\sigma_p (p < \frac{i+2}{2}$\\$\textrm{ or } p > \frac{2l-i+1}{2})$\\ \end{tabular}  \\  
\hline 
\begin{tabular}{c} $-\{\beta \in \Delta_n^+ : \beta \le $\\$\phi_1+\cdots +\phi_l\}$ \end{tabular} & 
\begin{tabular}{c} $\{\beta \in \Delta_n^+ :\beta \ge $\\$\phi_1+\cdots +\phi_{l-1}+2\phi_l\}$\\ \\
$\{\beta \in \Delta_n^+ : \beta \ge $\\$\phi_1+\cdots +\phi_{j-1}$\\$+2\phi_j+\cdots+2\phi_l\}$\\$(2 \le j \le l-1)$\\ \\ empty \end{tabular} & 
\begin{tabular}{c} $x^{l-1}t^l$\\ \\$x^{j-1}t^l+x^{j}t^{l+1}$\\$+\cdots+x^{l-1}t^{2l-j}$\\$(2 \le j \le l-1)$\\ \\$t^l+xt^{l+1}+$\\$\cdots+x^{l-1}t^{2l-1}$   \end{tabular} & 
\begin{tabular}{c} $\sigma_p$\\  \\  \\$\sigma_p$\\    \\    \\ $\sigma_p$\\ \end{tabular}  \\  
 \hline
\begin{tabular}{c} $-\{\beta \in \Delta_n^+ : \beta \le $\\$\phi_1+\cdots +\phi_{i-1}$\\$+2\phi_i+\cdots+2\phi_l\},$\\$(3\le i \le l)$ \end{tabular} & 
\begin{tabular}{c} $\{\beta \in \Delta_n^+ : \beta \ge $\\$\phi_1+\cdots +\phi_{i-2}$\\$+2\phi_{i-1}+\cdots+2\phi_l\}$\\ \\
$\{\beta \in \Delta_n^+ : \beta \ge $\\$\phi_1+\cdots +\phi_{j-1}$\\$+2\phi_j+\cdots+2\phi_l\}$\\$(2 \le j \le i-2)$\\ \\ empty \end{tabular} & 
\begin{tabular}{c} $x^{i-2}t^{2l-i+1}$\\ \\$x^{j-1}t^{2l-i+1}+x^jt^{2l-i+2}$\\$+\cdots+x^{i-2}t^{2l-j}$\\$(2 \le j \le i-2)$\\ \\$t^{2l-i+1}+xt^{2l-i+2}$\\$+\cdots+x^{i-2}t^{2l-1}$   \end{tabular} &   
\begin{tabular}{c} $\sigma_p$\\  \\  \\$\sigma_p$\\    \\    \\ $\sigma_p$\\ \end{tabular}  \\  
\hline
\begin{tabular}{c} $-\{\beta \in \Delta_n^+ : \beta \le \phi_1$\\$+2\phi_2+\cdots+2\phi_l\}$ \end{tabular} & 
empty & $t^{2l-1}$ & $\sigma_p$ \\ 
\hline
\end{tabular} 
\end{table}

\begin{table}
\caption{Involutions with no $A_\frak{q}$-component for $G = SO_0(2,2l-2)(l \ge 2):$}\label{d-table}
\begin{tabular}{||c|c|c|c||}
\hline
$\Delta(\frak{u} \cap \frak{p}_-)$ & $\Delta(\frak{u} \cap \frak{p}_+)$ & $P_\frak{q}(x, t)$ & 
\begin{tabular}{c} $\sigma$ with no \\ $A_\frak{q}$-component \end{tabular} \\ 
\hline
\hline
empty ($l \ge 3$) & 
\begin{tabular}{c} $\{\beta \in \Delta_n^+ : \beta \ge $\\$\phi_1\}$\\ \\ 
$\{\beta \in \Delta_n^+ : \beta \ge $\\$\phi_1+\cdots +\phi_i\}$\\$(2 \le i \le l-2)$\\ \\$\{\beta \in \Delta_n^+ : \beta \ge $\\$ \xi_i\} (i =1,2)$\\ \\
$\{\beta \in \Delta_n^+ : \beta \ge $\\$ \xi_1 \textrm{ or } \xi_2\}$\\ \\empty \end{tabular} & 
\begin{tabular}{c} $x^{2l-2}$\\ \\$x^{2l-i-1}+x^{2l-i}t$\\$+\cdots+x^{2l-2}t^{i-1}$\\ \\$x^{l-1}+x^lt+$\\$\cdots+x^{2l-2}t^{l-1}$\\ \\ 
$x^l+x^{l+1}t+$\\$\cdots+x^{2l-2}t^{l-2}$\\ \\$1+xt+\cdots+$\\$2x^{l-1}t^{l-1}+\cdots$\\$+x^{2l-2}t^{2l-2}$  \end{tabular} & 
\begin{tabular}{c} $\sigma_0,\sigma_p,\tau'_p$\\  \\$\sigma_p,\tau'_p;$\\$\tau_p,\mu_p$(if $i$ is even)\\    \\
$\sigma_0,\sigma_p,\tau'_p;$\\$\tau_p,\mu_p$(if $l$ is even)\\  \\
$\sigma_0,\sigma_p, \tau'_p;$\\$\sigma'_0, \tau_p,\mu_p$(if $l$ is odd)\\   \\ $-$  \\ \end{tabular}  \\   
\hline 
\begin{tabular}{c} $-\{\beta \in \Delta_n^+ : \beta \le $\\$\phi_1+\cdots +\phi_i\}$\\$(1 \le i \le l-3)$\\$(l \ge 4)$ \end{tabular} & 
\begin{tabular}{c} $\{\beta \in \Delta_n^+ : \beta \ge $\\$\phi_1+\cdots +\phi_{i+1}\}$\\ \\$\{\beta \in \Delta_n^+ : \beta \ge $\\$\phi_1+\cdots +\phi_j\}$\\$(i+2 \le j \le l-2)$\\ \\
$\{\beta \in \Delta_n^+ : \beta \ge $\\$\xi_i\}(i=1,2)$\\ \\$\{\beta \in \Delta_n^+ : \beta \ge $\\$\xi_1 \textrm{ or } \xi_2\}$\\ \\ 
$\{\beta \in \Delta_n^+ : \beta \ge $\\$\phi_1+\cdots +\phi_i$\\$+2\phi_{i+1}+\cdots +$\\$2\phi_{l-2}+\phi_{l-1}+\phi_l\}$ \end{tabular} & 
\begin{tabular}{c} $x^{2l-2-i}t^i$\\ \\$x^{2l-j-1}t^i+x^{2l-j}t^{i+1}$\\$+\cdots+x^{2l-2-i}t^{j-1}$\\ \\$x^{l-1}t^i+x^lt^{i+1}+$\\$\cdots+x^{2l-2-i}t^{l-1}$\\ \\ 
$x^lt^i+x^{l+1}t^{i+1}$\\$+\cdots+x^{2l-2-i}t^{l-2}$\\ \\ $x^it^i+x^{i+1}t^{i+1}+$\\$\cdots +2x^{l-1}t^{l-1}+\cdots$\\$+x^{2l-i-2}t^{2l-i-2}$  \end{tabular} & 
\begin{tabular}{c} $\sigma_p,\tau'_p$\\  \\$\sigma_p,\tau'_p;$\\$\tau_p,\mu_p$(if $j-i$ is even)\\    \\
$\sigma_p,\tau'_p;$\\$\tau_p,\mu_p$(if $l-i$ is even)\\  \\$\sigma_p, \tau'_p;$\\$\tau_p,\mu_p$(if $l-i$ is odd)\\   \\   
$\sigma_p(p < \frac{i+2}{2} \textrm{ or } p > \frac{2l-i}{2}),$\\$\tau'_p(p < \frac{i+1}{2} \textrm{ or } p > \frac{2l-i-1}{2})$     \end{tabular}  \\   
\hline 
\begin{tabular}{c} $-\{\beta \in \Delta_n^+ : \beta \le $\\$\phi_1+\cdots +\phi_{l-2}\}$\\$(l \ge 3)$ \end{tabular} & 
\begin{tabular}{c} $\{\beta \in \Delta_n^+ : \beta \ge $\\$\xi_1 \textrm{ or }\xi_2\}$\\ \\$\{\beta \in \Delta_n^+ : \beta \ge $\\$\xi_i\}(i =1,2)$\\ \\
$\{\beta \in \Delta_n^+ : \beta \ge $\\$\phi_1+\cdots+\phi_l\}$ \end{tabular} & 
\begin{tabular}{c} $x^lt^{l-2}$\\ \\$x^{l-1}t^{l-2}+x^lt^{l-1}$\\ \\$x^{l-2}t^{l-2}+$\\$2x^{l-1}t^{l-1}+x^lt^l$  \end{tabular} &  
\begin{tabular}{c} $\sigma_0,\sigma_p,\tau'_p$\\  \\$\sigma_0,\sigma_p,\tau'_p,\sigma'_0,\tau_p,\mu_p$\\    \\
$\sigma_p (p \neq \frac{l}{2}, \frac{l+1}{2}, \frac{l+2}{2}, l-1, l),$\\$\tau'_p (p \neq \frac{l-1}{2}, \frac{l}{2}, \frac{l+1}{2})$    \end{tabular}  \\   
\hline
\begin{tabular}{c} $-\{\beta \in \Delta_n^+ : \beta \le $\\$\xi_i\}(i=1,2)(l \ge 3)$ \end{tabular} & 
\begin{tabular}{c} $\{\beta \in \Delta_n^+ : \beta \ge $\\$\xi_j\}(j \in \{1,2\}\setminus \{i\})$\\ \\$\{\beta \in \Delta_n^+ : \beta \ge $\\$\phi_1+\cdots+\phi_l\}$ \end{tabular} & 
\begin{tabular}{c} $x^{l-1}t^{l-1}$\\ \\$x^{l-2}t^{l-1}+x^{l-1}t^l$ \end{tabular} & 
\begin{tabular}{c} $\sigma_p (p \neq \frac{l+1}{2}, l-1, l),$\\$\tau'_p (p \neq \frac{l}{2})$\\   \\$\sigma_0,\sigma_p,\tau'_p,\sigma'_0,\tau_p,\mu_p$    \end{tabular}  \\  
\hline 
\end{tabular}
\end{table}

\begin{table} 
\begin{tabular}{||c|c|c|c||}
\hline
,, & \begin{tabular}{c} $\{\beta \in \Delta_n^+ : \beta \ge $\\$ \phi_1+\cdots +\phi_{j-1}+$\\$2\phi_j+\cdots +2\phi_{l-2}$\\$+\phi_{l-1}+\phi_l\}(2\le $\\$ j\le l-2, l\ge 4)$\\ \\ empty \end{tabular} & 
\begin{tabular}{c} $x^{j-1}t^{l-1}+x^jt^l$\\$+\cdots+x^{l-1}t^{2l-j-1}$\\ \\$t^{l-1}+xt^l+$\\$\cdots+x^{l-1}t^{2l-2}$ \end{tabular} & 
\begin{tabular}{c} $\sigma_p,\tau'_p;$\\$\tau_p,\mu_p$(if $l-j$ is odd)\\    \\$\sigma_0,\sigma_p,\tau'_p;$\\$\tau_p,\mu_p$(if $l$ is even)   \end{tabular}  \\      
\hline
\begin{tabular}{c} $-\{\beta \in \Delta_n^+ : \beta \le $\\$\xi_1,\textrm{or }\xi_2\}(l \ge 3)$ \end{tabular} & 
\begin{tabular}{c} $\{\beta \in \Delta_n^+ : \beta \ge $\\$\phi_1+\cdots + \phi_l\}$\\ \\
$\{\beta \in \Delta_n^+ : \beta \ge $\\$\phi_1+\cdots +\phi_{j-1}+$\\$2\phi_j+\cdots +2\phi_{l-2}$\\$+\phi_{l-1}+\phi_l\}(2\le $\\$ j\le l-2, l\ge 4)$\\ \\empty \end{tabular} & 
\begin{tabular}{c} $x^{l-2}t^l$\\ \\$x^{j-1}t^l+x^jt^{l+1}$\\$+\cdots+x^{l-2}t^{2l-j-1}$\\ \\$t^l+xt^{l+1}+$\\$\cdots+x^{l-2}t^{2l-2}$   \end{tabular} & 
\begin{tabular}{c} $\sigma_0,\sigma_p,\tau'_p$\\    \\$\sigma_p,\tau'_p;$\\$\tau_p,\mu_p$(if $l-j$ is even)\\   \\$\sigma_0,\sigma_p,\tau'_p;$\\$\sigma'_0,\tau_p,\mu_p$(if $l$ is odd)   \end{tabular}  \\      
\hline 
\begin{tabular}{c} $-\{\beta \in \Delta_n^+ : \beta \le $\\$\phi_1+\cdots+\phi_l\}$\\$(l \ge 4)$ \end{tabular} & 
\begin{tabular}{c} $\{\beta \in \Delta_n^+ : \beta \ge $\\$\phi_1+\cdots+\phi_{l-3}+$\\$2\phi_{l-2}+\phi_{l-1}+\phi_l\}$\\ \\
$\{\beta \in \Delta_n^+ : \beta \ge $\\$\phi_1+\cdots +\phi_{j-1}+$\\$2\phi_j+\cdots+2\phi_{l-2}$\\$+\phi_{l-1}+\phi_l\}(2 $\\$\le j\le l-3)$\\ \\empty \end{tabular} & 
\begin{tabular}{c} $x^{l-3}t^{l+1}$\\ \\$x^{j-1}t^{l+1}+x^jt^{l+2}$\\$+\cdots+x^{l-3}t^{2l-j-1}$\\ \\$t^{l+1}+xt^{l+2}$\\$+\cdots+x^{l-3}t^{2l-2}$   \end{tabular} & 
\begin{tabular}{c} $\sigma_p,\tau'_p$\\    \\$\sigma_p,\tau'_p;$\\$\tau_p,\mu_p$(if $l-j$ is odd)\\   \\$\sigma_p,\tau'_p;$\\$\tau_p,\mu_p$(if $l$ is even)   \end{tabular}  \\      
\hline 
\begin{tabular}{c} $-\{\beta \in \Delta_n^+ : \beta \le $\\$\phi_1+\cdots +\phi_{i-1}+$\\$2\phi_i+\cdots+2\phi_{l-2}$\\$+\phi_{l-1}+\phi_l\}$\\$(3\le i \le l-2)$\\$(l \ge 4)$ \end{tabular} & 
\begin{tabular}{c} $\{\beta \in \Delta_n^+ : \beta \ge $\\$ \phi_1+\cdots +\phi_{i-2}+$\\$2\phi_{i-1}+\cdots+2\phi_{l-2}$\\$+\phi_{l-1}+\phi_l \}$\\ \\
$\{\beta \in \Delta_n^+ : \beta \ge $\\$\phi_1+\cdots +\phi_{j-1}$\\$+2\phi_j+\cdots+2\phi_l\}$\\$(2 \le j \le i-2)$\\ \\ empty \end{tabular} & 
\begin{tabular}{c} $x^{i-2}t^{2l-i}$\\ \\$x^{j-1}t^{2l-i}+x^jt^{2l-i+1}$\\$+\cdots +x^{i-2}t^{2l-j-1}$\\ \\$t^{2l-i}+xt^{2l-i+1}$\\$+\cdots+x^{i-2}t^{2l-2}$   \end{tabular} & 
\begin{tabular}{c} $\sigma_p,\tau'_p$\\    \\$\sigma_p,\tau'_p;$\\$\tau_p,\mu_p$(if $i-j$ is even)\\   \\$\sigma_p,\tau'_p;$\\$\tau_p,\mu_p$(if $i$ is odd)   \end{tabular}  \\      
\hline
\begin{tabular}{c} $-\{\beta \in \Delta_n^+ : \beta \le $\\$\phi_1+2\phi_2+\cdots+$\\$2\phi_{l-2}+\phi_{l-1}+\phi_l\}$\\$(l \ge 3)$ \end{tabular} & 
empty & $t^{2l-2}$ & $\sigma_0,\sigma_p, \tau'_p$  \\
\hline 
empty ($l = 2$)  &  
\begin{tabular}{c} empty \\$\{\phi_i\}(i = 1,2)$ \\$\{\phi_1,\phi_2\}$ \end{tabular}  & 
\begin{tabular}{c} $1+2xt+x^2t^2$ \\$x+x^2t$ \\$x^2$ \end{tabular}  &  
\begin{tabular}{c} $-$ \\$\sigma_1, \tau'_1, \tau_1, \mu_1$ \\$\sigma_1, \eta_1, \eta_2$ \end{tabular}  \\  
\hline
$\{-\phi_i\}(i=1,2; l = 2)$  & 
\begin{tabular}{c} empty \\$\{\phi_j\}(j = 1,2, j \neq i)$  \end{tabular}  & 
\begin{tabular}{c} $t+xt^2$ \\$xt$ \end{tabular}  &  
\begin{tabular}{c} $\sigma_1, \tau'_1, \tau_1, \mu_1$ \\$\eta_1, \eta_2$ \end{tabular}  \\  
\hline
$\{-\phi_1,-\phi_2\}(l = 2)$  & empty   &   $t^2$    &    $\sigma_1, \eta_1, \eta_2$    \\   
\hline 
\end{tabular} 
\end{table}

{\bf Note:} In the Table \ref{d-table}, $\xi_1 = \phi_1 + \phi_2 + \cdots + \phi_{l-2} + \phi_{l-1}, \xi_2 = \phi_1 + \phi_2 + \cdots + \phi_{l-2} + \phi_l.$ 

\begin{remark} 
If $m = 2l-1 (l \ge 2),$ from the Table \ref{b-table}, we can see that $\sigma_l$ has no $A_\frak{q}$-component for all $[\frak{q}] \in \mathcal{Q}$ except the trivial representation and 
the one which corresponds to $\Delta(\frak{u} \cap \frak{p}) = \{-\phi_1, \delta\}.$ Since $\sigma_l$ corresponds to a non-zero cohomology class in 
$H^*(\frak{g}, K ; \Gamma_{\sigma_l} \backslash X)$ which is not represented by $G$-invariant differential forms on $X,$ 
we must have $m(\frak{q}_a, \Gamma_{\sigma_l}) \neq 0,$ where $\frak{q}_a$ corresponds to $\Delta(\frak{u} \cap \frak{p}) = \{-\phi_1, \delta\}.$ So $A_{\frak{q}_a}$ is an 
automorphic representation. If $m=2l-2(l \ge 3),$ from the Table \ref{d-table}, we can see that $\tau'_1$  has no $A_\frak{q}$-component for all $[\frak{q}] \in \mathcal{Q}$ 
except the trivial representation and the one which corresponds to $\Delta(\frak{u} \cap \frak{p}) = \{-\phi_1, \delta\}.$ 
Thus $A_{\frak{q}}$ is an automorphic representation, where $\frak{q}$ corresponds to $\Delta(\frak{u} \cap \frak{p}) = \{-\phi_1, \delta\}.$ This result is obtained earlier in \cite{mondal-sankaran2}. 
\end{remark}

\begin{remark} 
Assume that $m \neq 2.$ By a result of Kobayashi and Oda \cite{koboda}(see also \cite{kobpreprint}, \cite{mondal1}), If $\sigma$ is an involution 
commuting with $\theta$ such that $A_\frak{q}$ is discretely decomposable as an $(\frak{g}(\sigma), K(\sigma))$-module, then $\sigma$ has no 
$A_\frak{q}$-component. For $m = 2l-1(l \ge 2),$ if $\sigma = \sigma_p,$ then $A_\frak{q}$ is discretely decomposable as an $(\frak{g}(\sigma), K(\sigma))$-module 
for $\theta$-stable parabolic subalgebras $\frak{q}$ corresponding to $\Delta(\frak{u} \cap \frak{p}) = \Delta_n^+$ and $\Delta_n^-.$ For $m = 2l-2(l \ge 3),$ if 
$\sigma = \sigma_0, \sigma_p, \tau'_p$ then $A_\frak{q}$ is discretely decomposable as an $(\frak{g}(\sigma), K(\sigma))$-module 
for $\theta$-stable parabolic subalgebras $\frak{q}$ corresponding to $\Delta(\frak{u} \cap \frak{p}) = \Delta_n^+$ and $\Delta_n^-.$ See \cite[Table C.3]{kobosh}. Thus if 
$m = 2l-1(l \ge 2),$ then $\sigma_p$ has no $A_\frak{q}$-component for $\frak{q}$ corresponding to $\Delta(\frak{u} \cap \frak{p}) = \Delta_n^+$ and $\Delta_n^-,$ 
and if $m = 2l-2(l \ge 3),$ then $\sigma_0, \sigma_p,\tau'_p$ has no $A_\frak{q}$-component for $\frak{q}$ corresponding to $\Delta(\frak{u} \cap \frak{p}) = \Delta_n^+$ and $\Delta_n^-.$ 
These results are also obtained in Table \ref{b-table}, Table \ref{d-table}. 
\end{remark}  

\section*{acknowledgement}

Both authors acknowledge the financial support from the Department of Science and Technology (DST), Govt. of India under the Scheme 
"Fund for Improvement of S\&T Infrastructure (FIST)" [File No. SR/FST/MS-I/2019/41]. 
Ankita Pal acknowledges the financial support from Council of Scientific and Industrial Research (CSIR) [File No. 08/155(0091)/2021-EMR-I].

\end{document}